\documentclass[10pt, reqno]{amsart}
\usepackage{amsmath}    
\usepackage{amssymb}    
\usepackage{amsthm}     
\usepackage{amsfonts}   
\usepackage[all,cmtip]{xy}
\usepackage{mathtools}  
\usepackage{bbm}		
\usepackage{dsfont}		
\usepackage[mathscr]{euscript}
\usepackage{enumitem}
\usepackage{comment} 	
\usepackage{soul}		
\allowdisplaybreaks		
\usepackage[normalem]{ulem}  

\usepackage{amscd} 		
\usepackage{graphicx}   
\usepackage{tikz-cd}    

\usepackage{geometry}   
\usepackage{setspace}   
\usepackage{enumerate}  
\usepackage[left, pagewise]{lineno}

\usepackage{hyperref}
\hypersetup{
	colorlinks=true,
	linkcolor=blue,
	citecolor=red,
	urlcolor=blue
}

\theoremstyle{definition}
\newtheorem{definition}{Definition}[section]
\newtheorem{theorem}[definition]{Theorem}
\newtheorem{prop}[definition]{Proposition}
\newtheorem{lemma}[definition]{Lemma}
\newtheorem{corollary}[definition]{Corollary}
\newtheorem{remark}[definition]{Remark}

\newtheorem{example}[definition]{Example}

\numberwithin{equation}{section}

\newcommand{\K}{\mathbb K}
\newcommand{\p}{\mathcal{P}}
\newcommand{\q}{\mathcal{Q}}
\newcommand{\g}{\mathcal{G}}
\newcommand{\fa}{\mathfrak{a}}
\newcommand{\Sm}{\mathbb{S}}
\newcommand{\T}{\mathbb{T}}
\newcommand{\I}{\mathrm{I}}
\newcommand{\Op}{\mathcal{O}\mathrm{p}}

\title{Universal Coacting Hopf Algebra of a finite-dimensional Algebra over an Operad}

\author{Saikat Goswami}
\address{Saikat Goswami}
\address{TCG CREST, Institute for Advancing Intelligence, Kolkata-700091, India.}
\address{Ramakrishna Mission Vivekananda Educational and Research Institute, Howrah-711202, India.}
\email{saikatgoswami.math@gmail.com}

\author{Satyendra Kumar Mishra}
\address{Satyendra Kumar Mishra}
\address{Indian Institute of Technology (BHU), Department of Mathematical Sciences,  Varanasi-221005, India.}
\email{satyamsr10@gmail.com}

\author{Suman Pattanayak}
\address{Suman Pattanayak}
\address{TCG CREST, Institute for Advancing Intelligence, Kolkata-700091, India.}
\address{Academy of Scientific and Innovative Research (AcSIR), Ghaziabad-201002, India.}
\email{sumanp.math@gmail.com}


\subjclass[2020]{16S37, 16T10, 17A36, 17A45, 18D20,  18M70.\\
Email: saikatgoswami.math@gmail.com (Goswami), satyamsr10@gmail.com (Mishra), sumanp.math@gmail.com (Pattanayak)}

\begin{document}

\maketitle

\vspace{-0.8cm}
\begin{abstract}
	A. L. Agore and G. Militaru constructed a new invariant (a ``universal coacting Hopf algebra") for some finite-dimensional binary quadratic algebras such as Lie/Leibniz algebras, associative algebras, and Poisson algebras with prominent applications. In our recent work, we extended their construction from the binary case to Lie-Yamaguti algebras (an algebra with a binary and a ternary bracket). In this paper, we give a construction of universal coacting bi/Hopf algebra for any finite-dimensional algebra over a symmetric operad $\mathcal{P}$. Precisely, we construct a universal algebra $\mathcal{C}(\mathfrak{a})$ for a finite-dimensional $\mathcal{P}$-algebra $\mathfrak{a}$. Furthermore, we show that the category of finite dimensional $\mathcal{P}$-algebras is enriched over the dual category of commutative algebras. This enrichment gives a unique bialgebra structure on the universal algebra $\mathcal{C}(\mathfrak{a})$, making it a universal coacting bialgebra of the $\mathcal{P}$-algebra $\mathfrak{a}$. Subsequently, we obtain a universal coacting Hopf algebra of the $\mathcal{P}$-algebra $\mathfrak{a}$. We also show that universal coacting Hopf algebra constructed here coincides with the existing cases of Lie/Leibniz, Poisson, and associative algebras. Furthermore, our operadic approach helps us construct a universal coacting algebra for algebras over a graded symmetric operad (graded algebras with finite-dimensional homogeneous components). This allows us to discuss the universal constructions for $k$-ary quadratic algebras and graded algebras like graded Leibniz, graded Poisson algebras, Gerstenhaber algebras, BV algebras, etc. In the end, we characterize $\mathcal{P}$-algebra automorphisms in terms of the invertible group-like elements of the finite dual bialgebra $\mathcal{C}(\mathfrak{a})$. We also give a characterization of the abelian group gradings of finite dimensional $\mathcal{P}$-algebras.  
\end{abstract}

\vspace{0.3cm}
\tableofcontents{\setcounter{tocdepth}{2}}


\section{\large Introduction}
	Operads are structures that formalize the concept of operations with multiple inputs. They provide a framework to study and encode various types of algebras. The term ``operad" was first introduced by J. P. May \cite{May2}, though its origins trace back to Lazard’s work \cite{Lazard}. Operads have played a key role in the study of loop spaces, with notable contributions from topologists such as M. Boardman and R. Vogt \cite{Boardman-Vogt}, J. P. May  \cite{May1}, and J. Stasheff \cite{Stasheff}. Beyond topology, operads have found applications in differential geometry, quantum field theory, category theory, string topology, and combinatorics \cite{May2,Markl,Vallette}.
	There are multiple equivalent ways to define an operad, but the monoidal definition is particularly straightforward to state. Let $\mathrm{Vect}$ denote the category of vector spaces over a field $\K$ and $\mathrm{Func_{Vect}}$ the category of endofunctors on $\mathrm{Vect}$. An algebraic operad $\mathcal{P}$ is a monoid $(\mathcal{P}, \gamma, \eta)$ in the category $\mathrm{Func_{Vect}}$. The other equivalent descriptions of operads, can be found in \cite[Chapter 5]{Loday-Vallette}. In this paper, we focus on symmetric (graded) operads, which we will explicitly define in the next section. Representations of an operad $\mathcal{P}$ are called algebras over $\mathcal{P}$ or $\mathcal{P}$-algebras. Classical algebras such as associative, Lie, and commutative algebras can be defined as algebras over operads $\mathcal{A}ss,~\mathcal{C}om,$ and $\mathcal{L}ie$, respectively. Other operads, such as those encoding pre-Lie, Leibniz, Zinbiel, perm, and Poisson algebras can be found in \cite[Chapter 13]{Loday-Vallette}. The operadic approach offers new insights, even for classical algebras, by treating different types of algebras within a unified framework. 

	\medskip
	Given two unital associative algebras $A$, $B$ over a field $\K$, a coalgebra measuring from $A$ to $B$ is a pair $(C,\phi)$, where $C$ is a coalgebra over $\K$ and $\phi:C\rightarrow \mathrm{Hom}(A,B)$ is a $\K$-linear map such that 
	\begin{equation*}
		\phi(c)(ab) ~=~ \sum \phi(c_{(1)})(a)\cdot \phi(c_{(2)})(b) 
		\quad \text{and} \quad 
		\phi(c)(1_A) ~=~ \varepsilon(c)1_B 
		\quad\mbox{for all } c\in C,
	\end{equation*}
	where $a,b\in A,~\Delta(c)=\sum c_{(1)}\otimes c_{(2)}$.	
	A coalgebra measuring $(C,\phi)$ can be seen as a generalization of algebra homomorphisms from $A$ to $B$, since for any group-like element $c\in C$, $\phi(c): A\rightarrow B$ is an algebra homomorphism. M. E. Sweedler \cite{Sweedler} first introduced a universal measuring coalgebra for associative algebras. This universal construction leads to an enrichment of the category of associative algebras over the monoidal category of cocommutative coalgebras and finally, associates a natural universal acting Hopf algebra to an algebra (see \cite[Chapter 7]{Sweedler}). Recent articles \cite{Banerjee, Banerjee24,Hyland} studied measurings between monoids, algebras over operads, Hopf algebroids and discussed the associated enrichments and other applications. The dual notion of an existing notion is equally important in Hopf algebra theory \cite{Sweedler}. Yu. I. Manin \cite{Manin} and D. Tambara \cite{Tambara} dualised Sweedler's universal construction and called it `universal coacting bi/Hopf algebra'. These objects have many applications in non-commutative algebraic geometry \cite{Raed}. The universal coacting algebra exists only for finite-dimensional algebras. 
In last five years, there are several successful attempts to explicitly construct a universal coacting bi/Hopf algebra for associative algebras \cite{Mil}, Lie and Leibniz algebras \cite{Ag-Mil}, Poisson algebras\cite{Ag-Mil2}, and Lie-Yamaguti algebras \cite{Lie-Y}. These constructions provided tools to solve hard and complex problems in the literature such as:  1. characterizing the group of automorphisms of an algebra, and 2. characterizing the abelian group gradings on an algebra. We also refer the reader to a construction of universal (co)acting objects with generalized duality results in the case of $\Omega$-algebras \cite{Agore-omega}.

	\medskip
	In this paper, we consider finite-dimensional algebras over a symmetric operad and construct a universal coacting algebra unifying the different constructions \cite{Ag-Mil, Ag-Mil2, Lie-Y, Mil} in literature. We show that the category of finite-dimensional $\mathcal{P}$-algebras is enriched over the dual category of commutative algebras. Consequently, we obtain a universal coacting bi/Hopf algebra for any finite-dimensional algebra over $\mathcal{P}$. We also consider the notion of graded operads and algebras over a graded operad to construct a universal algebra for graded algebras. We explicitly show that our universal coacting bi/Hopf algebra coincides with the universal coacting bi/Hopf algebra of Lie/Leibniz algebras \cite{Ag-Mil}, Poisson algebras \cite{Ag-Mil2}, and associative algebras \cite{Mil} when the underlying operad $\p$ is $\mathcal{L}ie/\mathcal{L}eib$, $\mathcal{P}ois$, and $\mathcal{A}ss$ respectively. In a sequel, we also discuss the universal constructions in the special cases of some (graded) quadratic algebras such as Zinbiel algebras, pre-Lie algebras, perm algebras, $k$-ary quadratic algebras, graded Lie algebras, graded Poisson algebras, Gerstenhaber algebras, Batalin-Vilkovisky algebras, and Loday type algebras. In the end, we give a complete description of automorphism group of a finite dimensional $\p$-algebra and characterize the abelian group gradings on any finite dimensional $\p$-algebra. 
	
	\medskip\noindent
	{\bf Organization of the Paper:}~\\
	
	\vspace{-0.4cm}
	In Section \ref{sec-2}, we recall all the necessary definitions and basic notions related to the theory of operad. In Section \ref{sec-3}, we first construct a universal algebra $\mathcal{C}(\mathfrak{a})$ for a finite dimensional $\mathcal{P}$-algebra $\mathfrak{a}$. In the process, we prove Theorem \ref{thm_left_adjoint}, which generalizes the result obtained for the case of associative algebras \cite[Theorem 1.1]{Mil}, Lie/Leibniz algebras \cite[Theorem 2.1]{Ag-Mil}, and Poisson algebras \cite[Theorem 2.2]{Ag-Mil2} to the case of any finite dimensional algebra over a symmetric operad. In the end of the section, we give a construction for algebras (with finite-dimensional homogeneous components) over graded operads. Section \ref{sec-4} is devoted to a discussion on the enrichment of the category of finite-dimensional $\mathcal{P}$-algebras over the category dual to the category $\mathrm{ComAlg}$ of commutative algebras. As a consequence of this enrichment, we obtain a natural universal coacting bi/Hopf algebra for any finite dimensional $\mathcal{P}$-algebra. In Section \ref{sec-5}, we show that the universal polynomials for finite-dimensional algebras over binary quadratic operad are generated by universal  polynomials of degree $2$. It follows that universal coacting bi/Hopf algebra obtained for a finite dimensional $\p$-algebra coincides with the universal coacting bi/Hopf algebra obtained for the case of Lie/Leibniz algebras \cite{Ag-Mil}, Poisson algebras \cite{Ag-Mil2}, and associative algebras \cite{Mil} if we take the operad $\p$ to be $Lie$, $Pois$, and $Ass$ respectively. We also describe universal polynomials and algebras for other (graded) quadratic algebras. In Section \ref{sec-6}, we give a complete description of automorphism group of a finite dimensional $\mathcal{P}$-algebra $\mathfrak{a}$ in Theorem \ref{App-bijection} and characterize the $G$-gradings on $\mathfrak{a}$ for an abelian group $G$ in Theorem \ref{App-gradings}.
	
	\medskip\noindent
	{\bf Notations and Conventions:}~\\
	
	\vspace{-0.4cm}
	In this paper, we fix a field $\K$ of characteristic $0$. We take all vector spaces, linear maps and tensor products over $\K$, unless stated otherwise. By a commutative algebra, we always mean a commutative, associative algebra. We denote by $\mathrm{Alg}$, $\mathrm{ComAlg}$, $\mathrm{ComBiAlg}$, and $\mathrm{ComHopf}$, the categories of associative algebras, commutative algebras, commutative bialgebras, and commutative Hopf algebras, respectively. For any given operad $\p$, we denote by $\p$-Alg the category of all $\p$-algebras. 
	Throughout this article, $\Sm_n$ denotes the usual group of permutations on $n$ elements. For foundational material on Hopf algebra theory and category theory, we refer the readers to \cite{Sweedler} and \cite{MacLane}, respectively.


\medskip
\section{\large Preliminaries}\label{sec-2}
	
	In the section, we recall the all the necessary definitions and notions related to operads and graded operads, that we use throughout this article.  
	
	\begin{definition}
		\label{S-module}
		An {\bf $\Sm$-module}  over $\K$ is a family 
		$\p = (\p(0),\p(1),\ldots,\p(n),\ldots)$	
		of right $\K[\Sm_n]$-modules $\p(n)$. The \textbf{tensor product} of two $\Sm$-modules $\p$ and $\q$ is given by 
		\begin{equation*}
			\p \otimes \q(n) := \bigoplus_{i+j=n} \p(i) \otimes \q(j) \otimes \K[Sh(i,j)]
		\end{equation*}
		where $Sh(i,j)$ denotes the set of all $(i,j)$-shuffles of $\Sm_n$. The category of $\Sm$-modules, denoted by $\Sm$-mod, forms a monoidal category given by the following \textbf{composition product} 
		\begin{equation*}
			\p \circ \q (n) := \bigoplus_{k \ge 0} \p(k) \otimes_{\Sm_k} \q^{\otimes k}
		\end{equation*}
		The notation $\q^{\otimes k}$ denotes the tensor product of $k$ copies of the $\Sm$-module $\q$. The unit for the composition is given by $\I=(0,\K,0,\cdots)$. 
		See \cite[Section 5.1, pg 123]{Loday-Vallette} for further details regarding $\Sm$-modules and their products.   
		
	\end{definition}
	
	\begin{definition}[\bf Monoidal Definition]\label{monoidal_def}
		An {\bf operad} $\mathcal{P}$ is a monoid in the monoidal category of $\Sm$-modules. More precisely, it is an $\Sm$-module $\p= \{\p(n)\}_{n \ge 0}$ endowed with a morphisms of $\Sm$-module 
		\begin{equation*}
			\gamma: \p \circ \p \to \p 
		 	\quad \text{and} \quad 
		 	\eta: \I \to \p
		\end{equation*} 
		called composition map and unit map respectively, which makes $\p$ into a monoid. 
		For a more detailed description of the monoidal definition we refer the readers to \cite[Section 5.2.1, pg 131]{Loday-Vallette}. 
	\end{definition}
	
	There are several other equivalent definitions of an operad, including the classical, combinatorial, and partial definitions. We describe the partial definition below, while the classical and combinatorial definitions can be found in \cite[pg 142 and pg 161]{Loday-Vallette}, respectively.
	 
	\begin{definition}[\bf Partial Definition]\label{partial_def}
		An {\bf operad} $\p$ is a sequence $\{\p(n)\}_{n \in \mathbb{N}}$ of vector spaces, together with partial composition maps
		\begin{equation*}
			\circ_i : \p(m) \times \p(n) \rightarrow \p(m-1+n)
		\end{equation*}
		for $1 \leq i \leq m$ and $n \geq 0$, along with a unit element $\mathbbm{1} \in \p(1)$ satisfying the following conditions
		\begin{itemize}
			\item (Associativity) for all $\mu \in \p(m)$, $\nu \in \p(n)$, and $\eta \in \p(r)$
			
			\begin{equation}
				(\mu \circ_i \nu) \circ_j \eta = 
				\begin{cases}
					\begin{array}{ll}
					(\mu \circ_j \eta) \circ_{i+j-1} \nu \quad &1 \le j \le i-1, \\ 
					\mu \circ_i (\nu \circ_{j-i+1} \eta) \quad &i \le j \le m+i-1, \\
					(\mu \circ_{j-m+1} \eta) \circ_i \nu \quad &m+i \le j \le m-1+n.
					\end{array}
 				\end{cases}
			\end{equation}
			\smallskip
			\item (Equivariance) Each $\p(n)$ has a right action of $\Sm_n$, $\cdot:\p(n) \times \Sm_n \to \p(n)$, denoted by $(\mu,\sigma) \mapsto \mu \cdot \sigma$, such that for all $\mu \in \p(n)$, $\nu \in \p(m)$ and $\sigma \in \Sm_n$, $\tau \in \Sm_m$
			\begin{equation}
				(\mu\cdot \sigma) \circ_i (\nu \cdot \tau)= (\mu \circ_i\nu)\cdot (\sigma \circ_i \tau)
			\end{equation}
			where $\sigma \circ_i \tau$ is the block permutation, with the $i^{th}$ entry of $\sigma$ being replaced by the entire block $\tau$.  
			\smallskip
			\item (Unitality) There exists a unique element $\mathbbm{1} \in \p(1)$ such that $\mu \in \p(n)$ $(\text{with}~n \ge 1)$, and $1 \le i \le n$,
			\begin{equation}
				\mu \circ_i \mathbbm{1} ~=~ \mu ~=~ \mathbbm{1} \circ_1 \mu.
			\end{equation}
		\end{itemize}
		A pictorial description of the above mentioned conditions can be found in \cite[Section 5.3.4, pg 146]{Loday-Vallette}.
	\end{definition}
	
	Let us consider a fundamental example to illustrate the basic idea of operads: the endomorphism operad, associated with a given vector space serves as a key example to understand the operadic composition. 
	
	\begin{example}[\bf Endomorphism operad]\label{ExEnd}
		For any vector space $\fa$, the endomorphism operad $\mathrm{End}_\fa$ is given by
		\begin{equation}
			\mathrm{End}_\fa(n) := \mathrm{Hom}(\fa^{\otimes n},\fa) \quad \text{for all}~ n \ge 0,
		\end{equation}
		where $\fa^{\otimes n}$ denotes the $n$-fold tensor product of $\fa$, and by convention $\fa^{\otimes 0} = \K$. The partial composition maps given by $\circ_i: \mathrm{Hom}(\fa^{\otimes n}, \fa) \times \mathrm{Hom}(\fa^{\otimes m},\fa) \to \mathrm{Hom}(\fa^{\otimes m-1+n},\fa)$ are defined as follows: 	
		\begin{equation}\label{def_partial_comp_end_operad}
			f \circ_i g (a_1,\ldots, a_{m-1+n}) := f(a_1, \ldots,a_{i-1}, g(a_i,\ldots,a_{i+m-1}), a_{i+m},\ldots, a_{m-1+n})
		\end{equation}
		for all $a_1,\ldots,a_{m-1+n} \in \fa$. The right action $\cdot: \mathrm{Hom}(\fa^{\otimes n}, \fa) \times \Sm_n \to \mathrm{Hom}(\fa^{\otimes n}, \fa)$ is given by 
		\begin{equation}\label{def_end_opd_sym_action}
			(f \cdot \sigma) (a_1,\ldots,a_{n}) := f(a_{\sigma(1)},\ldots,a_{\sigma(n)})  
		\end{equation}
		and the unit element $\mathbbm{1}$ is the identity map $\mathrm{id}_\fa \in \mathrm{End}_\fa(1)$. 
	\end{example}

	\begin{example}\label{com}
		Consider a constant sequence of the field $\K$ as follows:   
		\begin{equation*}
			Com(n) := \mathbb{K} \quad \text{for all } n \ge 0.
		\end{equation*}
		together with partial compositions, $\mathbb{S}_n$ right action, and unit element described below. 
		
		\medskip\noindent 
		{\it Partial composition:} For any $a \in Com(m) = \mathbb{K}$ and $b \in Com(n)= \mathbb{K}$ define 
		\begin{equation*} \label{Com_circle_product}
			a \circ_{i} b := ab \quad \text{for all}~ 1 \le i \le m.
		\end{equation*}
		(the usual multiplication of $\mathbb{K}$). 
		The associativity of the partial composition maps follows from the associativity and commutativity of the multiplication in $\mathbb{K}$.
		
		\medskip \noindent
		{\it $\mathbb{S}_n$ action on $Com(n)$:}
		The right action of $\mathbb{S}_n$ on $Com(n)$ is assumed to be trivial for each $n \ge 1$, i.e., for any $\sigma \in \mathbb{S}_n$ and $f \in Com(n)$; $f \cdot \sigma = f$.
		
		\medskip \noindent
		{\it Unit element:}
		For any $n \ge 1$ let $1_{\mathbb{K}}^{(n)} \in Com(n) = \mathbb{K}$ denote the multiplicative identity of $Com(n) = \mathbb{K}$. 
		With this notation $1_{\mathbb{K}}^{(1)}$ is the unit element in $\mathcal{C}om$.	
		Thus, making $\mathcal{C}om$ an operad.
	\end{example}

	Next, we give an example of an operad $\p$ in which all the spaces $\p(n)$ are equal to a fixed vector space $V$.
	
	\begin{example}\label{exm_V}
		Let $V$ be a vector space with a basis being given by the set $\{b_1,\ldots,b_p\} = B$. Now, we consider the following sequence of vector spaces: 
		\begin{equation}
			\p(n) := \mathrm{Hom}(B,\K) \cong V \quad \text{for any } n \ge 0,
		\end{equation} 
		where $\mathrm{Hom}(B,\K)$ denotes the vector space consisting of all set maps from $B$ to $\K$. We define the partial composition, $\mathbb{S}_n$ right action on $\p(n)$, and the unit element as follows:
		
		\medskip\noindent 
		{\it Partial composition:} For any $f \in \p(m) = \mathrm{Hom}(B,\K)$ and $g \in \p(n)= \mathrm{Hom}(B,\K)$ we define 
		\begin{equation*} \label{Com_circle_product}
			f \circ_{i} g := f \cdot g \quad \text{for all}~ 1 \le i \le m,
		\end{equation*}
		where $f \cdot g$ is defined pointwise using the multiplication defined on $\mathbb{K}$. 
		We note that the associativity of the partial composition maps follows from the associativity and commutativity of the multiplication in $\mathbb{K}$.
		
		\medskip \noindent
		{\it $\mathbb{S}_n$ action on $\p(n)$:}
		The right action of $\mathbb{S}_n$ on $\p(n)$ is assumed to be trivial for each $n \ge 1$, i.e., for any $\sigma \in \mathbb{S}_n$ and $f \in \p(n)$; $f \cdot \sigma = f$.
		
		\medskip \noindent
		{\it Unit element:}
		Considering the constant map $\mathbbm{1}: B \to \K$, defined as $\mathbbm{1}(b) := 1_\K$ for all $b \in B$, we note that $\mathbbm{1} \in \p(1)$ acts as the the unit element in $\p$.	
		Thus, turning $\p$ into an operad.
	\end{example}

	\begin{example}\label{ass}
		Consider a sequence of vector spaces over $\K$ as follows: 
		\begin{equation} \label{def_Ass}
			Ass(n) := 
			\begin{cases}
				0 \quad \qquad n = 0, \\
				\mathbb{K} ~\quad \quad n=1, \\
				\mathbb{K}[\mathbb{S}_n] \quad n \ge 2.
			\end{cases}
		\end{equation}
		along with  partial compositions, $\mathbb{S}_n$ right action, and unit element described below.
		
		\medskip\noindent
		{\it Partial composition:} 
		Let $\sigma \in \mathbb{K}[\mathbb{S}_n]$ and $\tau \in \mathbb{K}[\mathbb{S}_m]$ be two basis elements. Define 
		\begin{equation*}
			\sigma \circ_i \tau := \text{The block 	permutation, where the} ~i^{th}~ \text{entry of} ~\sigma~ \text{is replaced by the entire block} ~\tau,	
		\end{equation*} 
		then extend it linearly to whole of $\mathbb{K}[\mathbb{S}_n]$ and $\mathbb{K}[\mathbb{S}_m]$.
		
		\medskip\noindent 
		{\it $\mathbb{S}_n$ action on $Ass(n)$:} 
		The right action of $\mathbb{S}_n$ on $Ass(n) = \mathbb{K}[\mathbb{S}_n]$ is the regular action that is defined as follows. For any basis element $\sigma \in \mathbb{K}[\mathbb{S}_n]$ and $\tau \in \mathbb{S}_n$, defining $\sigma \tau := \sigma \cdot \tau$ (group multiplication of $\mathbb{S}_n$), and then extending it linearly to whole $\mathbb{K}[\mathbb{S}_n]$ gives us the $\mathbb{S}_n$ right action on $\mathbb{K}[\mathbb{S}_n]$. 
		
		\medskip\noindent
		{\it Unit element:} 
		$1 \in \mathbb{K} = Ass(1)$ is the unit element of $Ass$. Thus, making $\mathcal{A}ss$ an operad.  
	\end{example}

	For more examples of operads, we refer the reader to \cite{Apurba-Da}, where a comprehensive list of examples along with detailed descriptions can be found.
	
	\begin{definition}
		Let $\mathcal{P} = \{\mathcal{P}(n) : n \ge 0\}$ and $\mathcal{Q} = \{\mathcal{Q}(n) : n \ge 0\}$ be two operads. A {\bf morphism of operads} $\gamma: \mathcal{P} \to \mathcal{Q}$ consists of a sequence of linear maps $\gamma_n : \mathcal{P}(n) \to \mathcal{Q}(n)$ for all $n \ge 0$ such that 
		\begin{equation} \label{morphism_unital_cond}
			\gamma_1 (\mathbbm{1}_{\mathcal{P}}) = \mathbbm{1}_{\mathcal{Q}}
		\end{equation}
		and for all $n,m \ge 1$ and $\sigma \in \mathbb{S}_n$ the following diagrams commute.
		\begin{equation} \label{morphism_diagram}
			\begin{tikzcd}
				{\mathcal{P}(n) \times \mathcal{P}(m)} & {\mathcal{P}(m-1+n)} &&& {\mathcal{P}(n)} & {\mathcal{P}(n)} \\
				{\mathcal{Q}(n) \otimes \mathcal{Q}(m)} & {\mathcal{Q}(m-1+n)} &&& {\mathcal{Q}(n)} & {\mathcal{Q}(n)}
				\arrow["{\circ_i}", from=1-1, to=1-2]
				\arrow["{\circ_i}", from=2-1, to=2-2]
				\arrow["{\gamma_n \otimes \gamma_m}"', from=1-1, to=2-1]
				\arrow["{\gamma_{m-1+n}}", from=1-2, to=2-2]
				\arrow["\sigma", from=1-5, to=1-6]
				\arrow["\sigma", from=2-5, to=2-6]
				\arrow["{\gamma_n}"', from=1-5, to=2-5]
				\arrow["{\gamma_n}", from=1-6, to=2-6]
			\end{tikzcd}
		\end{equation}
	\end{definition}
	
	\begin{definition}
		A {\bf $\p$-algebra} is a vector space $\fa$ along with a morphism of operads $\gamma: \mathcal{P} \to \mathrm{End}_\fa$. Let $\fa$ and $\mathfrak{b}$ be two $\p$-algebras with the $\p$-algebra structures given by $\gamma: \p \to \mathrm{End}_\fa$ and $\widetilde{\gamma}:\p \to \mathrm{End}_\mathfrak{b}$, respectively. A {\bf $\p$-algebra morphism} from $\fa$ to $\mathfrak{b}$ is a linear map $f: \fa \to \mathfrak{b}$  such that for any $n \in \mathbb{N}$ and $\mu \in \p(n)$ 
		\begin{equation}
			f\big(\gamma_n(\mu)(a_{1},\ldots,a_{n})\big) = \widetilde{\gamma}_n(\mu)(f(a_{1}),\ldots,f(a_{n})).   
		\end{equation}
		Equivalently, the following diagram commutes for any $k \in \mathbb{N}$ and $\mu \in \p(k)$
		\begin{equation*}
			\begin{tikzcd}
				{\mathfrak{a}^{\otimes k}} & {\mathfrak{a}} \\
				{\mathfrak{b}^{\otimes k}} & {\mathfrak{b}}
				\arrow["{\gamma_k}", from=1-1, to=1-2]
				\arrow["{f^{\otimes k}}"', from=1-1, to=2-1]
				\arrow["f", from=1-2, to=2-2]
				\arrow["{\widetilde{\gamma}_k}", from=2-1, to=2-2]
			\end{tikzcd}
		\end{equation*}
		An equivalent definition of $\p$-algebras, using the language of category theory, can be found in \cite[Section 5.3.4, pg 146]{Loday-Vallette}.
	\end{definition}

	\begin{remark}
		Any algebra over the operad $\mathcal{C}om$ (see Example \ref{com}) is a commutative and associative algebra. 
		In contrast, an algebra over the operad described in Example \ref{exm_V} is a vector space endowed with $p$ distinct binary operations, each of which is commutative and associative, but with no imposed compatibility conditions between them. 
		Finally, an algebra over the operad $\mathcal{A}ss$ (see Example \ref{ass}) is an associative algebra.
	\end{remark}
	
	\begin{definition}
		Let $M$ be an $\Sm$-module. The \textbf{free operad} over the $\Sm$-module $M$ is an operad $\T(M)$ equipped with an $\Sm$-module morphism $\eta(M): M \to \T(M)$ which satisfies the following universal property:
		Any $\Sm$-module morphism $f: M \to \p$, where $\p$ is an operad, extends uniquely into an operad morphism $\widetilde{f}: \T(M) \to \p$ such that the following diagram commutes
		\begin{equation*}
			\begin{tikzcd}
				M & {\T(M)} \\
				& \p
				\arrow["{\eta(M)}", from=1-1, to=1-2]
				\arrow["f"', from=1-1, to=2-2]
				\arrow["{\widetilde{f}}", from=1-2, to=2-2]
			\end{tikzcd}
		\end{equation*}
		In other words, the functor $\T: \Sm$-mod $\to \Op$ is left adjoint to the forgetful functor $\Op \to \Sm$-mod, where $\Op$ denotes the category of operads over $\K$. An explicit construction of the free operad $\T(M)$ is described as follows: 
		\begin{equation}
			\T(M)(n) := \bigoplus_{t \in RT(n)} M(t),
		\end{equation}
		where $RT(n)$ denotes the set of all rooted trees with $n$ leaves such that each vertex has atleast one input. And $M(t)$ denotes the tree wise tensor product, for a detailed description of $M(t)$ see \cite[Section 5.6.1, pg 160]{Loday-Vallette}. 
	\end{definition}

	We illustrate this construction by considering a specific example, primarily highlighting the structure of a binary quadratic operad. 	
		
	\begin{example}\label{EXAMPLE}
		Let $E$ be a right $\Sm_2$-module. Consider the binary quadratic $\Sm$-module $\overline{E} := (0,0,E,0,0, \ldots)$.
		We denote the free operad over $\overline{E}$ by $\mathbb{T}(E)$ instead of $\mathbb{T}(\overline{E})$. We now describe a way to construct $T(E)(n)$, for each $n \in \mathbb{N}$. The non-trivial summands of $\T(E)(n)$ can be identified with a collection of subsets $S \subseteq \{1,2,\ldots,n\}$ which satisfies the following properties:  
		\begin{itemize}
			\item Each $S$ in the collection has cardinality greater than $1$, 
			
			\item The collection includes $\{1,2,\ldots,n\}$, and 
			
			\item For subsets $S, S'$ belonging to the collection, either $S \subseteq S'$ or $S' \subseteq S$ or $S \cap S' = \emptyset$. 
		\end{itemize}
		
		One may think of each $S$ in the collection as the set of leaves lying above a given vertex in the tree (and in this situation each vertex of the tree is uniquely determined by the set of leaves above it). So for the set $\{1,2,3\}$ there are four possibilities:
 		\begin{equation*}
 			\lbrace \lbrace 1, 2\rbrace, \lbrace 1, 2, 3\rbrace\rbrace, \quad \lbrace \lbrace 1, 3\rbrace, \lbrace 1, 2, 3\rbrace\rbrace, \quad \lbrace \lbrace 2, 3 \rbrace, \lbrace 1, 2, 3\rbrace\rbrace, \quad \lbrace\lbrace 1, 2, 3\rbrace\rbrace.
 		\end{equation*}
 		The last possibility indexes a trivial summand (because $\overline{E}(3)=0$), so that just leaves the first three, where the treewise tensor products are each isomorphic to $E \otimes E$.
		The first four terms of $\mathbb{T}(E)$ are described below:
		\begin{equation*}
			\T(E)(0) \cong 0, \quad \T(E) (1) \cong \K, \quad \T(E)(2) \cong E, \quad \T(E) (3) \cong 3 (E \otimes E), \quad \T(E) (4) \cong 15 (E \otimes E \otimes E)
		\end{equation*}
		where $k(E \otimes \cdots \otimes E)$ means direct sum of $k$ copies of $E \otimes \cdots \otimes E$. 
	\end{example}

	Having established the foundational concepts and examples related to operads, we now consider the graded context. While many of the definitions and constructions carry over, we present them briefly, focusing on the key changes needed for graded operads.
	
	\begin{definition}[\textnormal{\cite[Section 6.2.1, pg 197]{Loday-Vallette}}]
		A {\bf graded $\Sm$-module} $\g$ over $\K$ is a family 
		$\g = (\g(0),\g(1),\ldots,\g(n),\ldots)$
		of graded vector spaces over $\K$ equipped with a graded right $\K[\Sm_n]$-module structure. Analogous to the category $\Sm$-mod, one can show that the category g$\Sm$-mod of graded $\Sm$-modules over $\K$ also forms a monoidal category.
	\end{definition}

	\begin{definition}
		A {\bf graded operad $\g$} is a monoid in the monoidal category g$\Sm$-mod of graded $\Sm$-modules. For further details we refer the readers to \cite[Section 6.3.1, pg 199]{Loday-Vallette}.  
	\end{definition}

	Building on the partial definition of operads in Definition \ref{partial_def}, we can similarly define a partial definition for graded operads. We now provide an example of a graded operad.
	
	\begin{example}[\bf Graded endomorphism operad]
		Let $\mathfrak{A}=\oplus_{p \ge 0} \mathfrak{A}_p$ be a graded vector space. The graded endomorphism operad $\mathrm{End}_{\mathfrak{A}}$ is defined as follows:
		\begin{equation}
			\mathrm{End}_{\mathfrak{A}}(n) := \mathrm{Hom}_{\mathrm{g}}(\mathfrak{A}^{\otimes n},\mathfrak{A}) \quad \text{for all}~ n \ge 0.
		\end{equation}
		where $\mathfrak{A}^{\otimes n}$ denotes the $n$-fold tensor product of $\mathfrak{A}$ over $\K$, by convention $\mathfrak{A}^{\otimes 0} = \K$, and $\mathrm{Hom}_{\mathrm{g}}(\mathfrak{A}^{\otimes n},\mathfrak{A})$ denotes the graded vector space of all graded linear maps $f:\mathfrak{A}^{\otimes n} \to \mathfrak{A}$ of any degree, that is, for any $a_1 \in \mathfrak{A}_{p_1}, \ldots, a_n \in \mathfrak{A}_{p_n}$, we have $f(a_1,\ldots,a_n) \in \mathfrak{A}_{p_1 + \cdots + p_n+|f|}$. The partial composition maps $\circ_i: \mathrm{Hom}_\mathrm{g}(\mathfrak{A}^{\otimes n}, \mathfrak{A}) \times \mathrm{Hom}_\mathrm{g}(\mathfrak{A}^{\otimes m},\mathfrak{A}) \to \mathrm{Hom}_\mathrm{g}(\mathfrak{A}^{\otimes m-1+n},\mathfrak{A})$ and the right action $\cdot: \mathrm{Hom}_\mathrm{g}(\mathfrak{A}^{\otimes n}, \mathfrak{A}) \times \Sm_n \to \mathrm{Hom}_\mathrm{g}(\mathfrak{A}^{\otimes n}, \mathfrak{A})$ are defined in the same way as defined in Equation \eqref{def_partial_comp_end_operad} and \eqref{def_end_opd_sym_action}, which preserves the degree. 	
		The unit element $\mathbbm{1}$ is the identity map $\mathrm{id}_{\mathfrak{A}} \in \mathrm{End}_\mathfrak{A}(1)$. 
	\end{example}

	\begin{definition}
		Let $\mathcal{G} = \{\mathcal{G}(n) : n \ge 0\}$ and $\mathcal{H} = \{\mathcal{H}(n) : n \ge 0\}$ be two graded operads. A {\bf morphism of graded operads} $\gamma: \mathcal{G} \to \mathcal{H}$ consists of a sequence of degree zero linear maps $\gamma_n : \mathcal{G}(n) \to \mathcal{H}(n)$ for all $n \ge 0$ such that 
		\begin{equation} \label{graded_morphism_unital_cond}
			\gamma_1 (\mathbbm{1}_{\mathcal{G}}) = \mathbbm{1}_{\mathcal{H}}
		\end{equation}
		and for all $n,m \ge 1$ and $\sigma \in \mathbb{S}_n$, Diagram \ref{morphism_diagram} should commute. 
	\end{definition}
	
	\begin{definition}\label{algebra_graded_def}
		A {\bf $\g$-algebra} is a graded vector space $\mathfrak{c}$ equipped with a morphism of graded operads $\gamma: \mathcal{G} \to \mathrm{End}_\mathfrak{c}$.
	\end{definition}

	With these notations and definitions, we proceed to the construction of a universal algebra associated with a given finite-dimensional algebra over an arbitrary operad. 


\bigskip
\section{\large The construction of a Universal algebra for algebras over (Graded) Operads}\label{sec-3}

	In this section, we present a universal construction for $\p$-algebras, where $\p$ is an arbitrary (graded) operad. For any given $\p$-algebra $\mathfrak{a}$, we establish the existence of a functor \(\mathfrak{a} \otimes - : \mathrm{ComAlg} \to \p\)\text{-}\(\mathrm{Alg}\) from the category of commutative algebras to the category of $\p$-algebras (cf. Proposition \ref{funct_com_to_p-alg}). Additionally, in Theorems \ref{thm_left_adjoint} and \ref{thm_graded_left_adjoint}, we prove that this functor admits a left adjoint \(\mathcal{C}(\fa, -): \mathcal{P}\)\text{-}\(\mathrm{Alg} \to \mathrm{ComAlg}\), provided certain conditions are satisfied.
	Our construction not only unifies the results of \cite[Theorem 2.1]{Ag-Mil}, \cite[Theorem 2.2]{Ag-Mil2}, \cite[Theorem 3.1]{Lie-Y}, and \cite[Theorem 1.1]{Mil} (obtained for 
	Leibniz/Lie, Poisson, Lie-Yamaguti, and associative algebras, respectively),  
	but also naturally extends them to their graded counterparts.

\medskip	
\subsection{Algebras over an Operad}

	\begin{prop}\label{funct_com_to_p-alg}
		Let $\p$ be an operad. For any $\p$-algebra $\fa$ we obtain a functor $\fa  \otimes - : \mathrm{ComAlg} \to  \p$-Alg. 
	\end{prop}
	
	\begin{proof}
		Let $\gamma:\p \to \mathrm{End}_\fa$ be the $\p$-algebra structure on $\fa$. Now, for any $C \in \mathrm{ComAlg}$, we show that $\fa\otimes C$ is a $\p$-algebra. We define a map $\overline{\gamma}:\p\rightarrow \mathrm{End}_{\fa\otimes C}$ by 
		\begin{equation}\label{morph}
			\overline{\gamma}_k(\mu)(a_1\otimes c_1,\ldots,a_k\otimes c_k):=\gamma_k(\mu)(a_1,\ldots,a_k)\otimes c_1\cdots c_k,
		\end{equation}
		for all $\mu\in \p(k);~a_1,a_2,\ldots,a_k\in \fa;$ and $c_1,\ldots,c_k\in C$. To show that $\overline{\gamma}$ defines a $\p$ algebra structure on $\fa \otimes C$, we first establish the commutativity of the following diagram
		\begin{equation}\label{prop_diag1_P-Alg_tensor_ComAlg}
			\begin{tikzcd}
				{\p(k) \otimes \p(l)} & {\p(k-1+l)} \\
				{\mathrm{End}_{\mathfrak{a} \otimes C}(k) \otimes \mathrm{End}_{\mathfrak{a} \otimes C}(l)} & {\mathrm{End}_{\mathfrak{a} \otimes C}(k-1+l),}
				\arrow["{\circ_i}", from=1-1, to=1-2]
				\arrow["{\overline{\gamma}_k \otimes \overline{\gamma}_l}"', from=1-1, to=2-1]
				\arrow["{\overline{\gamma}_{k-1+l}}", from=1-2, to=2-2]
				\arrow["{\circ_i}", from=2-1, to=2-2]
			\end{tikzcd}
		\end{equation}
		for any $k,l \in \mathbb{N}$. Let $\mu \in \p(k)$ and $\nu \in \p(l)$. Then, we obtain the following expression.  
		\begin{align*}
			\overline{\gamma}&_{k-1+l}(\mu \circ_i \nu) (a_1 \otimes c_1, \ldots, a_{k-1+l} \otimes c_{k-1+l}) \\
			&\stackrel{(\ref{morph})}{=} 
			\gamma_{k-1+l} (\mu \circ_i \nu) (a_1,\ldots,a_{k-1+l}) \otimes c_1 \cdots c_{k-1+l} \\
			&\stackrel{(\ref{morphism_diagram})}{=}
			\gamma_k(\mu) \circ_i \gamma_l(\nu) (a_1,\ldots,a_{k-1+l}) \otimes c_1 \cdots c_{k-1+l} \\
			&\stackrel{(\ref{def_partial_comp_end_operad})}{=}\gamma_k(\mu)(a_1,\ldots,a_{i-1}, \gamma_l(\nu) (a_i,\ldots,a_{i+l-1}),a_{i+l},\ldots,a_{k-1+l}) \otimes c_1 \cdots c_{k-1+l} \\
			&\stackrel{(\ref{morph})}{=}
			\overline{\gamma}_k(\mu) (a_1 \otimes c_1, \ldots, a_{i-1} \otimes c_{i-1}, \gamma_l(\nu)(a_i,\ldots,a_{i+l-1}) \otimes c_{i} \cdots c_{i+l-1}, a_{i+l}\otimes c_{i+l}, \ldots, a_{k-1+l} \otimes c_{k-1+l}) \\
			&\stackrel{(\ref{morph})}{=}
			\overline{\gamma}_k(\mu)(a_1 \otimes c_1, \ldots, a_{i-1} \otimes c_{i-1}, \overline{\gamma}_l(\nu)(a_i \otimes c_i,\ldots,a_{i+l-1}\otimes c_{i+l-1}), a_{i+l}\otimes c_{i+l}, \ldots, a_{k-1+l} \otimes c_{k-1+l})\\
			&\stackrel{(\ref{def_partial_comp_end_operad})}{=}
			\overline{\gamma}_k(\mu) \circ_i \overline{\gamma}_l(\nu) (a_1 \otimes c_1, \ldots, a_{k-1+l} \otimes c_{k-1+l}), 
		\end{align*}
		for any $a_1,\ldots,a_{k-1+l} \in \fa$ and $c_1,\ldots,c_{k-1+l} \in C$. Therefore, Diagram \eqref{prop_diag1_P-Alg_tensor_ComAlg} commutes for any $k,l \in \mathbb{N}$. Furthermore, it is obvious that $\overline{\gamma}_1(1)=\mathrm{id}_{\fa\otimes C}$. Next, to obtain  $\overline{\gamma}_k(\mu\cdot\sigma)=\overline{\gamma}_k\cdot\sigma$ for any $\sigma\in S_k$, we do the following computation. Let $\mu \in \p(k)$ and $\sigma \in \Sm_k$,  
		\begin{eqnarray*}
			\overline{\gamma}_k(\mu) \cdot \sigma (a_1 \otimes c_1, \ldots, a_k \otimes c_k) 
			&\stackrel{(\ref{def_end_opd_sym_action})}{=}& \overline{\gamma}_k(\mu) (a_{\sigma(1)} \otimes c_{\sigma(1)}, \ldots, a_{\sigma(k)} \otimes c_{\sigma(k)}) \\
			&\stackrel{(\ref{morph})}{=}& \gamma_k(\mu)(a_{\sigma(1)},\ldots,a_{\sigma(k)}) \otimes c_{\sigma(1)} \cdots c_{\sigma(k)} \\
			(\text{Since} ~C \in \mathrm{ComAlg})
			&=& \gamma_k(\mu)(a_{\sigma(1)},\ldots,a_{\sigma(k)}) \otimes c_1 \cdots c_k \\
			&\stackrel{(\ref{morphism_diagram})}{=}&
			\gamma_k(\mu \cdot \sigma) (a_1,\ldots,a_k) \otimes c_1 \cdots c_k \\  		 
			&\stackrel{(\ref{morph})}{=}& 
			\overline{\gamma}_k(\mu \cdot \sigma)(a_1 \otimes c_1, \ldots, a_k \otimes c_k), 
		\end{eqnarray*}
		for any $a_1,\ldots,a_k \in \fa$ and $c_1,\ldots,c_k \in C$. Hence, $\fa\otimes C$ is a $\p$-algebra. 
		
		\medskip
		Next, we show that a morphism $f: C \to D$ in $\mathrm{ComAlg}$ induces a morphism $\mathfrak{a} \otimes f:\fa \otimes C \to \fa \otimes D$ in $\p$-Alg. We define it as follows
		\begin{equation}
			\fa\otimes{f}(a \otimes c):=a \otimes f(c),
		\end{equation}
		for any $a \in \fa$ and $c \in C$. Now, using Equation \eqref{morph}, we show that $\fa \otimes f: \fa \otimes C \to \fa \otimes D$ is a $\p$-algebra morphism. In other words, we prove the commutativity of the following diagram
		\begin{equation}
			\begin{tikzcd}
				{(\mathfrak{a} \otimes C)^{\otimes k}} & 	{\mathfrak{a} \otimes C} \\
				{(\mathfrak{a} \otimes D)^{\otimes k}} & 	{\mathfrak{a} \otimes D,}
				\arrow["{\overline{\gamma}_k(\mu)}", 	from=1-1, to=1-2]
				\arrow["{(\mathrm{id} \otimes f)^{\otimes k}}"', from=1-1, to=2-1]
				\arrow["{\mathrm{id} \otimes f}", 	from=1-2, to=2-2]
				\arrow["{\overline{\gamma}_k'(\mu)}", 	from=2-1, to=2-2]
			\end{tikzcd}
		\end{equation}
		for any $\mu \in \p(k)$. Let $a_1,\ldots,a_k \in \fa$ and $c_1,\ldots,c_k \in C$, then 
		\begin{eqnarray*}
			(\mathrm{id} \otimes f) \big(\overline{\gamma}_k(\mu) (a_{1} \otimes c_1, \ldots, a_k \otimes c_k)\big) 
			&\stackrel{(\ref{morph})}{=}& 
			(\mathrm{id} \otimes f) \big(\gamma_k(\mu)(a_1,\ldots,a_k) \otimes c_1 \cdots c_k\big) \\
			&=&
			\gamma_k(\mu)(a_1,\ldots,a_k) \otimes f(c_1 \cdots c_k) \\
			&=&
			\gamma_k(\mu)(a_1,\ldots,a_k) \otimes f(c_1)\cdots f(c_k) \\
			&\stackrel{(\ref{morph})}{=}& 
			\overline{\gamma}_k'(\mu) (a_1 \otimes f(c_1), \ldots, a_k \otimes f(c_k)) \\ 
			&=& 
			\overline{\gamma}_k'(\mu) \big((\mathrm{id} \otimes f)^{\otimes k}(a_1\otimes c_1, \ldots, a_k\otimes c_k)\big). 
		\end{eqnarray*}
		Thus, any morphism $f: C \to D$ in $\mathrm{ComAlg}$ induces a morphism $\mathfrak{a} \otimes f:\fa \otimes C \to \fa \otimes D$ in $\p$-Alg. 
		Hence, giving a functor $\fa \otimes - : \mathrm{ComAlg} \to \p$-Alg.
	\end{proof}

	\begin{theorem}\label{thm_left_adjoint}
		Let $\p$ be an operad and $\fa$ be a $\p$-algebra. Then, the functor $\fa\otimes -: \mathrm{ComAlg} \to \p$-Alg admits a left adjoint if and only if $\fa$ is finite-dimensional.
	\end{theorem}
	
	\begin{proof}
		($\impliedby$)
		We assume $\fa$ to be a finite-dimensional $\p$-algebra with a vector space basis $\{a_1,\ldots,a_n\}$. Let the $\p$-algebra structure on $\fa$ is given by $\gamma:\p\rightarrow \mathrm{End}_\fa$. Then for any $\mu \in \p(k)$ we obtain 
		\begin{equation}\label{struc_const_rel1}
			\gamma_k(\mu)(a_{i_1},\ldots,a_{i_k})
			~=~\sum_{s=1}^n \alpha^s_{\gamma_k(\mu),i_1,\ldots,i_k} a_s,
		\end{equation}
  		where $i_1,\ldots,i_k = 1,\ldots,n$. Therefore, the set of scalars $\{\alpha^s_{\gamma_k(\mu),i_1,\ldots,i_k}~|~s,i_1,\ldots,i_k=1,\ldots,n,~\mu\in \p(k)\}$ yields the structure constants of $\fa$.
		Now, for an arbitrary $\p$-algebra $\mathfrak{b}$, we construct a commutative algebra $\mathcal{C}(\fa,\mathfrak{b}) \in \mathrm{ComAlg}$.
		Let $\{b_i|~i\in I\}$ be a vector space basis of $\mathfrak{b}$. Assuming the $\p$-algebra structure on $\mathfrak{b}$ to be given $\widetilde{\gamma}:\p\rightarrow \mathrm{End}_{\mathfrak{b}}$, for any $\mu \in \p(k)$ we obtain 
		\begin{equation}\label{struc_const_rel2}
			\widetilde{\gamma}_k(\mu)(b_{i_1},\ldots,b_{i_k}) \quad =
			\sum_{u\in I_{\widetilde{\gamma}_k(\mu),i_1,\ldots,i_k}} \beta^u_{\widetilde{\gamma}_k(\mu),i_1,\ldots,i_k} b_u,
		\end{equation}
		where $i_1,\ldots,i_k \in I$ and $I_{\widetilde{\gamma}_k(\mu),i_1,i_2,\ldots,i_k}$ is a finite subset of $I$. Thus, the structure constants of $\mathfrak{b}$ is given by the set
		$\{\beta^u_{\widetilde{\gamma}_k(\mu),i_1,\ldots,i_k}|~i_1,\ldots,i_k\in I,~u\in I_{\widetilde{\gamma}_k(\mu),i_1,i_2,\ldots,i_k},~\mbox{and}~\mu\in \p(k)\}$.
		Let $\mathbb{K}\big[X_{si}|~s=1,2,\ldots,n, ~i\in I\big]$ be the polynomial algebra with indeterminates $X_{si}$. We define 
		\begin{equation}\label{C(a,b)}
			\mathcal{C}(\fa,\mathfrak{b}) := \mathbb{K}\big[X_{si}|~s=1,2,\ldots,n, ~i\in I\big] \big/ J,
		\end{equation}
		where $J$ is the ideal generated by the polynomials of the form
		\begin{equation}
			P^{(\fa,\mathfrak{b})}_{(\mu,a,i_1,\ldots,i_k)} ~:=
			\sum_{u\in I_{\widetilde{\gamma}_k(\mu),i_1,\ldots,i_k}} \beta^u_{\widetilde{\gamma}_k(\mu),i_1,\ldots,i_k} X_{au}
			\hspace{0.2cm} -
			\sum_{s_1,s_2,\ldots, s_k=1}^n \alpha^a_{\gamma_k(\mu),s_1,\ldots,s_k}X_{s_1 i_1}\cdots X_{s_k i_k},
		\end{equation}
		for any $\mu\in \p(k)$, $a=1,2,\ldots, n$, and $i_1,\ldots, i_k\in I$. Let $x_{si}$ denote the class $\widehat{X_{si}}$ of $X_{si}$ in the algebra $\mathcal{C}(\fa,\mathfrak{b})$. This gives us the following equality in  $\mathcal{C}(\fa,\mathfrak{b})$: 
		\begin{equation}\label{thm_rel_struc_const_a_b}
			\sum_{u\in I_{\widetilde{\gamma}_k(\mu),i_1,\ldots,i_k}} \beta^u_{\widetilde{\gamma}_k(\mu),i_1,\ldots,i_k} x_{au}
			\hspace{0.2cm} =~ 
			\sum_{s_1,s_2,\ldots, s_k=1}^n \alpha^a_{\gamma_k(\mu),s_1,\ldots,s_k}x_{s_1 i_1}\cdots x_{s_k i_k}, 
		\end{equation} 
		for any $\mu \in \p(k)$, $a =1,\ldots,n$, and  $i_1,\ldots,i_k \in I$.
		
		\bigskip
		Next, we define a map 
		$\eta_{\mathfrak{b}}:\mathfrak{b}\rightarrow \fa\otimes \mathcal{C}(\fa,\mathfrak{b})$ by 
		\begin{equation} \label{thm_def_eta_b}
			\eta_\mathfrak{b}(b_i):=\sum_{s=1}^n a_s \otimes x_{si}\quad \mbox{for any } i\in I,
		\end{equation}
		and show that $\eta_\mathfrak{b}$ is a $\p$-algebra morphism. That is, the following diagram commutes for any $\mu \in \p(k)$
		
		\begin{equation}
			\begin{tikzcd}
				{\mathfrak{b}^{\otimes k}} & {\mathfrak{b}} \\
				{(\fa \otimes \mathcal{C}(\fa,\mathfrak{b}))^{\otimes k}} & {\fa \otimes \mathcal{C}(\fa,\mathfrak{b}),}
				\arrow["{\widetilde{\gamma}_{k}(\mu)}", from=1-1, to=1-2]
				\arrow["{\eta_{\mathfrak{b}}^{\otimes k}}"', from=1-1, to=2-1]
				\arrow["{\eta_{\mathfrak{b}}}", from=1-2, to=2-2]
				\arrow["{\overline{\gamma}_{k}(\mu)}", from=2-1, to=2-2]
			\end{tikzcd}
		\end{equation}
		where the operad morphism $\overline{\gamma}: \p \to \mathrm{End}_{\fa \otimes \mathcal{C}(\fa,\mathfrak{b})}$ is induced from the operad morphism $\gamma: \p \to \mathrm{End}_\mathfrak{a}$ as defined in Equation (\ref{morph}). The commutativity of the above diagram follows from the computation given below. Let $\mu \in \p(k)$, and ${i_1},\ldots,{i_k} \in I$, then
		\begin{eqnarray*}
			\overline{\gamma}_k(\mu) \big(\eta_\mathfrak{b}^{\otimes k}(b_{i_1}, \ldots, b_{i_k})\big)
			&\stackrel{(\ref{thm_def_eta_b})}{=}&
			\overline{\gamma}_k(\mu) \Big(\sum_{s_1=1}^n a_{s_1} \otimes x_{s_1i_1}, \ldots, \sum_{s_k=1}^n a_{s_k} \otimes x_{s_ki_k}\Big) \\
			&=& 
			\sum_{s_1,\ldots,s_k=1}^n \gamma_k(\mu) (a_{s_1},\ldots,a_{s_k}) \otimes x_{s_1i_i} \cdots x_{s_ki_k} \\
			&\stackrel{(\ref{struc_const_rel1})}{=}&
			\sum_{i=1}^n ~a_i \hspace{0.2cm} \otimes \sum_{s_1,\ldots,s_k=1}^n \alpha^i_{\gamma_k(\mu),s_1,\ldots,s_k} x_{s_1i_i} \cdots x_{s_ki_k} \\
			&\stackrel{(\ref{thm_rel_struc_const_a_b})}{=}&\sum_{i=1}^n~ a_i \hspace{0.2cm} \otimes \sum_{u\in I_{\widetilde{\gamma}_k(\mu),i_1,\ldots,i_k}} \beta^u_{\widetilde{\gamma}_k(\mu),i_1,\ldots,i_k} x_{iu}\\
			&=&
			\sum_{u\in I_{\widetilde{\gamma}_k(\mu),i_1,\ldots,i_k}} \beta^u_{\widetilde{\gamma}_k(\mu),i_1,\ldots,i_k} \Big(\sum_{i=1}^n a_i \otimes x_{iu}\Big) \\
			&\stackrel{(\ref{thm_def_eta_b})}{=}&
			\sum_{u\in I_{\widetilde{\gamma}_k(\mu),i_1,\ldots,i_k}} \beta^u_{\widetilde{\gamma}_k(\mu),i_1,\ldots,i_k} ~\eta_\mathfrak{b}(b_u)\\
			&=&
			\eta_\mathfrak{b}~\Big(\sum_{u\in I_{\widetilde{\gamma}_k(\mu),i_1,\ldots,i_k}} \beta^u_{\widetilde{\gamma}_k(\mu),i_1,\ldots,i_k} ~b_u\Big) \\
			&\stackrel{(\ref{struc_const_rel2})}{=}&
			\eta_\mathfrak{b} \big(\widetilde{\gamma}_k(\mu)(b_{i_1}, \ldots, b_{i_k})\big).
  		\end{eqnarray*} 
		Thus, $\eta_{\mathfrak{b}}$ is $\p$-algebra morphism. Let $C \in \mathrm{ComAlg}$ and $f:\mathfrak{b}\rightarrow \fa\otimes C$ be a $\p$-algebra morphism, then we show that there exists a unique algebra morphism $\Phi: \mathcal{C}(\fa,\mathfrak{b})\rightarrow C$ for which the following diagram commutes
		
		\begin{equation}\label{universal prop}
			\begin{tikzcd}
				{\mathfrak{b }} & {\fa \otimes \mathcal{C}(\fa, \mathfrak{b})} \\
				& {\fa \otimes C.}
				\arrow["{\eta_\mathfrak{b}}", from=1-1, to=1-2]
				\arrow["f"', from=1-1, to=2-2]
				\arrow["{id \otimes \Phi}", from=1-2, to=2-2]
			\end{tikzcd}
		\end{equation}
		Let $\{c_{si}|~ s=1,2,\ldots,n, ~i\in I\}$ be a subset of $C$ such that for any $i \in I$, we have
		\begin{equation}
			f(b_i)=\sum_{s=1}^n a_s\otimes c_{si}. 
		\end{equation}
		Since $f:\mathfrak{b}\rightarrow \fa\otimes C$ is a morphism of $\p$-algebras, we get
		\begin{equation}\label{thm_morph_f}
			f\big(\widetilde{\gamma}_k(\mu)(b_{i_1},\ldots,b_{i_k})\big) ~=~ \overline{\gamma}_k(\mu) \big(f(b_{i_1}),\ldots,f(b_{i_k})\big), \mbox{ for any } ~\mu\in \p(k),~ i_1,\ldots,i_k \in I.
		\end{equation}
		Simplifying the left hand side of Equation (\ref{thm_morph_f}), 
		\begin{equation*}
			f\big(\widetilde{\gamma}_k(\mu)(b_{i_1},\ldots,b_{i_k})\big) ~=~ \sum_{s=1}^n a_s \otimes 
			\Big(\sum_{u\in I_{\widetilde{\gamma}_k(\mu),i_1,\ldots,i_k}} \beta^u_{\widetilde{\gamma}_k(\mu),i_1,\ldots,i_k} c_{su}\Big).
		\end{equation*}
		Similarly, simplifying the right hand side of Equation (\ref{thm_morph_f}), we obtain
		\begin{eqnarray*}
			\overline{\gamma}_k(\mu) \big(f(b_{i_1}),\ldots,f(b_{i_k})\big) 
			&=& 
			\sum_{s_1,\ldots,s_k=1}^n \gamma_{k}(\mu) (a_{s_1},\ldots,a_{s_k}) \otimes c_{s_1 i_1}\cdots~ c_{s_k i_k} \\
			&=& 
			\sum_{s=1}^n a_s \otimes ~\Big(\sum_{s_1,\ldots,s_k=1}^n \alpha^s_{\gamma_k(\mu),s_1,\ldots,s_k}~
			c_{s_1 i_1}\cdots~ c_{s_k i_k}\Big), ~\mbox{for any }~ i_1,\ldots,i_k \in I.
		\end{eqnarray*}
		 Therefore, the elements of the set $\{c_{si}|~ ~s=1,2,\ldots,n,~ i\in I\}$ satisfy the following equation 
		\begin{equation}\label{relations on csi}
			\sum_{u\in I_{\widetilde{\gamma}_k(\mu),i_1,\ldots,i_k}} \beta^u_{\widetilde{\gamma}_k(\mu),i_1,\ldots,i_k} c_{su}
			\hspace{0.2cm} =
			\sum_{s_1,\ldots, s_k=1}^n 	\alpha^s_{\gamma_k(\mu),s_1,\ldots,s_k}~c_{s_1 i_1}\cdots~ c_{s_k i_k},
		\end{equation}
		for any $\mu \in \p(k)$, $i_1,\ldots, i_k\in I$, and $s=1,2,\ldots,n$. By defining $\phi(X_{si})=c_{si}$, we obtain a unique algebra morphism $\phi: \mathbb{K}[X_{si}]\rightarrow C$. Equation \eqref{relations on csi} implies that $J \subseteq \mathrm{Ker}(\phi)$. Thus, inducing an algebra morphism on the quotient algebra, which we denote by the map ${\Phi}:\mathcal{C}(\fa,\mathfrak{b})\rightarrow C$. Note that
		\begin{equation}
			f(b_i)=\sum_{s=1}^n a_s\otimes c_{si}=\sum_{s=1}^n a_s\otimes \Phi(x_{si})=(\mathrm{id}\otimes \Phi)\Big(\sum_{s=1}^n a_s\otimes x_{si}\Big)=(\mathrm{id}\otimes \Phi)\circ \eta_\mathfrak{b}(b_i),\quad \mbox{for all }i\in I.
		\end{equation}
		Hence, Diagram \eqref{universal prop} commutes. One can show easily that the obtained map $\Phi$ is unique. Therefore, for any $\p$-algebra $\mathfrak{b}$ and commutative algebra $C$, we have the following bijection
		\begin{equation}\label{bijection1}
			\Omega_{\mathfrak{b},C}:\mathrm{Hom}_{\mathrm{ComAlg}}(\mathcal{C}(\fa,\mathfrak{b}),C)\longrightarrow \mathrm{Hom}_{\p\text{-Alg}}(\mathfrak{b},\fa\otimes C) \quad \mbox{defined by }~ \Omega_{\mathfrak{b},C}(\Phi)=(\mathrm{id}\otimes \Phi)\circ \eta_\mathfrak{b}.
		\end{equation}
		
		\medskip
		Now, one can verify that assigning to each $\p$-algebra $\mathfrak{b}$ the commutative algebra $\mathcal{C}(\fa,\mathfrak{b})$, defines a functor $$\mathcal{C}(\fa,-): \p\text{-Alg} \longrightarrow \mathrm{ComAlg}.$$ 
		Furthermore, for any $\p$-algebra $\mathfrak{b}^\prime$ with a $\p$-algebra morphism $\alpha:\mathfrak{b}\rightarrow \mathfrak{b}^\prime$ and for any commutative algebra $C$ with an algebra morphism $\beta:C\rightarrow C^\prime$, we obtain a commutative diagram as follows
		\begin{equation}
			\begin{CD}
				\mathrm{Hom}_{\mathrm{ComAlg}}(\mathcal{C}(\fa,\mathfrak{b}^\prime),C) @> \mathcal{C}(\alpha)^* >>
				\mathrm{Hom}_{\mathrm{ComAlg}}(\mathcal{C}(\fa,\mathfrak{b}),C) @> \beta_{*} >> \mathrm{Hom}_{\mathrm{ComAlg}}(\mathcal{C}(\fa,\mathfrak{b}),C^\prime)\\
				@V\Omega_{\mathfrak{b}^\prime,C} VV @V\Omega_{\mathfrak{b},C} VV @V\Omega_{\mathfrak{b},C^\prime} VV\\
				\mathrm{Hom}_{\p\text{-Alg}}(\mathfrak{b}^\prime,\fa\otimes C) @> \alpha^* >>
				\mathrm{Hom}_{\p\text{-Alg}}(\mathfrak{b},\fa\otimes C) @> \fa\otimes \beta_{*} >> \mathrm{Hom}_{\p\text{-Alg}}(\mathfrak{b},\fa\otimes C^\prime).
			\end{CD}
		\end{equation}
		Thus, $\Omega_{\mathfrak{b},C}$ is natural in both $\mathfrak{b}$ and $C$. Consequently, $\mathcal{C}(\fa,-):\p\text{-Alg}\rightarrow \mathrm{ComAlg}$ is left adjoint to the functor $\fa\otimes-:\mathrm{ComAlg} \to \p$-Alg. 
		
		\medskip\noindent
		($\implies$)
		Before proving the converse part, 
		we show that the category $\p$-Alg is closed under arbitrary product.
		Let $(\fa_i)_{i \in I}$ be a family of objects in $\p$-Alg with the $\p$-algebra structures on $\fa_i$ given by $\gamma^i: \p \to \mathrm{End}_{\fa_i}$. 
		Since the category $\mathrm{Vect}$ is closed under arbitrary product, we obtain the product of $(\fa_i)_{i \in I}$ in $\mathrm{Vect}$, given by 
		\begin{equation}\label{arb_product_P_alg}
			\prod_{i \in I} \fa_i := \big\{(a_i)_{i \in I} :a_i \in \fa_i\big\}. 
		\end{equation}
		along with projection maps $\pi_j: \prod_{i \in I} \fa_i \to \fa_j$ for each $j \in I$. 
		Now, to show that the product of $(\fa_i)_{i \in I}$ in $\p$-Alg is also given by the vector space $\prod_{i \in I} \fa_i$ we must show that the following things hold:
		\begin{itemize}
			\item There exists a $\p$-algebra structure on $\prod_{i \in I} \fa_i$. 
			\item Each projection maps $\pi_j: \prod_{i \in I} \fa_i \to \fa_j$ are $\p$-algebra morphisms. 
			\item For an arbitrary $\p$-algebra $\mathfrak{b}$ along with $\p$-algebra morphisms $\phi_j: \mathfrak{b} \to \fa_j$ for each $j \in I$, there exists a unique $\theta:\mathfrak{b}\to \prod_{i \in I}{\fa_i} $ such that the following diagram commutes for every $j\in I$.
			\begin{equation}\label{universality_prod}
				\begin{tikzcd}
				& {\mathfrak{b}} \\
				& {\prod_{i \in I} \mathfrak{a}_{i}} \\
				{\mathfrak{a}_j}
				\arrow["\theta", dashed, from=1-2, to=2-2]
				\arrow["{\phi_j}"', bend right=35, from=1-2, to=3-1]
				\arrow["{\pi_j}", from=2-2, to=3-1]
			\end{tikzcd}
		\end{equation} 
 		\end{itemize} 
		We note that the $\p$-algebra structure, given by $\gamma^i: \p \to \mathrm{End}_{\fa_i}$, on each $\fa_i$ induces a $\p$-algebra structure on $\prod_{i \in I } \fa_i$ which is denoted by $\psi:\p \to \mathrm{End}_{\prod_{i \in I} \fa_i}$ and is defined as follows:
		\begin{equation}\label{arb_product_P_alg_structure}
			\psi_k(\mu)\big((a^1_{i})_{i \in  I},\ldots,(a^k_{i})_{i \in I}\big)  
			:= 
			\big(\gamma^i_k(\mu)(a^1_{i},\ldots,a^k_{i})\big)_{i \in I}	\quad \text{ for all } \mu \in \p(k). 
		\end{equation}		 
		Thus, $\prod_{i \in I} \fa_i$ is an object in $\p$-Alg. 
		Next, to show that each projection map $\pi_j : \prod_{i \in I} \fa_i \to \fa_j$ is a $\p$-algebra morphism we show the commutativity the diagram given below: for each $k \in \mathbb{N}$ and $\mu \in \p(k)$
		\begin{equation}\label{proj_p_alg_mor_diag}
			\begin{tikzcd}
				{(\prod \fa_i)^{\otimes k}} & {\prod\fa_i} \\
				{\fa_j^{\otimes k}} & {\fa_j}
				\arrow["{\psi_k(\mu)}", from=1-1, to=1-2]
				\arrow["{\pi^{\otimes k}_j}"', from=1-1, to=2-1]
				\arrow["{\pi_j}", from=1-2, to=2-2]
				\arrow["{\gamma^{j}_k(\mu)}", from=2-1, to=2-2]
			\end{tikzcd}
		\end{equation}
		The commutativity of the above diagram follows from the following computation:
		\begin{eqnarray*}
			\big(\pi_j \circ \psi_k(\mu)\big)\big((a^1_{i})_{i \in I},\ldots,(a^k_{i})_{i \in I}\big) 
			&\stackrel{\eqref{proj_p_alg_mor_diag}}{=}& \pi_j\big(\big(\gamma^i_{k}(\mu)(a^1_{i},\ldots,a^k_{i})\big)_{i\in I}\big)\\
			&=& \gamma^j_{k}(\mu)\big(a^1_{j},\ldots,a^k_{j}\big)\\
			&=& \gamma^j_{k}(\mu)\big(\pi_j((a^1_{i})_{i\in I}),\ldots,\pi_j((a^k_{i})_{i \in I})\big)\\
			&=& \big(\gamma^j_{k}(\mu)\circ \pi_j^{\otimes k}\big)\big((a^1_{i})_{i \in I},\ldots,(a^k_{i})_{i \in I}\big).
		\end{eqnarray*}
		Hence, we obtain that each projection maps $\pi_j : \prod_{i \in I} \fa_i \to \fa_j$ are $\p$-algebra morphisms. 
		Finally, for an arbitrary $\p$-algebra $\mathfrak{b}$, given by $\beta :\p \to \text{End}_{\mathfrak{b}}$, along with $\p$-algebra morphism $\phi_{j} : \mathfrak{b}\to \fa_j$ for each $j \in I$, we show the existence of an unique $\p$-algebra morphism $\theta:\mathfrak{b}\to \prod_{i \in I}{\fa_i} $ for which Diagram \eqref{universality_prod} commutes for every $j\in I$.
		At first we note that since $\phi_j: \mathfrak{b} \to \mathfrak{a}_j$ is a $\p$-algebra morphism we have the following equality:
		\begin{equation}
			(\phi_j\circ\beta_k(\mu))(b^1,\ldots,b^k) ~=~ \gamma^j_{k}(\mu)(\phi_j(b^1),\ldots,\phi_j(b^k))
		\end{equation}
		for any $b^1,\cdots, b^k \in \mathfrak{b}$. We define a map $\theta: \mathfrak{b} \to \prod_{i \in I} \fa_i$ as follows:
		\begin{equation}
			\theta(b) := \big(\phi_i(b)\big)_{i \in I} \quad \text{for all } b \in \mathfrak{b}.  
		\end{equation}
		From the definition of $\theta$ it is clear that Diagram \eqref{universality_prod} commutes, since we have the following equality: 
		\begin{equation}
			(\pi_j\circ\theta)(b) ~=~ \pi_j((\phi_i(b))_{i \in I}) ~=~ \phi_j(b), \quad \text{for all } b \in \mathfrak{b}. 
		\end{equation}
		Thus, it only remains to show that $\theta: \mathfrak{b} \to \prod_{i \in I} \fa_i$ is $\p$-algebra morphism, which follows from the computation below: for any $b^1,\cdots,b^k\in \mathfrak{b}$
		\begin{eqnarray*}
			\theta \circ \beta_k(\mu)\big(b^1,\ldots,b^k\big)
			&=&\big(\phi_i\big(\beta_k(\mu)(b^1,\ldots,b^k)\big)\big)_{i\in I}\\
			&=&\big(\gamma^{i}_k(\mu)\big(\phi_i(b^1),\ldots,\phi_i(b^k)\big)\big)_{i\in I}\\
			&=&\psi_k(\mu)\big((\phi_i(b^1))_{i\in I},\ldots ,(\phi_i(b^k))_{i\in I}\big)\\
			&=&\psi_k(\mu)\big(\theta(b^1),\ldots,\theta(b^k)\big)\\
			&=&\psi_k(\mu)\circ\theta^{\otimes k}\big(b^1,\ldots,b^k\big).
		\end{eqnarray*}
		The uniqueness of $\theta$ is easy to verify. 
		Hence, we obtain that $\prod_{i \in I} \fa_i$ is the product of $\big(\fa_i\big)_{i \in I}$ is the $\p$-Alg.
		
		\medskip
		Now, following the arguments in the proof of \cite[Theorem 2.1]{Ag-Mil} we prove the converse part of this theorem. Assuming the functor $\fa \otimes-$ admits a left adjoint, we aim to show that $\fa$ is finite-dimensional. Since $\fa \otimes -$ admits a left adjoint it preserves all the limits that exists in the category $\mathrm{ComAlg}$ (cf. \cite[pg. 118]{MacLane}). Therefore, it preserves arbitrary products in $\mathrm{ComAlg}$. Furthermore, Equation \eqref{arb_product_P_alg} and \eqref{arb_product_P_alg_structure}, tells us that in $\p$-Alg products are constructed as the products of the underlying vector spaces. 
		Imposing the condition that $\fa \otimes -$ preserves the product of countable copies of the base field $\K$ in $\mathrm{ComAlg}$, we get $\fa$ is finite-dimensional. 
	\end{proof} 

	\begin{corollary}
		Let $\fa$ and $\mathfrak{b}$ be two $\p$-algebras with $\fa$ being finite-dimensional. Then the following map is bijective
		\begin{equation}
			\Omega:\mathrm{Hom}_{\mathrm{ComAlg}}(\mathcal{C}(\fa,\mathfrak{b}),\K)\longrightarrow \mathrm{Hom}_{\p\text{-Alg}}(\mathfrak{b},\fa) \quad \mbox{defined by }~ \Omega(\Phi)=(\mathrm{id}\otimes \Phi)\circ \eta_\mathfrak{b}.
		\end{equation}
	\end{corollary}

	\begin{definition}\label{def_universal_algebra}
		Let $\fa$ be a finite-dimensional $\mathcal{P}$-algebra with basis $\{a_1,\ldots,a_n\}$. We say that the commutative algebra $\mathcal{C}(\fa,\fa)$ (cf. Equation (\ref{C(a,b)})) is said to be the \textbf{universal algebra of $\fa$}, and is denoted by $\mathcal{C}(\fa)$. Recall that 
		\begin{equation}\label{C(a)}
			\mathcal{C}(\fa) := \mathbb{K}\big[X_{si}|~s,i=1,2,\ldots,n\big] \big/ J,
		\end{equation}
		where $J$ is the ideal generated by the polynomials of the form
		\begin{equation}\label{def_universal_polynomial_a}
			\p^{(\fa)}_{(\mu,a,i_1,\ldots,i_k)} := \sum_{u=1}^n \alpha^u_{\gamma_k(\mu),i_1,\ldots,i_k} X_{au}
			\hspace{0.2cm} - \sum_{s_1,s_2,\ldots, s_k=1}^n \alpha^a_{\gamma_k(\mu),s_1,\ldots,s_k}X_{s_1 i_1}\cdots X_{s_k i_k}
		\end{equation} 
		for any $k \in \mathbb{N}$ with $\mu \in \p(k)$ and $a,i_1,\ldots,i_k =1,\ldots,n$. The polynomials defined above are called the {\bf universal polynomials of $\fa$}. 
	\end{definition}

 	For any finite-dimensional $\p$-algebra $\fa$, we now state as a corollary the universal property of $\mathcal{C}(\fa)$, that was obtained in the proof of Theorem \ref{thm_left_adjoint}.

	\begin{corollary}\label{cor_univ_prop}
		Let $\fa$ be a finite-dimensional $\p$-algebra. Then for any commutative algebra $C$ and any $\p$-algebra morphism $f: \fa \to \fa \otimes C$, there exists a unique commutative algebra morphism $\Phi: \mathcal{C}(\fa) \to C$ such that $f = (\mathrm{id} \otimes \Phi) \circ \eta_{\fa}$, i.e., the following diagram is commutative: 
		\begin{equation}\label{cor_diag_universal_property}
			\begin{tikzcd}
				{\mathfrak{a}} & {\fa \otimes \mathcal{C}(\fa)} \\
				& {\fa \otimes C}
				\arrow["{\eta_\mathfrak{a}}", from=1-1, to=1-2]
				\arrow["f"', from=1-1, to=2-2]
				\arrow["{id \otimes \Phi}", from=1-2, to=2-2]
			\end{tikzcd}
		\end{equation}   
	\end{corollary}
	
	\begin{proof}
		The proof follows directly from the bijection provided in Equation (\ref{bijection1}), for the case $\mathfrak{b} = \fa$. 
	\end{proof}
	
	\begin{remark}
	In the next section, we prove that the construction in Theorem \ref{thm_left_adjoint} yield an enrichment of the category of finite-dimensional $\mathcal{P}$-algebras over the category dual to the category $\mathrm{ComAlg}$ of commutative algebras. Subsequently, we obtain a natural universal coacting bi/Hopf algebra for any finite dimensional $\mathcal{P}$-algebra.
	\end{remark}

\medskip
\subsection{Algebras over a Graded Operad}\label{sub3.2}~~\\
	
	\vspace{-0.4cm}
	In this subsection, we present the construction of a universal graded algebra associated with an algebra over a graded operad. Our aim is to demonstrate that the results obtained for algebras over an operad extends in a natural way to the setting of graded operads.  
	
	\medskip
	Let $\mathcal{P}$ be a graded operad and $\mathfrak{A}$ be an arbitrary $\mathcal{P}$-algebra.  
	Recall that a {$\mathcal{P}$-algebra} structure on $\mathfrak{A}$ is given by a morphism of graded operads $\gamma: \mathcal{P} \to \mathrm{End}_\mathfrak{A}$. 
Similar to the proof of Proposition \ref{funct_com_to_p-alg}, one can verify that  for any $\mathfrak{A} \in \mathcal{P}$-Alg, the map $\mathfrak{A}  \otimes - : \mathrm{g}$-$\mathrm{ComAlg} \to  \mathcal{P}$-Alg defines a functor. In particular, if $C \in \mathrm{g}\text{-}\mathrm{ComAlg}$, then the $\mathcal{P}$-algebra structure on $\mathfrak{A} \otimes C$ is given by the graded operad morphism $\overline{\gamma}: \mathcal{P} \to \mathrm{End}_{\mathfrak{A}\otimes C}$, which is
	defined as 
	\begin{equation}\label{graded_case_gamma_bar}
		\overline{\gamma}_k(\mu)(a_{p_1i_1} \otimes c_{q_1j_1},\ldots,a_{p_ki_k} \otimes c_{q_kj_k}) ~:=~ 
		(-1)^{\varrho}~\gamma_{k}(\mu) (a_{p_1i_1}, \ldots, a_{p_ki_k}) \otimes c_{q_1j_1} \cdots c_{q_kj_k},
	\end{equation} 
	where $\varrho = (-1)^{p_kq_{k-1} + (p_{k-1} + p_k)q_{k-2} + \cdots + (p_2+\cdots+p_k)q_1}$ if $k \ge 2$, (the sign is obtained by following the Koszul sign convention) otherwise $\varrho = 0$. 
	
	\medskip
	Let $\mathfrak{A} = \oplus_{p \ge 0} \mathfrak{A}_p$  be a $\mathcal{P}$-algebra such that each component $\mathfrak{A}_p$ ($p \ge 0$) is finite-dimensional. 
	Then, we construct a universal graded algebra of $\mathfrak{A}$. Let  $$\{a_{01},\ldots,a_{0d_0};a_{11},\ldots,a_{1d_1};\ldots;a_{p1},\ldots,a_{pd_p};\ldots\}$$
	be a vector space basis of $\mathfrak{A}$, 
	where $\{a_{p1},\ldots,a_{p{d_p}}\}$ denotes a basis of the $p$-degree component $\mathfrak{A}_p$. If the $\mathcal{P}$-algebra structure on $\mathfrak{A}$ is given by $\gamma:\mathcal{P}\rightarrow \mathrm{End}_\mathfrak{A}$, then 
	\begin{equation}\label{struc_const_graded}
		\gamma_k(\mu)(a_{p_1i_1},\ldots,a_{p_ki_k})
		~=~\sum_{s=1}^{d_{\theta}} \alpha^{\theta,s }_{\gamma_k(\mu),p_1i_1,\ldots,p_ki_k} a_{\theta s}\quad \text{for any } \mu \in \mathcal{P}(k),
	\end{equation}
	where $\theta = p_1+\cdots+p_k+|\mu|$, $|\mu|$ denotes the degree of $\mu$ in the graded operad, $a_{p_ri_r} \in \mathfrak{A}_{p_r}$ for any $1 \le r \le k$. The set of scalars 
	\begin{equation*}
		\bigcup_{k \ge 1} \bigcup_{\mu \in \mathcal{P}(k)}\bigcup_{p_1,\ldots,p_k \ge 0}\Big\{
		\alpha^{\theta,s }_{\gamma_k(\mu),p_1i_1,\ldots,p_ki_k} ~\big|~ 1 \le i_r \le d_{p_r} \text{ for any } r=1,\ldots,k, \mbox{ and } 1 \le s \le d_{p_1 + \cdots + p_k + |\mu|}\Big\}
	\end{equation*}
 	yields the structure constants of $\mathfrak{A}$.
	Let $\mathbb{K}\left[X^{(\pi)}_{si}~\big|~s,i,\pi \ge 0 \right]$ be the graded commutative polynomial algebra, 
	with indeterminates $X^{(\pi)}_{si}$ of cohomological degree $\pi$. We define 
	\begin{equation}\label{C(A)}
		\mathcal{C}(\mathfrak{A}) := \mathbb{K}\big[X^{(\pi)}_{si}~\big|~s,i,\pi\geq 0 \big] \Big/ J,
	\end{equation}
	where $J$ is the ideal generated by graded polynomials (of cohomological degree $\omega \geq 0$) given by  
	\begin{multline}\label{graded universal poly}
		P^{(\mathfrak{A},\omega)}_{(\mu,a,p_1i_1,\ldots,p_ki_k)} 
		:=
		\sum_{u=1}^{d_{\theta}} \alpha^{\theta,u}_{\gamma_k(\mu),p_1i_1,\ldots,p_ki_k} ~X^{(\omega)}_{au}
		\\
		-
		\sum_{\substack{\varepsilon_1+\cdots+\varepsilon_k
		=\theta-\omega-|\mu|, \\ 0\leq \varepsilon_i \le p_i,~1\leq i\leq k.}}~
		\sum_{s_1=1}^{d_{p_1}} 
		\cdots 
		\sum_{s_k=1}^{d_{p_k}}	
		(-1)^\varrho
		\alpha^{\theta-\omega,a }_{\gamma_k(\mu),\varepsilon_1s_1,\ldots,\varepsilon_ks_k}X_{s_1 i_1}^{(p_1-\varepsilon_1)}\cdots X_{s_k i_k}^{(p_k-\varepsilon_k)},
	\end{multline}
	where $\mu \in \mathcal{P}(k)$, $p_1,\ldots,p_k \ge 0$, $\theta=p_1 + \cdots + p_k+|\mu|$, $0 \le \omega \le \theta$, $a=1,\ldots,d_{\theta-\omega}$, and $1\leq i_r\leq d_{p_r}$ for any $1\leq r\leq k$. Here, $\varrho$ in the second term of the right hand side of the expression \eqref{graded universal poly} is given by the sum 
	\begin{equation}\label{koszul_sign}
		\varrho =  \sum_{j=1}^{k-1}(\varepsilon_{j+1}+\cdots+\varepsilon_k)(p_j -\varepsilon_j) \mbox{ for any }k \ge 2,
	\end{equation}
	otherwise $\varrho = 0$. The polynomials in Equation \eqref{graded universal poly} are called the \textbf{universal graded polynomials of $\mathfrak{A}$}. We denote by $x^{(\pi)}_{si}$ the class of $X_{si}^{(\pi)}$ in the graded algebra $\mathcal{C}(\mathfrak{A})$.
	Next, we define a map 
	$\eta_{\mathfrak{A}}:\mathfrak{A}\rightarrow \mathfrak{A}\otimes \mathcal{C}(\mathfrak{A})$ by 
	\begin{equation} \label{def_eta_A}
		\eta_\mathfrak{A}(a_{pi}):=\sum_{\varepsilon =0}^p\sum_{s=1}^{d_\varepsilon} a_{\varepsilon s} \otimes x^{(p - \varepsilon)}_{si}\quad \mbox{for any } p \geq 0 ,~ i=1,2,\ldots,d_p.
	\end{equation}
	In the lemma below we show that the map $\eta_\mathfrak{A}: \mathfrak{A} \to \mathfrak{A} \otimes \mathcal{C}(\mathfrak{A})$ is a $\p$-algebra morphism. 
	
	\begin{lemma}\label{lemma_graded_eta_A}
		The map $\eta_\mathfrak{A}: \mathfrak{A} \to \mathfrak{A} \otimes \mathcal{C}(\mathfrak{A})$ is a $\mathcal{P}$-algebra morphism. 
	\end{lemma}

	\begin{proof}
		For any $k \ge 1$, let $a_{p_1i_1},\ldots,a_{p_ki_k} \in \mathfrak{A}$ of degree $p_1,\ldots,p_k$, respectively. To show that $\eta_\mathfrak{A}$ is a $\mathcal{P}$-algebra morphism, one needs to verify the commutativity of the following diagram for all $\mu \in \mathcal{P}(k)$.
		\begin{equation}
			\begin{tikzcd}
				{\mathfrak{A}^{\otimes k}} & {\mathfrak{A}} \\
				{(\mathfrak{A} \otimes \mathcal{C}(\mathfrak{A}))^{\otimes k}} & {\mathfrak{A} \otimes \mathcal{C}(\mathfrak{A}).}
				\arrow["{\gamma_{k}(\mu)}", from=1-1, to=1-2]
				\arrow["{\eta_{\mathfrak{A}}^{\otimes k}}"', from=1-1, to=2-1]
				\arrow["{\eta_{\mathfrak{A}}}", from=1-2, to=2-2]
				\arrow["{\overline{\gamma}_{k}(\mu)}", from=2-1, to=2-2]
			\end{tikzcd}
		\end{equation}
		Commutativity of the above diagram is obtained from the following computation:
		\begin{eqnarray*}
			\overline{\gamma}_k(\mu)
			\hspace{-0.7cm}&~&\hspace{-0.7cm} 
			\big(\eta_\mathfrak{A}(a_{p_1i_1}), \ldots
			,\eta_\mathfrak{A}(a_{p_ki_k})\big) \\
			&\stackrel{\eqref{def_eta_A}}{=}& 
			\overline{\gamma}_k(\mu) 
			\Big(\sum_{\varepsilon_1 = 0}^{p_1}\sum_{s_1=1}^{d_{\varepsilon_1}} a_{\varepsilon_1 s_1} \otimes x_{s_1i_1}^{(p_1-\varepsilon_1)}
			,\ldots, 
			\sum_{\varepsilon_k = 0}^{p_k}\sum_{s_k=1}^{d_{\varepsilon_k}} a_{\varepsilon_k s_k} \otimes x_{s_ki_k}^{(p_k-\varepsilon_k)}\Big) \\
			&\stackrel{\eqref{graded_case_gamma_bar}}{=}& 
			\sum_{\varepsilon_1=0}^{p_1} \cdots \sum_{\varepsilon_k=0}^{p_k}
			\sum_{s_1=1}^{d_{p_1}} \cdots 
			\sum_{s_k=1}^{d_{p_k}} 
			(-1)^{\varrho} \gamma_k(\mu) (a_{\varepsilon_1s_1},\ldots,a_{\varepsilon_ks_k}) 
			\otimes 
			x_{s_1i_1}^{(p_1-\varepsilon_1)}
			\cdots 
			~x_{s_ki_k}^{(p_k-\varepsilon_k)} \\
			&\stackrel{\eqref{struc_const_graded}}{=}& 
			\sum_{\varepsilon_1=0}^{p_1} \cdots \sum_{\varepsilon_k=0}^{p_k}
			\sum_{s_1=1}^{d_{p_1}} \cdots 
			\sum_{s_k=1}^{d_{p_k}}
			\sum_{r=1}^{d_{\varepsilon_1 + \cdots + \varepsilon_k + |\mu|}} 
			(-1)^\varrho
			\alpha_{\gamma_k(\mu),\varepsilon_1s_1,\ldots,\varepsilon_ks_k}^{\varepsilon_1 + \cdots + \varepsilon_k + |\mu|,r} 
			a_{(\varepsilon_1 + \cdots + \varepsilon_k + |\mu|)r}
			\otimes
			x_{s_1i_1}^{(p_1-\varepsilon_1)}
			\cdots 
			~x_{s_ki_k}^{(p_k-\varepsilon_k)} \\
			&=& 
			\sum_{\lambda = |\mu|}^{p_1+\cdots+p_k+|\mu|}
			\sum_{\substack{\varepsilon_1 + \cdots + \varepsilon_k = \lambda - |\mu|,\\
			0 \le \varepsilon_i \le p_i,~
			1 \le i \le k.}} 
			~\sum_{s_1=1}^{d_{p_1}} \cdots 
			\sum_{s_k=1}^{d_{p_k}}
			\sum_{r=1}^{d_{\lambda}} 
			(-1)^\varrho
			\alpha_{\gamma_k(\mu),\varepsilon_1s_1,\ldots,\varepsilon_ks_k}^{\lambda, r}
			a_{\lambda r}
			\otimes 
			x_{s_1i_1}^{(p_1-\varepsilon_1)}
			\cdots 
			~x_{s_ki_k}^{(p_k-\varepsilon_k)} \\
			&=& 
			\sum_{\lambda=|\mu|}^{p_1+\cdots+p_k+|\mu|}
			\sum_{r=1}^{d_{\lambda}}
			a_{\lambda r}~
			\otimes
			\sum_{\substack{\varepsilon_1 + \cdots		+\varepsilon_k = \lambda - |\mu|,\\
			0 \le \varepsilon_i \le p_i,~
			1 \le i \le k.}} 
			~\sum_{s_1=1}^{d_{p_1}} \cdots 
			\sum_{s_k=1}^{d_{p_k}}
			(-1)^\varrho
			\alpha_{\gamma_k(\mu),\varepsilon_1s_1,\ldots,\varepsilon_ks_k}^{\lambda, r}
			x_{s_1i_1}^{(p_1-\varepsilon_1)}
			\cdots 
			~x_{s_ki_k}^{(p_k-\varepsilon_k)} \\
			&\stackrel{\eqref{graded universal poly}}{=}&
			\sum_{\lambda = |\mu|}^{\theta}
			\sum_{r=1}^{d_{\lambda}}
			a_{\lambda r}
			~\otimes~
			\sum_{u=1}^{d_{\theta}} \alpha_{\gamma_k(\mu),p_1i_1,\ldots,p_ki_k}^{\theta,u} x_{ru}^{(\theta-\lambda)}   
			\\
			&=& 
			\sum_{\lambda = 0}^{\theta}
			\sum_{r=1}^{d_{\lambda}}
			a_{\lambda r}
			\otimes 
			\sum_{u=1}^{d_{\theta}} \alpha_{\gamma_k(\mu),p_1i_1,\ldots,p_ki_k}^{\theta,u} x_{ru}^{(\theta-\lambda)}
			\quad 
			\Big(\sum_{u=1}^{d_{\theta}} \alpha^{\theta,u}_{\gamma_k(\mu),p_1i_1,\ldots,p_ki_k} ~x^{( \theta-\lambda)}_{ru}=0,~ \mbox{for any }~ 0\leq \lambda \leq |\mu|-1 \Big)\\
			&=& 
			\sum_{u = 1}^{d_{\theta}} \alpha_{\gamma_k(\mu),p_1i_1,\ldots,p_ki_k}^{\theta,u}
			\Big(
			\sum_{\lambda = 0}^{\theta}
			\sum_{r=1}^{d_{\lambda}}
			a_{\lambda r}
			\otimes 
			x_{ru}^{(\theta-\lambda)}
			\Big) \\
			&\stackrel{\eqref{def_eta_A}}{=}& 
			\sum_{u = 1}^{d_{\theta}} \alpha_{\gamma_k(\mu),p_1i_1,\ldots,p_ki_k}^{\theta,u}
			~\eta_{\mathfrak{A}}\big(a_{\theta u}\big) \qquad 
			\\
			&=& 
			\eta_\mathfrak{A}
			\Big(\sum_{u = 1}^{d_{\theta}} \alpha_{\gamma_k(\mu),p_1i_1,\ldots,p_ki_k}^{\theta,u} ~a_{\theta u}\Big) 
			~\stackrel{\eqref{struc_const_graded}}{=}~
			\eta_\mathfrak{A} \big(\gamma_k(\mu)(a_{p_1i_1},\ldots,a_{p_ki_k})\big).
		\end{eqnarray*}
		Hence, the map $\eta_\mathfrak{A}:\mathfrak{A} \to \mathfrak{A} \otimes \mathcal{C}(\mathfrak{A})$ is a $\p$-algebra morphism. 
	\end{proof}
	
	Next, in the proposition below we show that the graded algebra $\mathcal{C}(\mathfrak{A})$ satisfies the following universal property. Thus, making $\mathcal{C}(\mathfrak{A})$ the \textbf{universal graded algebra of $\mathfrak{A}$}.
	
	\begin{prop}\label{Thm_univ_prop_graded_algebras}
		For any graded commutative algebra $C$ and a $\mathcal{P}$-algebra morphism $f: \mathfrak{A} \to \mathfrak{A} \otimes C$, there exists a unique graded commutative algebra morphism $\Phi: \mathcal{C}(\mathfrak{A}) \to C$ such that $f = (\mathrm{id} \otimes \Phi) \circ \eta_{\mathfrak{A}}$. In other words, the following diagram commutes. 
		\begin{equation}\label{Com-diag-graded-case}
			\begin{tikzcd}
				{\mathfrak{A}} & {\mathfrak{A} \otimes \mathcal{C}(\mathfrak{A})} \\
				& {\mathfrak{A} \otimes C.}
				\arrow["{\eta_\mathfrak{A}}", from=1-1, to=1-2]
				\arrow["f"', from=1-1, to=2-2]
				\arrow["{\mathrm{id} \otimes \Phi}", from=1-2, to=2-2]
			\end{tikzcd}
		\end{equation}
	\end{prop}

	\begin{proof}		
		For any homogeneous element $a_{pi} \in \mathfrak{A}$, there exists a finite subset $\cup_{0 \le \varepsilon \le p}\big\{c^{(p-\varepsilon)}_{si} \in C_{p-\varepsilon}~\big|~ s = 1 ,\ldots,d_{\varepsilon}\big\} \subset C$ such that the following equation holds: 
		\begin{equation}
		f(a_{p i}) =
		\sum_{\varepsilon =0}^p
		\sum_{s=1}^{d_\varepsilon} 
		a_{\varepsilon s} 
		\otimes 
		c^{(p - \varepsilon )}_{si}.
		\end{equation}
		Now, since the map $f:\mathfrak{A}\rightarrow \mathfrak{A}\otimes C$ is a graded $\mathcal{P}$-algebra morphism, for any homogeneous elements $a_{p_1i_1},\ldots,a_{p_ki_k} \in \mathfrak{A}$, and $\mu \in \mathcal{P}(k)$ we obtain,
		\begin{equation*}
			f\big(\gamma_k(\mu)(a_{p_1i_1},\ldots,a_{p_ki_k})\big) ~=~ \gamma_k(\mu)\big(f(a_{p_1i_1}),\ldots,f(a_{p_ki_k})\big).
		\end{equation*} 
		Evaluating the above equation we get 
		\begin{equation}\label{relations_graded_on_csi}
			\sum_{u=1}^{d_\theta} \alpha^{\theta,u}_{\gamma_k(\mu),p_1i_1,\ldots,p_ki_k} c^{(\theta - \lambda)}_{au}
		\hspace{0.2cm} =
		\sum_{\substack{\varepsilon_1+\cdots+\varepsilon_k=\lambda-|\mu|, \\ 0\leq \varepsilon_i \le p_i,~1\leq i\leq k.}}~
		\sum_{s_1=1}^{d_{p_1}}\cdots
		\sum_{s_k=1}^{d_{p_k}} 	
		(-1)^\varrho \alpha^{\lambda,a }_{\gamma_k(\mu),\varepsilon_1s_1,\ldots,\varepsilon_ks_k}c_{s_1 i_1}^{(p_1-\varepsilon_1)}\cdots c_{s_k i_k}^{(p_k-\varepsilon_k)},
		\end{equation}
		where $\theta=p_1 + \cdots + p_k+|\mu|$, $0 \le \lambda \le \theta$, and $a=1,\ldots,d_\theta$. Subsequently, the association $x^{(\pi)}_{si}\mapsto c^{(\pi)}_{si}$ induces a unique graded commutative algebra morphism ${\Phi}:\mathcal{C}(\mathfrak{A})\rightarrow C$ for which Diagram \eqref{Com-diag-graded-case} commutes. Thereby, the proof is complete. 
	\end{proof}

 We now prove the Theorem \ref{thm_left_adjoint} in the graded context. 
	 
	\begin{theorem}\label{thm_graded_left_adjoint}
		Let $\p$ be a graded operad and $\mathfrak{A}=\oplus_{p \ge 0}\mathfrak{A}_p \in \p$-algebra. The functor $\mathfrak{A}\otimes -: \mathrm{g}$-$\mathrm{ComAlg} \to \mathcal{P}$-Alg admits a left adjoint if and only if $\mathfrak{A}_p$ is finite-dimensional for each $p \ge 0$.
	\end{theorem}

	\begin{proof}
Let $\p$ be a graded operad and $\mathfrak{A}=\oplus_{p \ge 0}\mathfrak{A}_p \in \p$-algebra such that $\mathfrak{A}_p$ is finite-dimensional for each $p \ge 0$.	Similar to the construction of the universal graded algebra $\mathcal{C}(\mathfrak{A})$, one can construct a functor $\mathcal{C}(\mathfrak{A},-):\mathcal{P}$-Alg $\to \mathrm{g}$-$\mathrm{ComAlg}$.

Conversely, assume that the functor $\mathfrak{A}\otimes -: \mathrm{g}$-$\mathrm{ComAlg} \to \mathcal{P}$-Alg admits a left adjoint. Let us first observe that the product of a family of graded vector spaces $\{\mathfrak{A}^i,~i\in I\}$ is given by a graded vector space $\prod_{i \in I} \mathfrak{A}^i$ with the homogeneous component of degree $p$ defined as $$\big(\prod_{i \in I} \mathfrak{A}^i\big)_p:=\{(a^i)_{i \in I} : |a^i|=p \mbox{ in } \mathfrak{A}^i, \text{ for all } i \in I\}.$$
Furthermore, it easily follows that in the category $\p$-Alg for a graded operad $\p$, the products are same as the products of the underlying graded vector spaces. Since the functor $\fa \otimes-$ admits a left adjoint,  it preserves arbitrary products in $\mathrm{g}$-$\mathrm{ComAlg}$.  Let us consider $\mathbb{K}$ as a graded commutative algebra concentrated in degree zero. 
Then by assumption $\mathfrak{A} \otimes -$ preserves the product of countable copies of $\K$ in $\mathrm{g}$-$\mathrm{ComAlg}$. By the definition of the product of graded vector spaces, we have 
$$\big( \mathfrak{A}\otimes \prod_{i \geq 0}\mathbb{K}^i\big)_p= \prod_{i \geq 0} (\mathfrak{A}_p\otimes \mathbb{K}^i). $$  
Hence, $\mathfrak{A} \otimes -$ preserves the product of countable copies of $\K$ in $\mathrm{g}$-$\mathrm{ComAlg}$ only if $\mathfrak{A}$ is finite-dimensional in each homogeneous component. 
	\end{proof}
	

\medskip
\section{\large An Enrichment of the category $\p$-$\mathrm{FinAlg}$ over the dual category of $\mathrm{ComAlg}$} \label{sec-4}
	
	In this section, we prove that the category $\p$-FinAlg of all finite-dimensional $\p$-algebras is enriched over the monoidal category $(\mathrm{ComAlg}^{\mathrm{op}},\otimes, \mathbb{K},\alpha,l,r)$, where $\mathrm{ComAlg}^{\mathrm{op}}$ the dual category of $\mathrm{ComAlg}$. Subsequently, we show that $\mathcal{C}(\mathfrak{a})$ possesses a unique bialgebra structure, making it the initial object in the category of all the commutative bialgebras coacting on the $\p$-algebra $\mathfrak{a}$.
	Before proceeding, we recall the notion of enriched category.
	\begin{definition}\cite[Section 1.2, pg 8]{Kelly}
		Let $(\mathcal{M},\otimes, I, \alpha,l,r)$ be a monoidal category. 
		Then an $\mathcal{M}$-category (or a category enriched over $\mathcal{M}$) is a category $\mathcal{C}$ consisting of
		\begin{itemize}
			\item a class $ob(\mathcal{C})$ of objects, 
			\item a hom-object $\mathcal{M}(\fa,\mathfrak{b}) \in \mathcal{M}$ for each pair of objects $\fa,\mathfrak{b}$ in $\mathcal{C}$,
			\item a composition morphism  $\circ_{\fa,\mathfrak{b},\mathfrak{c}}: \mathcal{M}(\mathfrak{b},\mathfrak{c}) \otimes \mathcal{M}(\mathfrak{a},\mathfrak{b}) \to \mathcal{M}(\mathfrak{a},\mathfrak{c})$ in $\mathcal{M}$ for each triple of objects $\fa,\mathfrak{b},\mathfrak{c} \in \mathcal{C}$. Additionally, for any $\fa,\mathfrak{b},\mathfrak{c},\mathfrak{d} \in \mathcal{C}$ the composition morphisms should the associativity condition as described by the commutative diagram below:
			\begin{equation}\label{diag4.1}
				\begin{tikzcd}
					& 	{\mathcal{M}(\mathfrak{a},\mathfrak{d})} \\
					{\mathcal{M}(\mathfrak{b},\mathfrak{d}) 	\otimes \mathcal{M}(\mathfrak{a},\mathfrak{b})} && {\mathcal{M}(\mathfrak{c},\mathfrak{d}) \otimes \mathcal{M}(\mathfrak{a},\mathfrak{c})} \\
					{\big(\mathcal{M}(\mathfrak{c},\mathfrak{d})\otimes\mathcal{M}(\mathfrak{b},\mathfrak{c})\big)\otimes \mathcal{M}(\mathfrak{a},\mathfrak{b})} && {\mathcal{M}(\mathfrak{c},\mathfrak{d}) \otimes \big(\mathcal{M}(\mathfrak{b},\mathfrak{c}) \otimes \mathcal{M}(\mathfrak{a},\mathfrak{b})\big).}
					\arrow["{{\circ_{\mathfrak{a},\mathfrak{b},\mathfrak{d}}}}"', from=2-1, to=1-2]
					\arrow["{{\circ_{\mathfrak{a},\mathfrak{c},\mathfrak{d}}}}", from=2-3, to=1-2]
					\arrow["{{\circ_{\mathfrak{b},\mathfrak{c},\mathfrak{d}} \otimes \mathrm{id}}}"', from=3-1, to=2-1]
					\arrow["\alpha", from=3-1, to=3-3]
					\arrow["{{\mathrm{id} \otimes \circ_{\mathfrak{a},\mathfrak{b},\mathfrak{c}}}}", from=3-3, to=2-3]
				\end{tikzcd} 
			\end{equation}  
				
			\item an identity morphism $\varepsilon_\fa:I \to \mathcal{M}(\fa,\fa)$ for each object $\fa \in \mathcal{C}$ there exists such that the following diagrams commute:
			\begin{equation}\label{diag4.2}
					\begin{tikzcd}
					{\mathcal{M}(\mathfrak{a},\mathfrak{b})} & {\mathcal{M}(\mathfrak{a},\mathfrak{b}) \otimes \mathcal{M}(\mathfrak{a},\mathfrak{a})} && {\mathcal{M}(\mathfrak{a},\mathfrak{b})} & {\mathcal{M}(\mathfrak{b},\mathfrak{b}) \otimes \mathcal{M}(\mathfrak{a},\mathfrak{b})} \\
					& 	{\mathcal{M}(\mathfrak{a},\mathfrak{b}) \otimes \mathbb{K}} &&& {\mathbb{K}\otimes \mathcal{M}(\mathfrak{a},\mathfrak{b}).}
					\arrow["{{\circ_{\mathfrak{a},\mathfrak{a},\mathfrak{b}}}}"', from=1-2, to=1-1]
					\arrow["{{\circ_{\mathfrak{a},\mathfrak{b},\mathfrak{b}}}}"', from=1-5, to=1-4]
					\arrow["r", from=2-2, to=1-1]
					\arrow["{{\mathrm{id} \otimes 	\varepsilon_\mathfrak{a} }}"', from=2-2, to=1-2]
					\arrow["l", from=2-5, to=1-4]
					\arrow["{{\varepsilon_\mathfrak{b} 	\otimes  \mathrm{id} }}"', from=2-5, to=1-5]
				\end{tikzcd}
			\end{equation}
		\end{itemize}  
	\end{definition}
		
	\begin{theorem}
		Let $\mathrm{ComAlg}^{\mathrm{op}}$ be the dual category of $\mathrm{ComAlg}$. The category of all finite-dimensional $\p$-algebras is enriched over the monoidal category $(\mathrm{ComAlg}^{\mathrm{op}},\otimes, \mathbb{K},\alpha,l,r)$.
	\end{theorem}

	\begin{proof}
		In order to show that the category $\p$-FinAlg is enriched over the monoidal category $(\mathrm{ComAlg}^{\mathrm{op}},\otimes, \mathbb{K},\alpha,l,r)$, we need to find a hom-object $\mathcal{M}(\mathfrak{a},\mathfrak{b})$ in $\mathrm{ComAlg}^{\mathrm{op}}$ for any $\mathfrak{a},\mathfrak{b}\in \p$-FinAlg such that 
		\begin{itemize}
			\item[(1)]  morphisms $f: \mathfrak{a} \to \mathfrak{b}$ in the enriched category structure on $\p$-FinAlg are given by morphisms $f: \K \to \mathcal{M}(\mathfrak{a},\mathfrak{b})$ in  $\mathrm{ComAlg}^{\mathrm{op}}$.
		
			\item[(2)] for each object $\mathfrak{a}$ in $\p$-FinAlg, we must have a morphism $\varepsilon_\mathfrak{a}: \K \to \mathcal{M}(\mathfrak{a},\mathfrak{a})$ in $\mathrm{ComAlg}^{\mathrm{op}}$ which is called the identity morphism.
		
			\item[(3)] for any object $\fa,\mathfrak{b},\mathfrak{c}$ there exists a  composition morphism $\circ_{\mathfrak{a},\mathfrak{b},\mathfrak{c}}: \mathcal{M}(\mathfrak{b},\mathfrak{c}) \otimes \mathcal{M}(\mathfrak{a},\mathfrak{b})\to \mathcal{M}(\mathfrak{a},\mathfrak{c})$ in $\mathrm{ComAlg}^{\mathrm{op}}$ for $\mathfrak{a},\mathfrak{b},\mathfrak{c}\in \p$-FinAlg for which Diagram \eqref{diag4.1} commutes. 
			Additionally, for $\mathfrak{a},\mathfrak{b}\in \p$-FinAlg the composition morphisms $\circ_{\mathfrak{a},\mathfrak{a},\mathfrak{b}},\circ_{\mathfrak{a},\mathfrak{b},\mathfrak{b}}$ and the identity morphisms $\varepsilon_{\mathfrak{a}},\varepsilon_{\mathfrak{b}}$  makes Diagram \eqref{diag4.2} commutative.
 		\end{itemize}	
 		
 		\medskip
		Now, we step-by-step show that the universal construction done in Section \ref{sec-3} leads to an enrichment of $\p$-FinAlg. First, for any two objects $\mathfrak{a}$ and $\mathfrak{b}$ in $\p$-FinAlg, we fix the required object $\mathcal{M}(\mathfrak{a},\mathfrak{b})$ to be the universal object $\mathcal{C}(\mathfrak{b},\mathfrak{a})$ constructed in the proof of Theorem \ref{thm_left_adjoint}. Then, the morphism set from $\mathfrak{a}$ to $\mathfrak{b}$ in the enriched category structure on $\p$-FinAlg is defined to be $\mathrm{Hom}_{\mathrm{ComAlg^{\mathrm{op}}}}(\K,\mathcal{C}(\mathfrak{b},\mathfrak{a}))$. Note that using the universal property of $\mathcal{C}(\mathfrak{a},\mathfrak{b})$ obtained from the  proof of Theorem \ref{thm_left_adjoint} we get a bijective correspondence
		\begin{equation*}
			\mathrm{Hom}_{\p\text{-}\mathrm{FinAlg}}(\mathfrak{a},\mathfrak{b})  \longleftrightarrow \mathrm{Hom}_{\mathrm{ComAlg^{\mathrm{op}}}}\big(\K,\mathcal{C}(\mathfrak{b},\mathfrak{a})\big).
		\end{equation*}
		
		Next, for any object $\fa \in \p$-FinAlg, the $\p$-algebra morphism 
		$\eta_{\fa,\fa}:\fa \to \fa \otimes \mathcal{C}(\fa)$ (cf. Equation \eqref{thm_def_eta_b}) uniquely determines an algebra morphism $\varepsilon_\fa: \mathcal{C}(\fa) \to \K$ making the following diagram commutative 
		\begin{equation}\label{diag_coidentity}
			\begin{tikzcd}
				{\mathfrak{a}} & {\fa \otimes \mathcal{C}(\fa)} \\
				& {\fa \otimes \mathbb{K}}
				\arrow["{\eta_{\mathfrak{a},\fa}}", from=1-1, to=1-2]
				\arrow["\mathrm{i}"', from=1-1, to=2-2]
				\arrow["{id \otimes \varepsilon_{\fa}}", from=1-2, to=2-2]
			\end{tikzcd}
		\end{equation} 
		where $\mathrm{i}:\fa \to \fa \otimes \mathbb{K}$ is the canonical isomorphism from $\mathfrak{a}$ to $\fa \otimes \mathbb{K}$.  
		Evaluating the above diagram we get 
		\begin{equation}\label{def_epsilon_id}
			\varepsilon_\fa(x_{st})=\delta_{s,t} \quad \text{for all }s,t=1,2,\ldots,\mathrm{dim}~\fa.
		\end{equation}
		Thus, for any object $\fa$ in $\p$-FinAlg there exists a morphism $\varepsilon_\fa \in \mathrm{Hom}_{\mathrm{ComAlg}^{\mathrm{op}}}\big(\mathbb{K},\mathcal{C}(\mathfrak{a})\big)$ which serves as the identity morphism for $\fa$ in the enriched category.
		
		\medskip
		Finally, for any objects $\fa$, $\mathfrak{b}$, and $\mathfrak{c}$ in $\p$-FinAlg, 
		consider the linear map $f_\mathfrak{c}:\mathfrak{c}\rightarrow \fa\otimes \mathcal{C}(\fa,\mathfrak{b})\otimes \mathcal{C}(\mathfrak{b},\mathfrak{c})$ defined by
		\begin{equation}
			f_\mathfrak{c}:\mathfrak{c}
			\xrightarrow{\eta_{\mathfrak{c},\mathfrak{b}}}\mathfrak{b}\otimes \mathcal{C}(\mathfrak{b},\mathfrak{c}) \xrightarrow{\eta_{\mathfrak{b},\fa}\otimes \mathrm{id}} \fa\otimes \mathcal{C}(\fa,\mathfrak{b})\otimes \mathcal{C}(\mathfrak{b},\mathfrak{c}). 
		\end{equation}
		One can verify that $f_\mathfrak{c}$ is a $\p$-algebra morphism. Using the universal property of $\eta_{\mathfrak{c},\mathfrak{a}}: \mathfrak{c} \to \mathfrak{a} \otimes \mathcal{C}(\mathfrak{a},\mathfrak{c})$ we get an algebra morphism $\Delta_{\mathfrak{a},\mathfrak{b},\mathfrak{c}}: \mathcal{C}(\mathfrak{a}, \mathfrak{c}) \to \mathcal{C}(\fa,\mathfrak{b})\otimes \mathcal{C}(\mathfrak{b},\mathfrak{c})$ for which the following diagram is commutative
		\begin{equation}\label{diag_coen_id2}
			\begin{tikzcd}
				{\mathfrak{c}} & {\mathfrak{a} \otimes \mathcal{C}(\mathfrak{a},\mathfrak{c})} \\
				& {\mathfrak{a} 
					\otimes \mathcal{C}(\mathfrak{a},\mathfrak{b}) \otimes \mathcal{C}(\mathfrak{b},\mathfrak{c}).}
				\arrow["{\eta_{\mathfrak{c},\mathfrak{a}}}", from=1-1, to=1-2]
				\arrow["{f_\mathfrak{c}}"', from=1-1, to=2-2]
				\arrow["{\mathrm{id} \otimes \Delta_{\mathfrak{a},\mathfrak{b},\mathfrak{c}}}", from=1-2, to=2-2]
			\end{tikzcd}
		\end{equation}
		Let $x_{ij},~x_{ij}^\prime,$ and $x_{ij}^{\prime\prime}$ denote the equivalence classes in the algebras $\mathcal{C}(\mathfrak{a},\mathfrak{c}),$ $\mathcal{C}(\mathfrak{a},\mathfrak{b}),$ and $\mathcal{C}(\mathfrak{b},\mathfrak{c}),$ respectively. Then, evaluating the above diagram, the map $\Delta_{\fa,\mathfrak{b},\mathfrak{c}}$ can be described by the following expression:
		\begin{equation}\label{def_delta_abc}
			\Delta_{\fa,\mathfrak{b},\mathfrak{c}} (x_{st}) ~=~ \sum_{i=1}^{m} x_{si}^\prime \otimes x_{it}^{\prime\prime}, 
		\end{equation}
		for any $s=1,\ldots,\mathrm{dim}~\fa$ and $t = 1,\ldots, \mathrm{dim}~\mathfrak{c}$, where $m=\mathrm{dim}~\mathfrak{b}$. 
		Using Equation \eqref{def_delta_abc} we get, for any objects $\mathfrak{a}$, $\mathfrak{b}$, $\mathfrak{c}$, $\mathfrak{d}$ in $\p$-FinAlg the Diagram \eqref{diag4.1} is commutative, where we define the composition $\circ_{a,b,c}:=\Delta_{c,b,a}^{op}$. Similarly, using Equation \eqref{def_epsilon_id} and \eqref{def_delta_abc} one can also verify that Diagram \eqref{diag4.2} commutes. Therefore, the category $\p$-FinAlg is enriched over the  monoidal category $(\mathrm{ComAlg}^{\mathrm{op}}, \otimes, \mathbb{K},\alpha,l,r)$.
	\end{proof}

	\begin{theorem}
		Let $\fa$ be an object in $\p$-FinAlg. Then there exists a unique bialgebra structure on $\mathcal{C}(\fa)$ such that the $\p$-algebra morphism $\eta_\fa : \fa \to \fa \otimes \mathcal{C}(\fa)$ defines a right $\mathcal{C}(\fa)$-comodule structure on $\fa$.
	\end{theorem}
	
	\begin{proof}
		Taking $\fa=\mathfrak{b}=\mathfrak{c}$ in Diagram \eqref{diag_coen_id2}, we get a unique algebra morphism $\Delta_{\fa}:\mathcal{C}(\fa)\rightarrow\mathcal{C}(\fa)\otimes\mathcal{C}(\fa)$ defined as 
		\begin{equation}
			\Delta_\fa(x_{st})=\sum_{i=1}^n x_{si}\otimes x_{it} \quad \text{for all}~ s,t=1,2,\ldots, n.
		\end{equation}
		Using it we get that the following diagram is commutative 
		\begin{equation}\label{comodule action 1}
			\begin{CD}
				\fa @>\eta_\fa>> \fa\otimes \mathcal{C}(\fa)\\
				@V \eta_\fa VV @V \mathrm{id}\otimes \Delta_\fa VV\\
				\fa\otimes \mathcal{C}(\fa)@>\eta_\fa\otimes \mathrm{id}>> \fa\otimes \mathcal{C}(\fa)\otimes\mathcal{C}(\fa).
			\end{CD}
		\end{equation}
		Furthermore, note that Diagram \eqref{diag_coidentity} guarantees the existence of coidentity map 
		$\varepsilon_\fa:\mathcal{C}(\fa)\rightarrow \mathbb{K}$ given by
		\begin{equation}
			\varepsilon_\fa(x_{st})=\delta_{s,t} \quad \text{for all }s,t=1,2,\ldots,n.
		\end{equation}
		Thus, making $\mathcal{C}(\fa)$ a bialgebra with coproduct $\Delta_\fa$ and counit $\varepsilon_\fa$. Also note that, Diagram \eqref{comodule action 1} implies that $\eta_\fa:\fa\rightarrow \fa\otimes \mathcal{C}(\fa)$ defines a right $\mathcal{C}(\fa)$-comodule structure on $\fa$. 
	\end{proof}

	\begin{remark}\label{universal bialgebra const}
		The universal property of $\eta_\fa: \fa \to \fa \otimes \mathcal{C}(\fa)$ stated in Corollary \ref{cor_univ_prop} extends to bialgebras in the following way: For any commutative bialgebra $C$ with $C$-comodule action $\psi:\fa\rightarrow \fa\otimes C,$ there exists a unique bialgebra homomorphism $\Phi:\mathcal{C}(\fa)\rightarrow C$ such that the following diagram commutes
 		\begin{equation}\label{universal bialgebra prop}
 			\begin{tikzcd}
 				\fa & {\fa \otimes \mathcal{C}(\fa)} \\
 				& C.
 				\arrow["{\eta_\fa}", from=1-1, to=1-2]
 				\arrow["\psi"', from=1-1, to=2-2]
 				\arrow["{\mathrm{id} \otimes \Phi}", from=1-2, to=2-2]
 			\end{tikzcd} 
		\end{equation}
	\end{remark}

	We recall from \cite{Takeuchi} that, one can construct a free commutative Hopf algebra from a commutative bialgebra. The construction defines a functor $L: \operatorname{ComBiAlg} \to \operatorname{ComHopf}$ that is left adjoint to the forgetful functor $U:\operatorname{ComHopf} \to \operatorname{ComBiAlg}$. We denote by $\mu:\mathds{1}_{\operatorname{ComBiAlg}} \to UL,$ the unit of the adjunction $L \dashv U$. More precisely, for any $\p$-algebra $\fa$, the associated Hopf algebra $\mathcal{H}(\fa)$ and the comodule action $\lambda_\fa:\fa\rightarrow\fa\otimes \mathcal{H}(\fa)$ is defined as follows: 
	$$\mathcal{H}(\fa):=L(\mathcal{C}(\fa)),\quad\text{and}\quad\lambda_\fa:=(\mathrm{id}\otimes \mu_{\mathcal{C}(\fa)})\circ \eta_\fa.$$
	Furthermore, there exists an initial object $(\mathcal{H}(\fa),\lambda_\fa)$ in the category of all commutative Hopf algebras coacting on the finite-dimensional $\p$-algebra $\fa$. With the above discussion in mind, we state our next result.

	\begin{theorem}
		Let $H$ be a commutative Hopf algebra with right $H$-comodule action $\psi:\fa\rightarrow \fa\otimes H,$ then there exists a unique Hopf algebra homomorphism $\Phi:\mathcal{H}(\fa)\rightarrow H$ such that 
 		\begin{equation}\label{universal hopf prop}
 			\xymatrix{
			\mathfrak{a} \ar[r]^-{\lambda_\mathfrak{a}}\ar[rd]_\psi 
			&\fa\otimes \mathcal{H}(\fa)\ar[d]^{\mathrm{id}\otimes \Phi}\\
			&\fa\otimes H} 
		\end{equation}
	\end{theorem}

	\begin{proof}
		The proof is straightforward and similar to the proof given in \cite[Theorem 2.13]{Ag-Mil}.
	\end{proof}


\medskip
\section{\large Examples of universal construction for algebras over (Graded) Operads}\label{sec-5}
	
	In this section, we discuss examples of the universal constructions (introduced in Section \ref{sec-3}) for algebras over some binary quadratic operads, $k$-ary quadratic operads, non-symmetric operads, and graded operads. Furthermore, we show that the universal coacting bi/Hopf algebra associated with a finite-dimensional $\p$-algebra coincides the corresponding universal coacting algebra for Lie/Leibniz algebras \cite{Ag-Mil}, Poisson algebras \cite{Ag-Mil2}, and associative algebras \cite{Mil} if $\p$ is $\mathcal{L}ie$, $\mathcal{P}ois$, and $\mathcal{A}ss$, respectively. Moreover, we highlight the universal constructions for $k$-ary algebras and graded algebras using our operadic approach.
\bigskip
\subsection{Algebras over binary quadratic operads}~\\

	\vspace{-0.4cm}
	Here we will discuss some examples of binary quadratic operads and describe the universal algebra and universal polynomials for any finite-dimensional algebra governed by these operads. Specifically, we describe the universal algebras and universal polynomials for finite-dimensional commutative algebras, associative algebras, Leibniz algebras, Lie algebras, Zinbiel algebras, Poisson algebras, pre-Lie algebras, and Perm algebras. Before proceeding further, we fix some notation. We denote the basis elements $\mathrm{id}_{\Sm_2}$ and $(12)$ of the vector space $\K[\Sm_2]$ by $\mu$ and $\mu^{(12)}$, respectively.  
	
	\medskip
	We recall from Section \ref{sec-3} that for any finite-dimensional algebra $\mathfrak{a}$, over an operad $\mathcal{P}$, its universal algebra $\mathcal{C}(\mathfrak{a})$ is defined as
	\begin{equation}\label{section5_def_universal_algebra}
		\mathcal{C}(\fa) := \mathbb{K}\big[X_{si}|~s,i=1,2,\ldots,n\big] \big/ J,
	\end{equation}
	where $J$ is the ideal generated by the universal polynomials of $\fa$
	\begin{equation}\label{def_universal_poly_a}
		P^{(\fa)}_{(\mu,a,i_1,\ldots,i_k)} = \sum_{u=1}^n \alpha^u_{\gamma_k(\mu),i_1,\ldots,i_k} X_{au}
		\hspace{0.2cm} - \sum_{s_1,s_2,\ldots, s_k=1}^n \alpha^a_{\gamma_k(\mu),s_1,\ldots,s_k}X_{s_1 i_1}\cdots X_{s_k i_k},
	\end{equation} 
	for any $\mu\in \p(k)$ and $a,i_1,\ldots, i_k=1,\ldots,n$. If $\mathcal{P}$ is a binary quadratic operad, then we will show that the ideal $J \subseteq \mathcal{C}(\fa)$ is generated by the universal polynomials $P^{(\fa)}_{(\mu,a,i_1,\ldots,i_k)}$, where $\mu$ is a basis element of $\mathcal{P}(2)$.
	
	\medskip
	Any binary quadratic operad $\mathcal{P}$ can be described in terms of quadratic data (see \cite[Chapter 7]{Loday-Vallette}) as follows: 
	$$\mathcal{P}:= \T(E_\p)/(R_\p),$$
	where $\overline{E}_\p:=(0,0,E_\p,0,\ldots)$ is an $\mathbb{S}$-module concentrated in arity $2$ and $\T(E_\p)$ denotes the free operad over $\overline{E}_\p$ as described in Example \ref{EXAMPLE}.
	Here $(R_\p)$ is the operadic ideal generated by the $\Sm_3$-submodule $R_\p \subseteq \T(E_\p)(3)$, called the space of quadratic relations. 
	Now, let $\mathfrak{a}$ be any finite-dimensional $\mathcal{P}$-algebra with vector space basis $\{a_1,\ldots,a_n\}$. 
	Then, for any $\mu \in \mathcal{P}(k)$, $\nu \in \mathcal{P}(l)$, and $\sigma \in \Sm_{k+l-1}$ we have the following equations:
	\begin{equation}\label{eqn_freeOperad_struc_const_2}
		\gamma_k(\mu)(a_{r_1},\ldots,a_{r_k}) = \sum_{s=1}^n \alpha_{\gamma_k(\mu),r_1,\ldots,r_k}^sa_s, 
		\quad 
		\gamma_l(\nu)(a_{t_1},\ldots,a_{t_l}) = \sum_{s=1}^n \alpha_{\gamma_l(\nu),t_1,\ldots,t_l}^sa_s,
 	\end{equation} 
 	\begin{equation}\label{eqn_freeOperad_struct_const_gamma3}
 		\gamma_{k+l-1}\big((\mu \circ_1 \nu)^{\sigma}\big)(a_{i_1},\ldots,a_{i_{k+l-1}})
 		=
 		\sum_{s=1}^n \alpha^s_{\gamma_{k+l-1}(\mu \circ_1 \nu),i_{\sigma(1)},\ldots,i_{\sigma(k+l-1)}}~ a_s,
 	\end{equation}
	where $r_1,\ldots,r_k \in \{1,2,\ldots,n\}$, $t_1,\ldots,t_l \in \{1,2,\ldots,n\}$, and $i_1,\ldots,i_{k+l-1} \in \{1,2,\ldots,n\}$. 
	As a result, the structure constants of $\mathfrak{a}$ corresponding to $\mu \in \mathcal{P}(k)$, $\nu \in \mathcal{P}(l)$, and $(\mu \circ_1 \nu)^\sigma \in \mathcal{P}(k+l-1)$ are given by the following sets:  $\{\alpha^s_{\gamma_k(\mu),r_1,\ldots,r_k}~|~s,r_1,\ldots,r_k~=~1,\ldots,n\}$, $\{\alpha^s_{\gamma_l(\nu),t_1,\ldots,t_l}~|~s,t_1,\ldots,t_l~=~1,\ldots,n\}$, and   $\{\alpha^s_{\gamma_{k+l-1}(\mu \circ_1 \nu),i_{\sigma(1)},\ldots,i_{\sigma(k+l-1)}}~|~s,i_1,\ldots,i_{k+l-1}=1,\ldots,n\}$, respectively. The following lemma shows that these structure constants are related to each other. 

	\begin{lemma}\label{lemma_5.1}
		Let $\p$ be a binary quadratic operad and $\mathfrak{a}$ be a finite-dimensional $\p$-algebra. Then, for any $\mu \in \mathcal{P}(k)$ and $\nu \in \mathcal{P}(l)$, the structure constants of $\mathfrak{a}$ corresponding to $\mu$, $\nu$, and $(\mu \circ_1 \nu)^\sigma$ (as described in the discussion above) are related as follows:
		\begin{equation}\label{lem_freeOperad_struct_rel_gamma2_gamma3_gen}
			\alpha^s_{\gamma_{k+l-1}(\mu \circ_1 \nu),i_{\sigma(1)},\ldots,i_{\sigma(k+l-1)}} = \sum_{p=1}^n \alpha^p_{\gamma_l(\nu),i_{\sigma(1)},\ldots,i_{\sigma(l)}} \alpha^s_{\gamma_k(\mu),p,i_{\sigma(l+1)},\ldots,i_{\sigma(k+l-1)}},
		\end{equation}
		for any \(s = 1, 2, \ldots, n\). 
	\end{lemma} 
	
	\begin{proof}
		Expanding $\gamma_{k+l-1}\big((\mu \circ_1 \nu)^\sigma\big)(a_{i_1},\ldots,a_{i_{k+l-1}})$ we get
		\begin{eqnarray} 
			\label{eqn_freeOperd_struct_rel_gamma2_gamma3_gen}
			\gamma_{k+l-1}\big((\mu \circ_1 \nu)^\sigma\big)(a_{i_1},\ldots,a_{i_{k+l-1}}) 
			&=&
			\gamma_k(\mu)\big(\gamma_{l}(\nu)(a_{i_{{\sigma(1)}}},\ldots,a_{i_{\sigma(l)}}),a_{i_{\sigma(l+1)}},\ldots, a_{i_{\sigma(k+l-1)}}\big) \nonumber \\
			&\stackrel{(\ref{eqn_freeOperad_struc_const_2})}{=}& 
			\gamma_{k}(\mu)\Big(\sum_{p=1}^{n} \alpha^p_{\gamma_l(\nu),i_{\sigma(1)},\ldots,i_{\sigma(l)}}a_p, a_{i_{\sigma(l+1)}},\ldots,a_{i_{\sigma(k+l-1)}}\Big) \nonumber \\
			&=&
			\sum_{p=1}^{n} \alpha^p_{\gamma_l(\nu),i_{\sigma(1)},\ldots,i_{\sigma(l)}}
			\gamma_{k}(\mu)\big(a_p,a_{i_{\sigma(l+1)}},\ldots,a_{i_{\sigma(k+l-1)}}\big) \nonumber \\
			&\stackrel{(\ref{eqn_freeOperad_struc_const_2})}{=}&			
			\sum_{p=1}^{n} \alpha^p_{\gamma_l(\nu),i_{\sigma(1)},\ldots,i_{\sigma(l)}}
			\Big(\sum_{s=1}^n 	\alpha^s_{\gamma_{k}(\mu),p,i_{\sigma(l+1)},\ldots,i_{\sigma(k+l-1)}}~ a_s \Big)	\nonumber \\
			&=&
			\sum_{s=1}^n 
			\Big(\sum_{p=1}^n 	\alpha^p_{\gamma_l(\nu),i_{\sigma(1)},\ldots,i_{\sigma(l)}}
			~ 	\alpha^s_{\gamma_{k}(\mu),p,i_{\sigma(l+1)},\ldots,i_{\sigma(k+l-1)}} \Big)~ a_s
		\end{eqnarray}
		Thus, from Equation (\ref{eqn_freeOperad_struct_const_gamma3}) and Equation (\ref{eqn_freeOperd_struct_rel_gamma2_gamma3_gen}), we get 
		\begin{equation*}
			\alpha^s_{\gamma_{k+l-1}(\mu \circ_1 \nu),i_{\sigma(1)},\ldots,i_{\sigma(k+l-1)}}
			~=~ 
			\sum_{p=1}^n 	\alpha^p_{\gamma_l(\nu),i_{\sigma(1)},\ldots,i_{\sigma(l)}}	\alpha^s_{\gamma_{k}(\mu),p,i_{\sigma(l+1)},\ldots,i_{\sigma(k+l-1)}}
		\end{equation*}
		for any $s = 1,2,\ldots,n$. Hence, the proof is complete. 
	\end{proof}

	\begin{theorem}\label{thm_2_gen_rest_all}
		Let $\mathcal{P}= \T(E_\p)/(R_\p)$ be a binary quadratic operad. Let $\mathfrak{a}$ be a finite-dimensional $\mathcal{P}$-algebra and $\mathcal{C}(\mathfrak{a})$ be the associated universal algebra given by \eqref{section5_def_universal_algebra}. Then, the ideal $J $ is generated by the universal polynomials corresponding to the basis elements of $\mathcal{P}(2) = E_\p$. 
	\end{theorem}

	\begin{proof}
		Let $\mu \in \p(k),~\nu \in \p(l)$ and assume that the universal polynomials $P^\mathfrak{(a)}_{(\mu; a, r_1, \ldots, r_k)},~P^\mathfrak{(a)}_{(\nu; b, t_1, \ldots, t_l)} \in J$ for any $r_1,\ldots,r_k, t_1,\ldots,t_l =1,\ldots,n$. Therefore, using Equation \eqref{def_universal_poly_a}, we have
		\begin{equation} \label{freeOperad_eqn_univ_poly_rel_P2_gen1}
			\sum_{u=1}^n \alpha^u_{\gamma_k(\mu),r_1,\ldots,r_k} x_{au} 
			~=
			\sum_{s_1,\ldots,s_k=1}^n 
			\alpha^a_{\gamma_k(\mu),s_1,\ldots,s_k} x_{s_1r_1} \cdots x_{s_kr_k},
		\end{equation}
		\begin{equation} 
		\label{freeOperad_eqn_univ_poly_rel_P2_gen2}
			\sum_{u=1}^n \alpha^u_{\gamma_l(\nu),t_1,\ldots,t_l} x_{bu} 
			~=
			\sum_{s_1,\ldots,s_l=1}^n 
			\alpha^b_{\gamma_l(\nu),s_1,\ldots,s_l} x_{s_1t_1} \cdots x_{s_lt_l}, 
		\end{equation} 	
		where $1 \le a,r_1,\ldots, r_k, b,t_1,\ldots,t_l \le n$. For $\sigma \in \Sm_{k+l-1}$, using Equation \eqref{def_universal_poly_a} we obtain
 		\begin{multline*}
			P^\mathfrak{(a)}_{((\mu \circ_1 \nu)^\sigma;c,i_1,\ldots,i_{k+l-1})}
			~=~
			\overbrace{\sum_{u=1}^n \alpha^u_{\gamma_{k+l-1}(\mu \circ_1 \nu),i_{\sigma(1)},\ldots,i_{\sigma(k+l-1)}} x_{cu}}^{(A)} \\
			-
			\underbrace{\sum_{s_{1},\ldots,s_{k+l-1}=1}^n 
				\alpha^c_{\gamma_{k+l-1}(\mu\circ_1 \nu),s_1,\ldots,s_{k+l-1}} x_{s_1i_{\sigma(1)}} \cdots x_{s_{k+l-1}i_{\sigma(k+l-1)}}}_{(B)}, \qquad
		\end{multline*}
		where $1 \le c,i_1,\ldots,i_{k+l-1} \le n$. Expanding $(A)$ we get
		\begin{align*}
			\sum_{u=1}^n &\alpha^u_{\gamma_{k+l-1}(\mu \circ_1 \nu),i_{\sigma(1)},\ldots,i_{\sigma(k+l-1)}} x_{cu} \\
			&\stackrel{(\ref{lem_freeOperad_struct_rel_gamma2_gamma3_gen})}{=} 
			\sum_{u=1}^n \Big(\sum_{p=1}^n 	\alpha^p_{\gamma_l(\nu),i_{\sigma(1)},\ldots,i_{\sigma(l)}}
			\alpha^u_{\gamma_{k}(\mu),p,i_{\sigma(l+1)},\ldots,i_{\sigma(k+l-1)}}\Big)x_{cu} \\
			&\stackrel{(\ref{freeOperad_eqn_univ_poly_rel_P2_gen1})}{=} 
			\sum_{p=1}^n \alpha^p_{\gamma_l(\nu),i_{\sigma(1)},\ldots,i_{\sigma(l)}}
			\Big(\sum_{s_1,\ldots,s_k=1}^n 
			\alpha^c_{\gamma_k(\mu),s_1,\ldots,s_k} x_{s_1p}x_{s_2i_{\sigma(l+1)}}\ldots 
			x_{s_k i_{\sigma(k+l-1)}}\Big).
		\end{align*}
		Similarly, expanding $(B)$ we obtain the following expression. 
		\begin{eqnarray*}
			&~& \hspace{-2cm}
			\sum_{s_1,\ldots,s_{k+l-1}=1}^n 
			\alpha^c_{\gamma_{k+l-1}(\mu \circ_1 \nu),s_1,\ldots,s_{k+l-1}} x_{s_1i_{\sigma(1)}}\ldots x_{s_{k+l-1}i_{\sigma(k+l-1)}} \\
			&\stackrel{(\ref{lem_freeOperad_struct_rel_gamma2_gamma3_gen})}{=}&
			\sum_{s_1,\ldots,s_{k+l-1}=1}^n  \Big(\sum_{s=1}^n 	\alpha^s_{\gamma_l(\nu),s_1,\ldots,s_l}~
			\alpha^c_{\gamma_{k}(\mu),s,s_{l+1},\ldots,s_{k+l-1}}\Big)x_{s_1i_{\sigma(1)}}\ldots x_{s_{k+l-1}i_{\sigma(k+l-1)}} \\
			&\stackrel{\quad}{=}&\sum_{s_1,\ldots,s_{k+l-1}=1}^n \sum_{s=1}^n \alpha^s_{\gamma_l(\nu),s_1,\ldots,s_l} x_{s_1i_{\sigma(1)}}\ldots x_{s_li_{\sigma(l)}} ~\alpha^c_{\gamma_k(\mu),s,s_{l+1},\ldots,s_{k+l-1}}x_{s_{l+1}i_{\sigma(l+1)}} \ldots x_{s_{k+l-1}i_{\sigma(k+l-1)}}\\
			&\stackrel{(\ref{freeOperad_eqn_univ_poly_rel_P2_gen2})}{=}&
			\sum_{s=1}^n \Big(\sum_{p=1}^n\alpha^p_{\gamma_l(\nu),i_{\sigma(1)},\ldots,i_{\sigma(l)}}x_{sp}\Big)\Big(\sum_{s_{l+1},\ldots,s_{k+l-1}=1}^n\alpha^c_{\gamma_k(\mu),s,s_{l+1},\ldots,s_{k+l-1}}
			x_{s_{l+1}i_{\sigma(l+1)}} \cdots x_{s_{k+l-1}i_{\sigma(k+l-1)}}\Big)\\
			&=&
			\sum_{p=1}^n \alpha^p_{\gamma_l(\nu),i_{\sigma(1)},\ldots,i_{\sigma(l)}} 
			\Big(\sum_{s,s_{l+1},\ldots,s_{k+l-1}=1}^n
			\alpha^c_{\gamma_k(\mu),s,s_{l+1},\ldots,s_{k+l-1}}x_{sp}x_{s_{l+1}i_{\sigma(l+1)}} \cdots x_{s_{k+l-1}i_{\sigma(k+l-1)}}\Big).
		\end{eqnarray*} 
		As a result, we obtain 
		$$P^{(\mathfrak{a})}_{((\mu \circ_1 \nu)^\sigma;c,i_1,\ldots,i_{k+l-1})} = 0 ~ \text{ in } \mathcal{C}(\mathfrak{a}),$$ 
		for any $\sigma \in \Sm_{k+l-1}$ and $c,i_1,\ldots,i_{k+l-1} = 1,\ldots,n$. 
		
		\medskip
		Now, recall that any basis element of $\mathcal{P}(n)$ ($n \ge 3$) is of the form $(\mu \circ_1 \nu)^\sigma$, where $\mu$ and $\nu$ have arities less than $n$. Therefore, using mathematical induction on the arity of $\mathcal{P}$, we obtain the claim of the theorem.
	\end{proof}

	\begin{remark}\label{remark_5.3}
		We observe that Lemma \ref{lemma_5.1} and Theorem \ref{thm_2_gen_rest_all} extends naturally to the context of $k$-ary quadratic operads and graded (quadratic) operads. 
	\end{remark}

\medskip
\subsubsection{\it The operad $\mathcal{C}om$}~~\\

	\vspace{-0.4cm}
	Algebras over the operad $\mathcal{C}om$ are (non-unital) commutative, associative algebras. 
	The quadratic data given by $(E_{\mathcal{C}om}, R_{\mathcal{C}om})$ defines the operad $\mathcal{C}om := \T(E_{\mathcal{C}om})/(R_{\mathcal{C}om})$, where $E_{\mathcal{C}om}$ is the $\mathbb{S}$-module concentrated in arity $2$ with trivial $\Sm_2$-action
	\begin{equation*}
		E_{\mathcal{C}om} := (0, 0, \mathbb{K}\mu, 0, \ldots),
	\end{equation*}
	and $(R_{\mathcal{C}om})$ is the operadic ideal generated by $R_{\mathcal{C}om}$, which is a  $\mathbb{S}_3$-submodule of $\mathbb{T}(E_{\mathcal{C}om})(3)$ generated by the associator $\mu \circ_1 \mu - \mu \circ_2 \mu$.
	
	\medskip 
	Let $\mathfrak{c}$ be a finite-dimensional $\mathcal{C}om$-algebra. Also, let the set of scalars $\{\alpha^s_{i,j} ~|~ s,i,j = 1,\ldots,n\}$ denotes the structure constants of $\mathfrak{c}$ corresponding to $\mu$. Now, using Definition \ref{def_universal_algebra} and Theorem \ref{thm_2_gen_rest_all} we get that the universal algebra $\mathcal{C}(\mathfrak{c})$ of $\mathfrak{c}$ is given by:   
	\begin{equation*}
		\mathcal{C}(\mathfrak{c}):= \K[X_{si}:s,i=1,\ldots,n] \big/ J,
	\end{equation*}
 	where $J$ is the ideal generated by the universal polynomials of $\mathfrak{c}$ corresponding to the basis element $\mu \in \K\mu = \mathcal{C}om(2)$:
 	\begin{equation}
 		P^{(\mathfrak{c})}_{(a,i_1,i_2)} = \sum_{u=1}^n \alpha^u_{i_1, i_2} X_{au} - \sum_{s_1, s_2 = 1}^n \alpha^a_{s_1, s_2} X_{s_1 i_1} X_{s_2 i_2},
 	\end{equation}
 	for any $a, i_1, i_2 = 1, \ldots, n$. 
	
\bigskip	
\subsubsection{\it The operad $\mathcal{A}ss$}~~\\

	\vspace{-0.4cm}
	Algebras over the operad $\mathcal{A}ss$ are (non-unital) associative algebras. The quadratic data $(E_{\mathcal{A}ss},R_{\mathcal{A}ss})$, describing the operad $\mathcal{A}ss := \T(E_{\mathcal{A}ss})/(R_{\mathcal{A}ss})$, is defined as follows. $E_{\mathcal{A}ss}$ is the $\Sm$-module 
	$(0,0,\K[\Sm_2],0,\ldots)$
	and the space of relations $R_{\mathcal{A}ss}$ is the $\Sm_3$-submodule of $\mathbb{T}(E_{\mathcal{A}ss})(3)$ generated by the following element: 
	$$\mu \circ_1 \mu - \mu \circ_2 \mu.$$ 
	Note that, a $\K$-linear basis $\mathbb{T}(E_{\mathcal{A}ss})(3)$ is given by $\{(\mu \circ_1 \mu)^\sigma,(\mu \circ_2 \mu)^\sigma: \sigma \in \Sm_3\}$ and a $\K$-linear basis of $R_{\mathcal{A}ss}$ is given by $\{(\mu \circ_1 \mu - \mu \circ_2 \mu)^\sigma : \sigma \in \Sm_3\}$. For further details, see \cite[Section 7.6.4]{Loday-Vallette}. Now, taking the operad to be $\mathcal{A}ss$ in Theorem \ref{thm_left_adjoint} we obtain the following corollary.
	\begin{corollary}{\cite[Theorem 1.1]{Mil}}
		Let $A$ be a (non-unital) associative algebra. Then the functor $A \otimes - : \mathrm{ComAlg} \to \mathrm{Alg}$ admits a left adjoint if and only if $A$ is finite-dimensional. 
	\end{corollary}
	
	Furthermore, using Theorem \ref{thm_2_gen_rest_all} we obtain the following definition.
	\begin{definition}
		Let $A$ be a finite-dimensional associative algebra with structure constants $\{\alpha_{i,j}^s~|~ s,i,j = 1,\ldots,n\}$. Then, the universal polynomials of $A$ are of the form: 
		\begin{equation}
			P^{(A)}_{(a,i_1,i_2)} = \sum_{u=1}^n \alpha^u_{i_1, i_2} X_{au} - \sum_{s_1, s_2 = 1}^n \alpha^a_{s_1, s_2} X_{s_1 i_1} X_{s_2 i_2},
		\end{equation} 
		for any $a, i_1, i_2 = 1, \ldots, n$. The universal algebra $\mathcal{C}(A)$ of $A$ is given by 
		$\mathcal{C}(A) := \K[X_{si}~|~ s,i=1,\ldots,n] \big/J,$
		where $J$ is the ideal generated by the universal polynomials of $A$. 
	\end{definition}
	
	\begin{remark}
		Note that the universal algebra $\mathcal{C}(A)$ described above coincides with the ``quantum symmetry semigroup" $a(A)$ of $A$, which was introduced by G. Militaru in \cite[Definition 1.3]{Mil}.
	\end{remark}

\smallskip
\subsubsection{\it The operad $\mathcal{L}eib$}~~\\

	\vspace{-0.4cm}
	Any algebra over the operad $\mathcal{L}eib := \mathbb{T}(E_{\mathcal{L}eib})/(R_{\mathcal{L}eib})$ is a Leibniz algebras. The $\Sm$-module $E_{\mathcal{L}eib}$ is given by
	$E_{\mathcal{L}eib} := (0,0,\K[\Sm_2],0,\ldots)$
	and the space of relations $R_{\mathcal{L}eib}$ is the $\Sm_3$-submodule of $\mathbb{T}(E_{\mathcal{L}eib})(3)$ generated by 
	\begin{equation}\label{quad_rel_lbz}
		\mu \circ_1 \mu -\mu \circ_2 \mu - (\mu \circ_1 \mu)^{(23)}.
	\end{equation}
	Now, taking the operad $\mathcal{L}eib$ in Theorem \ref{thm_left_adjoint} we obtain the following corollary. 

	\begin{corollary}{\cite[Theorem 2.1]{Ag-Mil}}\label{cor_leib}
		Let $\mathfrak{h}$ be a Leibniz algebra. Then the functor $\mathfrak{h} \otimes - : \mathrm{ComAlg} \to \mathrm{Lbz}$ admits a left adjoint if and only if $\mathfrak{h}$ is finite-dimensional.  
	\end{corollary}
 	Furthermore, using Theorem \ref{thm_2_gen_rest_all} we obtain the following description of the universal algebra of a finite-dimensional Leibniz algebra. 
 	\begin{definition}
 		Let $\mathfrak{h}$ be a finite-dimensional Leibniz algebra with structure constants $\{\alpha_{i,j}^s~|~s,i,j=1,\ldots,n\}$. Then, universal algebra $\mathcal{C}(\mathfrak{h})$ of $\mathfrak{h}$ is given by $\mathcal{C}(\mathfrak{h}):= \K[x_{si} : s,i =1\ldots,n]/J$ of $\mathfrak{h}$, where $J$ is the ideal generated by the following universal polynomials: 
 		\begin{equation}\label{univ_poly_lbz}
 			P^{(\mathfrak{h})}_{(a,i_1,i_2)} = \sum_{u=1}^n \alpha^u_{i_1, i_2} X_{au} - \sum_{s_1, s_2 = 1}^n \alpha^a_{s_1, s_2} X_{s_1 i_1} X_{s_2 i_2},
 		\end{equation} 
 		for any $a, i_1, i_2 = 1, \ldots, n$. 
 	\end{definition}
 	
 	\begin{remark}
 		Note that the universal algebra and the universal polynomials of $\mathfrak{h}$ described above are identical with the ones introduced in \cite[Definition 2.6]{Ag-Mil}.
 	\end{remark}
 	
\smallskip	
\subsubsection{\it The operad $\mathcal{L}ie$}~~\\
	
	\vspace{-0.4cm} 
	The operad $\mathcal{L}ie :=  \mathbb{T}(E_{\mathcal{L}ie})/(R_{\mathcal{L}ie})$ corresponds to the category of Lie algebras. The $\Sm$-module $E_{\mathcal{L}ie}$ is defined as
	$E_{\mathcal{L}ie} := (0,0,\K c, 0, \ldots),$	
	where the right $\Sm_2$-action on $\K c$ is given by $c^{(12)} = - c$ (this means that $c$ is an antisymmetric operation). The space of relations $R_{\mathcal{L}ie}$ is the $\Sm_3$-submodule of $\T(E_{\mathcal{L}ie})(3)$ generated by the Jacobiator 
	$$c \circ_1 c ~+~ (c \circ_1 c)^{(123)} + (c \circ_1 c)^{(132)}.$$ 
	Note that a $\K$-linear basis of $\T(E_{\mathcal{L}ie})(3)$ is given by $\{c \circ_1 c, (c \circ_1 c)^{(123)}, (c \circ_1 c)^{(132)}\}$. We refer the reader to \cite[Section 7.6.2]{Loday-Vallette} for further details. 
	
	\medskip
	Let $\mathfrak{g}$ be a finite-dimensional Lie algebra with structure constants $\{\alpha_{i,j}^s ~|~ s,i,j=1,\ldots,n\}$. Then, the universal polynomials of $\mathfrak{g}$ are the same as described in Equation \eqref{univ_poly_lbz}. However, here the polynomials take a rather simplified form, because of the following restrictions on the structure constants: 
	\begin{equation*}
		\alpha_{i,i}^s = 0 \quad \alpha_{i,j}^s = - \alpha_{j,i}^s \quad \text{for all } s,i,j = 1,\ldots,n.
	\end{equation*} 
	Note that Corollary \ref{cor_leib} remains valid in the case of Lie algebras. Additionally, using Theorem \ref{thm_2_gen_rest_all} we obtain that the universal algebra $\mathcal{C}(\mathfrak{g})$ coincides with the ``universal algebra of $\mathfrak{g}$" as constructed in \cite[Definition 2.6]{Ag-Mil}.

\bigskip
\subsubsection{\it The operad $\mathcal{Z}inb$}~~\\

	\vspace{-0.4cm}
	The operad $\mathcal{Z}inb := \mathbb{T}(E_{\mathcal{Z}inb})/(R_{\mathcal{Z}inb})$ is the Koszul dual of the operad $\mathcal{L}eib$. It encodes the category of Zinbiel algebras. The $\Sm$-module $E_{\mathcal{Z}inb}$ is given by
	$E_{\mathcal{Z}inb} := (0,0,\K[\Sm_2],0,\ldots)$	
	and the space of relations $R_{\mathcal{Z}inb}$ is an $\Sm_3$-submodule of $\mathbb{T}(E_{\mathcal{Z}inb})(3)$ generated by 
	\begin{equation}\label{quad_rel_zinb}
		\mu \circ_1 \mu -\mu \circ_2 \mu - (\mu \circ_2 \mu)^{(23)}. 
	\end{equation}
	It is worth noting that the generators of \( R_{\mathcal{L}ieb} \) and \( R_{\mathcal{Z}inb} \), as specified in Equations \eqref{quad_rel_lbz} and \eqref{quad_rel_zinb}, respectively, are distinct from one another.

	\medskip
	Now, using Theorems \ref{thm_left_adjoint} and \ref{thm_2_gen_rest_all}, we present the following result and definition in the context of Zinbiel algebras. 
	\begin{corollary}
		Let $\mathfrak{z}$ be a Zinbiel algebra. Then the functor $\mathfrak{z} \otimes - : \mathrm{ComAlg} \to \mathrm{Zinb}$ admits a left adjoint if and only if $\mathfrak{z}$ is finite-dimensional.  
	\end{corollary}

	\begin{definition}
		Let $\mathfrak{z}$ be a finite-dimensional Zinbiel algebra with structure constants $\{\alpha_{i,j}^s~|~s,i,j=1,\ldots,n\}$. Then, the universal polynomials of $\mathfrak{z}$ are of the form
		\begin{equation}\label{univ_poly_zinb}
			P^{(\mathfrak{z})}_{(\mu,a,i_1,i_2)} = \sum_{u=1}^n \alpha^u_{i_1,i_2} X_{au}
			\hspace{0.2cm} - \sum_{s_1,s_2=1}^n \alpha^a_{s_1,s_2}X_{s_1 i_1} X_{s_2 i_2},
		\end{equation}
		for all $a,i_1,i_2 = 1,\ldots, n$. And the universal algebra $\mathcal{C}(\mathfrak{z})$ of $\mathfrak{z}$ is given by $\K[X_{si}~|~s,i=1,\ldots,n]\big/J$, where $J$ is the ideal generated by the universal polynomials of $\mathfrak{z}$ described in Equation \eqref{univ_poly_zinb}.  
	\end{definition}
	
\smallskip	
\subsubsection{\it The operad $\mathcal{P}ois$}~~\\

	\vspace{-0.4cm}
	The operad encoding Poisson algebras is denoted by $\mathcal{P}ois$. It can be constructed from two operads, $\mathcal{L}ie$ and $\mathcal{C}om$, using the ``distributive law" (cf. \cite[Section 8.6]{Loday-Vallette}). The quadratic presentation of $\mathcal{P}ois$ is given by 
	\begin{equation}\label{def_operad_pois}
		\mathcal{P}ois := \mathbb{T}(E_{\mathcal{L}ie} \oplus E_{\mathcal{C}om})/ (R_{\mathcal{L}ie} \oplus R_{\mathcal{C}om} \oplus D)
	\end{equation} 
	where $D$ is an $\Sm_3$-submodule of $\mathbb{T}(E_{\mathcal{L}ie} \oplus E_{\mathcal{C}om})(3)$ generated by $c \circ_1 \mu - (\mu \circ_1 c)^{(23)} - \mu \circ_2 c$. 
	Now, as a corollary of Theorem \ref{thm_left_adjoint} we obtain the result proved by Agore and Militaru for Poisson algebras. 
	\begin{corollary}\cite[Theorem 2.2]{Ag-Mil2}
		Let $\mathfrak{p}$ be a Poisson algebra. Then the functor $\mathfrak{p} \otimes - : \mathrm{ComAlg} \to \mathrm{Pois}$ admits a left adjoint if and only if $\mathfrak{p}$ is finite-dimensional.  
	\end{corollary}
	
	Let $\mathfrak{p}$ be a finite-dimensional $\mathcal{P}ois$-algebra with structure constants $\{\alpha_{i,j}^s ~|~ s,i,j=1,\ldots,n\}$ and $\mathcal{C}(\mathfrak{p}) := \K[X_{si}~|~ s,i=1,\ldots,n]/J$ be its universal algebra. Using Theorem \ref{thm_2_gen_rest_all}, we get that the ideal $J$ is generated by the universal polynomials of $\mathfrak{p}$ corresponding to $c$ and $\mu \in \K c \oplus \K \mu = \mathcal{P}ois(2)$:
	\begin{equation}\label{eqn_pois_1}
		P^{(\mathfrak{p})}_{(c,a,i_1,i_2)} = \sum_{u=1}^n \alpha^u_{\gamma_2(c),i_1,i_2} X_{au}
		\hspace{0.2cm} - \sum_{s_1,s_2=1}^n \alpha^a_{\gamma_2(c),s_1,s_2}X_{s_1 i_1}X_{s_2 i_2}
	\end{equation} 
	\begin{equation}\label{eqn_pois_2}
		P^{(\mathfrak{p})}_{(\mu,a,i_1,i_2)} = \sum_{u=1}^n \alpha^u_{\gamma_2(\mu),i_1,i_2} X_{au}
		\hspace{0.2cm} - \sum_{s_1,s_2=1}^n \alpha^a_{\gamma_2(\mu),s_1,s_2}X_{s_1 i_1}X_{s_2 i_2}
	\end{equation} 
	for all $a,i_1,i_2 = 1,\ldots,n$. Note that, the universal polynomials described in Equation (\ref{eqn_pois_1}) and (\ref{eqn_pois_2}) coincides with the universal polynomials given in \cite[Equation (8) and (9), pg. 39]{Ag-Mil2}. Hence, the universal algebra $\mathcal{C}(\mathfrak{p})$ defined here is the same as the ``universal algebra of $\mathfrak{p}$" constructed in \cite[pg. 39]{Ag-Mil2}.

\bigskip
\subsubsection{\it The operad $pre\mathcal{L}ie$}~~\\
	
	\vspace{-0.4cm}
	The notion of pre-Lie algebras appeared first in the study of Hochschild cohomology and in studying flat affine connections on a given manifold. The operad $pre\mathcal{L}ie := \mathbb{T}(E_{pre\mathcal{L}ie})/(R_{pre\mathcal{L}ie})$ encodes the category of pre-Lie algebras. The $\Sm$-module $E_{pre\mathcal{L}ie}$ is given by 
	$E_{pre\mathcal{L}ie} := (0,0,\K[\Sm_2],0,\ldots)$	
	and the space of relations $R_{pre\mathcal{L}ie}$ is an $\Sm_3$-submodule of $\mathbb{T}(E_{pre\mathcal{L}ie})(3)$ generated by 
	$$\mu \circ_1 \mu -\mu \circ_2 \mu - (\mu \circ_1 \mu)^{(23)} + (\mu \circ_2 \mu)^{(23)}.$$ 
	
	\begin{corollary}
		Let $\mathfrak{l}$ be a pre-Lie algebra. Then the functor $\mathfrak{l} \otimes - : \mathrm{ComAlg} \to \mathrm{pre}\text{-}\mathrm{Lie}$ admits a left adjoint if and only if $\mathfrak{p}$ is finite-dimensional.
	\end{corollary}
	
	The universal polynomials of a finite-dimensional pre-Lie algebra $\mathfrak{l}$ is of the same form as described in Equation \eqref{univ_poly_lbz}.
	And The universal algebra of $\mathfrak{l}$ is given by $\mathcal{C}(\mathfrak{l}):=\K[X_{si}: s,i=1,\ldots,n]/J$,where $J$ is the ideal generated by the universal polynomials of $\mathfrak{l}$ corresponding to $\mu \in \K[\Sm_2] = pre\mathcal{L}ie(2)$.

\bigskip
\subsubsection{\it The operad $\mathcal{P}erm$}~~\\

	\vspace{-0.4cm}
	The Koszul dual of the operad $pre\mathcal{L}ie$ is the operad $\mathcal{P}erm$. It is defined as $\mathcal{P}erm := \mathbb{T}(E_{\mathcal{P}erm})/(R_{\mathcal{P}erm})$ which encodes the category of Perm algebras. The $\Sm$-module $E_{\mathcal{P}erm}$ is defined as 
	$E_{\mathcal{P}erm} := (0,0,\K[\Sm_2],0,\ldots).$	
	The space of relations $R_{\mathcal{P}erm}$ is an $\Sm_3$-submodule of $\mathbb{T}(E_{\mathcal{P}erm})(3)$ generated by 
	$$\mu \circ_1 \mu -\mu \circ_2 \mu,~\mu \circ_1 \mu - (\mu \circ_2 \mu)^{(23)}, \text{ and } \mu \circ_2 \mu - (\mu \circ_1 \mu)^{(23)}.$$ 
	
	\begin{corollary}
		Let $\mathfrak{q}$ be a Perm algebra. Then the functor $\mathfrak{q} \otimes - : \mathrm{ComAlg} \to \mathrm{Perm}$ admits a left adjoint if and only if $\mathfrak{q}$ is finite-dimensional.
	\end{corollary}

	Let $\mathfrak{q}$ be a finite-dimensional Perm-algebra and $\mathcal{C}(\mathfrak{q})$ be its universal algebra.  Then the ideal $J$ is generated by the universal polynomials of $\mathfrak{q}$ corresponding to $\mu \in \K[\Sm_2] = \mathcal{P}erm(2)$, which is of the same form as described in Equation \eqref{univ_poly_lbz}.
	
\bigskip
\subsection{Algebras over $k$-ary quadratic operads}~~\\
	
	\vspace{-0.4cm}
	Here we focus on algebras over certain $k$-ary quadratic operads. More precisely, we describe the universal algebra and the universal polynomials for finite-dimensional totally associative $k$-ary algebras, partially associative $k$-ary algebras, Leibniz $k$-algebras, and Lie $k$-algebras. Before proceeding, we fix some notation. For any $k \in \mathbb{N}$, the basis elements of $\K[\Sm_k]$ are denoted as follows: $\mu$ denotes $\mathrm{id}_{\Sm_k}$ and $\mu^{\sigma}$ denotes the element $\sigma \in \Sm_k$. 
	
	\medskip
	Similar to binary quadratic operads, any $k$-ary quadratic operad $\mathcal{P}$ can be described in terms of quadratic data as follows: 
	$\mathcal{P}:= \T(E_\mathcal{P})\big/(R_\mathcal{P}),$
	where $$\overline{E}_\mathcal{P} := \big(0,0,\ldots,\overbrace{E_\mathcal{P}}^{k^{th}\text{-term}},0\big),$$
	is an $\Sm$-module concentrated in arity $k$.
	We noted in Remark \ref{remark_5.3} that Theorem \ref{thm_2_gen_rest_all}, initially proven for binary quadratic operads, naturally extends to $k$-ary quadratic operads. This gives us the following theorem.
	
	\begin{theorem}\label{thm_k-ary}
		Let $\mathcal{P} = \T(E_\p)/(R_\p)$ be a $k$-ary quadratic operad. Consider a finite-dimensional $\mathcal{P}$-algebra $\mathfrak{a}$ with its universal algebra $\mathcal{C}(\mathfrak{a})$. Then, the ideal $J$, as described in Equation \eqref{C(a)}, is generated by the universal polynomials corresponding to the basis elements of $\mathcal{P}(k) = E_\p$. 
	\end{theorem} 
	
	This extension enables us to describe the universal algebras for algebras over $k$-ary operads. Below, we present examples of such operads, which include several well-known $k$-ary algebras. 
	
\bigskip
\subsubsection{\it Totally associative $k$-ary algebras}~~\\

	\vspace{-0.4cm}
	A totally associative $k$-ary algebra is a vector space $A$ equipped with a linear map $\mu: A^{\otimes k} \to A$ such that for any $1 \le i,j \le k$ the following equality holds
	\begin{equation*}
		\mu \big(\mathrm{id}^{\otimes i-1} \otimes \mu \otimes \mathrm{id}^{\otimes k-i}\big) 
		= \mu \big(\mathrm{id}^{\otimes j-1} \otimes \mu \otimes \mathrm{id}^{\otimes k-j}\big).		
	\end{equation*}
	The operad $t\mathcal{A}ss := \T(E_{t\mathcal{A}ss})/ (R_{t\mathcal{A}ss})$ encodes the category of totally associative $k$-ary algebras, where the $\Sm$-module $E_{t\mathcal{A}ss}$ is given by 
	\begin{equation*}
		E_{t\mathcal{A}ss} := (0,\ldots,\overbrace{\K[\Sm_k]}^{k^{th}},0,\ldots).
	\end{equation*}
	The space of relations $R_{t\mathcal{A}ss}$ is the $\Sm_{2k-1}$-submodule of $\T(E_{t\mathcal{A}ss})(2k-1)$ generated by $\{\mu \circ_i \mu - \mu \circ_j \mu : \text{for all}~ i,j = 1,\ldots,k\}$. Now, as a corollary of Theorem \ref{thm_left_adjoint} we obtain the following result. 
	
	\begin{corollary}
		Let $A$ be a totally associative $k$-ary algebra. Then the functor $A \otimes - : \mathrm{ComAlg} \to \mathrm{tAss}$ admits a left adjoint if and only if $A$ is finite-dimensional.
	\end{corollary}
	
	Let $A$ be a finite dimensional totally associative $k$-ary algebra. Also, let the set of scalars  $\{\alpha_{i_1,\ldots,i_k}^s~|~s,i_1,\ldots,i_k=1,\ldots,n\}$ denote the structure constants of $A$. Then, the universal polynomials of $A$ are given by 
	\begin{equation}\label{univ_poly_tAss}
		P^{(A)}_{(\mu,a,i_1,\ldots,i_k)} = \sum_{u=1}^n \alpha^u_{i_1,\ldots,i_n} X_{au}
		\hspace{0.2cm} - \sum_{s_1,\ldots,s_k=1}^n \alpha^a_{s_1,\ldots,s_k}X_{s_1 i_1} \cdots X_{s_k i_k}
	\end{equation} 
	for all $a,i_1,\ldots, i_k = 1,\ldots, n$. The universal algebra of $A$ is defined as $\mathcal{C}(A) := \K[X_{si} ~|~ s,i=1,\ldots,n]\big/J$, where $J$ is the ideal generated by the universal polynomials of $A$, as described in Equation \eqref{univ_poly_tAss}. 
	
\bigskip
\subsubsection{\it Partially associative $k$-ary algebras}~~\\

	\vspace{-0.4cm}
	A partially associative $k$-ary algebra is a vector space $A$ equipped with a linear map $\mu: A^{\otimes k} \to A$ such that the following equality holds
	\begin{equation*}
		\sum_{i=1}^{k} (-1)^{(i+1)(k-1)} \mu \big(\mathrm{id}^{\otimes i-1} \otimes \mu \otimes \mathrm{id}^{\otimes k-i}\big) 
		= 0.		
	\end{equation*}
	The operad $p\mathcal{A}ss := \T(E_{p\mathcal{A}ss})/(R_{p\mathcal{A}ss})$ encodes the category of partially associative $k$-ary algebras, where $\Sm$-module $E_{p\mathcal{A}ss}$ is defined as follows 
	\begin{equation*}
		E_{p\mathcal{A}ss} := (0,\ldots,\overbrace{\K[\Sm_k]}^{k^{th}},0,\ldots).
	\end{equation*}
	The space of relations $R_{p\mathcal{A}ss}$ is an $\Sm_{2k-1}$-submodule of $\T(E_{p\mathcal{A}ss})(2k-1)$ generated by $\sum_{i=1}^k (-1)^{(i+1)(k-1)} \mu \circ_i \mu$. Now, as a corollary of Theorem \ref{thm_left_adjoint} we obtain the following result. 
	
	\begin{corollary}
		Let $A$ be a partially associative $k$-ary algebra. Then the functor $A \otimes - : \mathrm{ComAlg} \to \mathrm{pAss}$ admits a left adjoint if and only if $A$ is finite-dimensional.
	\end{corollary} 
	
	Let $A$ be a finite dimensional partially associative $k$-ary algebra. The universal polynomials of $A$ takes the same form as described in Equation \eqref{univ_poly_tAss}. 

\bigskip	
\subsubsection{\it Lie $k$-algebras}~~\\

	\vspace{-0.4cm}
	A Lie $k$-algebra is a vector space $L$ equipped with a linear map $\mu: L ^{\otimes k} \to L$ which satisfies the following equalities: 
	\begin{equation}
		\mu (\mu \otimes\mathrm{id}^{k-1})
		= \sum_{i=1}^{k} \mu \big(\mathrm{id}^{\otimes i-1} \otimes \mu \otimes \mathrm{id}^{\otimes k-i}\big)^{\sigma_i}
		\quad \text{and} \quad 
		\mu^\sigma = \mathrm{sgn}(\sigma) \mu \text{ for all } \sigma \in \Sm_k, 
	\end{equation}	
	where $\sigma_i \in \Sm_{2k-1}$ is given by 
	$\sigma_i := \left[\underbrace{1,\ldots,i-1},\overbrace{i,k+1,k+2,\ldots,k+k-1},\underbrace{i+1,\ldots,k}\right].$
	The permutation $\sigma_i$ is described using the notation given by J. L. Loday and B. Vallete in \cite[Introduction, pg xxiv]{Loday-Vallette}. The operad $k$-$\mathcal{L}ie := \T(E_{k\text{-}\mathcal{L}ie})/(R_{k\text{-}\mathcal{L}ie})$ is described in terms of the quadratic data $(E_{k\text{-}\mathcal{L}ie},R_{k\text{-}\mathcal{L}ie})$ where the $\Sm$-module $E_{k\text{-}\mathcal{L}ie}$ is given by 
	\begin{equation}
		E_{k\text{-}\mathcal{L}ie} := (0,\ldots,\overbrace{\K[\Sm_k]}^{k^{th}},0,\ldots)
	\end{equation}
 	and the space of relations $R_{k\text{-}\mathcal{L}ie}$ is an $\Sm_{2k-1}$-submodule of $\T(E_{k\text{-}\mathcal{L}ie})(2k-1)$ generated by 
 	\begin{equation*}
 		\mu (\mu \otimes\mathrm{id}^{k-1})
 		- \sum_{i=1}^{k} \mu \big(\mathrm{id}^{\otimes i-1} \otimes \mu \otimes \mathrm{id}^{\otimes k-i}\big)^{\sigma_i}
 		\quad \text{and} \quad 
 		\mu^\sigma - \mathrm{sgn}(\sigma)\mu = 0 \text{ for any } \sigma \in \Sm_k.
 	\end{equation*}	 
 	Again, as a corollary of Theorem \ref{thm_left_adjoint} we obtain the following result. 
 	
 	\begin{corollary}
 		Let $L$ be a Lie $k$-algebra. Then the functor $L \otimes - : \mathrm{ComAlg} \to k\text{-}\mathrm{Lie}$ admits a left adjoint if and only if $L$ is finite-dimensional.
 	\end{corollary} 
	
 	Let $L$ be a finite dimensional Lie $k$-algebra and $\mathcal{C}(L)$ be its universal algebra. Then, the ideal $J$ is generated by the universal polynomials of $L$ associated with $\mu \in \K[\Sm_k] = k\text{-}\mathcal{L}ie(k)$:
 	\begin{equation}\label{univ_poly_Lie_k}
 		P^{(L)}_{(\mu,a,i_1,\ldots,i_k)} = \sum_{u=1}^n \alpha^u_{i_1,\ldots,i_n} X_{au}
 		\hspace{0.2cm} - \sum_{s_1,\ldots,s_k=1}^n \alpha^a_{s_1,\ldots,s_k}X_{s_1 i_1} \cdots X_{s_k i_k},
 	\end{equation} 
 	for all $a,i_1,\ldots, i_k = 1,\ldots, n$.

\bigskip	
\subsubsection{\it Leibniz $k$-algebras}~~\\

	\vspace{-0.4cm}
	A Leibniz $k$-algebra is a vector space $L$ equipped with a linear map $\mu: L^{\otimes k} \to L$ for which the following equality holds
	\begin{equation}
		\mu (\mu \otimes\mathrm{id}^{k-1})
		= \sum_{i=1}^{k} \mu \big(\mathrm{id}^{\otimes i-1} \otimes \mu \otimes \mathrm{id}^{\otimes k-i}\big)^{\sigma_i},	
	\end{equation}
	where $\sigma_i \in \Sm_{2k-1}$ is the same as defined in the case of Lie $k$-algebras.
	The operad $k\text{-}\mathcal{L}eib := \T(E_{k\text{-}\mathcal{L}eib})/(R_{k\text{-}\mathcal{L}eib})$ encodes the category of Leibniz $k$-algebras, The $\Sm$-module $E_{k\text{-}\mathcal{L}eib}$ is defined as 
	\begin{equation}
		E_{k\text{-}\mathcal{L}eib} := (0,\ldots,\overbrace{\K[\Sm_k]}^{k^{th}},0,\ldots).
	\end{equation}
	The space of relations $R_{k\text{-}\mathcal{L}eib}$ is an $\Sm_{2k-1}$-submodule of $\T(E_{k\text{-}\mathcal{L}eib})(2k-1)$ generated by 
	\begin{equation}
		\mu (\mu \otimes\mathrm{id}^{k-1})
		- \sum_{i=1}^{k} \mu \big(\mathrm{id}^{\otimes i-1} \otimes \mu \otimes \mathrm{id}^{\otimes k-i}\big)^{\sigma_i}.
	\end{equation}	 
	Again, as a corollary of Theorem \ref{thm_left_adjoint} we obtain the following result. 
	\begin{corollary}
		Let $L$ be a Leibniz $k$-algebra. Then the functor $L \otimes - : \mathrm{ComAlg} \to k\text{-}\mathrm{Leib}$ admits a left adjoint if and only if $L$ is finite-dimensional.
	\end{corollary} 
	
	For any finite dimensional Leibniz $k$-algebra $L$. The universal polynomials of $L$ is of the same form, as described in Equation \eqref{univ_poly_Lie_k}. 
		
\bigskip
\subsection{Algebras over non-symmetric operads}~~\\

	\vspace{-0.4cm}
	Any given non-symmetric operad $\mathcal{P} = (\mathcal{P}(0),\mathcal{P}(1),\ldots)$, one can construct a symmetric operad $\widetilde{\mathcal{P}}$ given as follows:
	\begin{equation*}
		\widetilde{\mathcal{P}}(n) := \mathcal{P}(n) \otimes \K[\Sm_n],
	\end{equation*}
	where the action of the symmetric group $\Sm_n$ on $\widetilde{\mathcal{P}}(n)$ is given by the regular representation $\K[\Sm_n]$. The association $\mathcal{P} \mapsto \widetilde{\mathcal{P}}$ is refer to as ``symmetrization". Furthermore, we note that the category of algebras over the non-symmetric operad $\mathcal{P}$ and over its associated symmetric operad $\widetilde{\mathcal{P}}$ are the same. See \cite[Section 5.9.11]{Loday-Vallette} for further details. Thus, using symmetrization our universal constructions also holds for algebras over non-symmetric operads. 
	
			
	\medskip	
	We refer the readers to \cite{Apurba-Da} for examples of non-symmetric operads that encode various Loday-type algebras such as dialgebras, associative trialgebras, dendriform algebras, dendriform trialgebras, quadri algebras, and ennea algebras. Consequently, using symmetrization to obtain the associated symmetric operads for the operads encoding Loday-type algebras, we can construct the universal algebras for these Loday-type algebras using the results obtained in Section \ref{sec-3}.

\bigskip
\subsection{Algebras over graded operads}~~\\
	
	\vspace{-0.4cm}
	Here, we consider examples of graded algebras and operads encoding them. Precisely, we give the construction of universal coacting algebras for graded Lie/Leibniz algebras, graded Poisson algebras, Garstenhaber algebras, and  Batalin-Vilkovisky algebras. 

\bigskip
\subsubsection{\it Graded Leibniz algebras}~~\\
	
	\vspace{-0.4cm}
	A graded Leibniz algebra is a graded vector space $\mathfrak{H} = \oplus_{p \ge 0}\mathfrak{H_p}$ equipped with a degree $0$ binary bracket ``$[\cdot, \cdot]$",  called the graded Leibniz bracket. It satisfies the following identity
	\[
	[x, [y, z]] = [[x, y], z] + (-1)^{|x||y|} [y, [x, z]],
	\]
	where $|x|$ and $|y|$ denotes the degree of the homogeneous elements $x$ and $y$, respectively. This identity is known as the graded Leibniz identity.
	
	\medskip
	The operad $\mathcal{L}eib$ encodes the category of graded Leibniz algebras. Note that we can consider the operad $\mathcal{L}eib$ as a graded operad concentrated at degree zero in every arity. Then, a graded Leibniz algebra structure on a graded vector space $\mathfrak{H} = \oplus_{p \ge 0}\mathfrak{H_p}$ is given by a graded operad morphism $\gamma : \mathcal{L}eib \rightarrow \mathrm{End}_{\mathfrak{H}}$.  
	
	\medskip 
 	The following corollary to Theorem \ref{thm_graded_left_adjoint} naturally extends the universal construction of A. L. Agore and G. Militaru \cite[Theorem 2.1]{Ag-Mil} to the graded context. 
	
	\begin{corollary}\label{cor_graded_Leib}
		Let $\mathfrak{H} = \oplus_{p \ge 0} \mathfrak{H}_p$ be a graded Leibniz algebra over $\K$. Then the functor $\mathfrak{H} \otimes - : \mathrm{g}\text{-}\mathrm{ComAlg} \to \mathrm{g}\text{-}\mathrm{Leib}$ admits a left adjoint if and only if each component $\mathfrak{H}_p$ of $\mathfrak{H}$ is finite-dimensional. 
	\end{corollary}
	
	\begin{proof}
		($\implies$)
		Since each component $\mathfrak{H}_p$ of $\mathfrak{H}=\oplus_{p \ge 0} \mathfrak{H}_p$ is finite-dimensional, let   
		\begin{equation*}
			\{a_{01},\ldots,a_{0d_0};a_{11},\ldots,a_{1d_1}; \ldots;a_{p1},\ldots,a_{pd_p};\ldots\}
		\end{equation*}
		be a vector space basis of $\mathfrak{H}$, where $\{a_{p1},\ldots,a_{pd_p}\}$ denotes a vector space basis of the $p$-degree component $\mathfrak{H}_p$. The structure constants of $\mathfrak{H}$ are obtained as follows: for any $p,q \ge 0$
		\begin{equation}\label{graded_case_struct_const_h}
			[a_{pi},a_{qj}] = \sum_{s=1}^{d_{p+q}} \alpha_{pi,qj}^{\theta,s}~a_{(p+q)s} 
		\end{equation} 
		where $\theta = p+q$, $1 \le i \le d_p$, and $1 \le j \le d_q$. Consequently, the set of scalars given below yields the structure constants of $\mathfrak{H}$:
		$$\bigcup_{p,q \ge 0}\{\alpha_{pi,qj}^{\theta,s}~|~ \theta = p+1;~ 1 \le i \le d_p;~ 1 \le j \le d_q;~ 1 \le s \le d_{\theta}\}.$$ 
		Now, for any object $\mathfrak{G} \in \mathrm{g}\text{-}\mathrm{Leib}$ we shall construct a graded commutative algebra $\mathcal{C}(\mathfrak{H},\mathfrak{G})$. 
		Let $\{b_{pi}: p \ge 0, i \in I_p\}$ be a basis of $\mathfrak{G}$. Then, for any $p,q \ge 0$, $i \in I_p$, and $j \in I_q$ there exists a finite set $B_{pi,qj} \subset I_{p+q}$ for which
		\begin{equation}\label{graded_case_struc_const_g}
			[b_{pi},b_{qj}] ~= \sum_{u \in B_{pi,qj}} \beta_{pi,qj}^{\theta,u} ~b_{(p+q)u}, 
		\end{equation}
		where $\theta = p+q$.
		Let $\K\big[X_{si}^{(\pi)}: s,i,\pi \ge 0\big]$ be the usual graded commutative polynomial algebra with indeterminants $X^{(\pi)}_{si}$ of cohomological degree $\pi$. We define
		\begin{equation}
			\mathcal{C}(\mathfrak{H},\mathfrak{G}) := \K\big[X_{si}^{(\pi)}: s,i,\pi \ge 0\big] \big/ J
		\end{equation} 
		where $J$ is the ideal generated by the polynomials of the following form: for any $p,q \ge 0$
		\begin{equation}
			P_{(r,pi,qj)}^{(\mathfrak{H},\mathfrak{G},\omega)} 
			:= 
			\sum_{u \in B_{pi,qj}} \beta_{pi,qj}^{\theta,u} X_{ru}^{(\omega)} 
			- 
			\sum_{\substack{\varepsilon + \tau = \theta-\omega,\\ 0 \le \varepsilon \le p, \\ 
					0 \le \tau \le q.}}
			\sum_{s=1}^{d_\varepsilon} 
			\sum_{t=1}^{d_\tau} 
			(-1)^{\tau (p-\varepsilon)}
			\alpha_{\varepsilon s,\tau t}^{\theta-\omega,r} X_{si}^{(p-\varepsilon)} X_{tj}^{(q-\tau)}, 
		\end{equation} 
		where $\theta = p+q$, $0 \le \omega \le \theta$, $i \in I_p$, $j \in I_q$, and $r= 1, \ldots, d_{\theta}$.
		Let $x_{si}^{(\pi)} := \widehat{X_{si}^{(\pi)}}$ denote the class of $X_{si}^{(\pi)}$ in the algebra $\mathcal{C}(\mathfrak{H},\mathfrak{G})$. 
		One can show that the construction of $\mathcal{C}(\mathfrak{H},\mathfrak{G})$ defines a functor $\mathcal{C}(\mathfrak{G},-): \mathrm{g}\text{-}\mathrm{Leib} \to \mathrm{g}\text{-}\mathrm{ComAlg}$.
		
		\medskip
		Next, we define a map $\eta_\mathfrak{G}: \mathfrak{G} \to \mathfrak{H} \otimes \mathcal{C}(\mathfrak{H},\mathfrak{G})$ as follows: for any $p \ge 0$ and $i \in I_p$ 
		\begin{equation}\label{graded_case_eta_g}
			\eta_\mathfrak{G}(b_{pi}) := \sum_{\varepsilon = 0}^{p}\sum_{s=1}^{d_\varepsilon} a_{\varepsilon s} \otimes x_{si}^{(p-\varepsilon)} 
		\end{equation}
		Following the proof of Lemma \ref{lemma_graded_eta_A} one can verify that $\eta_\mathfrak{G}$ is a graded Leibniz algebra morphism. 
		Additionally, we have the following bijection: 
		\begin{equation}
			\Psi_{\mathfrak{G},A}: \mathrm{Hom}_{\mathrm{g}\text{-}\mathrm{ComAlg}}\big(\mathcal{C}(\mathfrak{H},\mathfrak{G}),A
			\big) 
			\longrightarrow \mathrm{Hom}_{\mathrm{g}\text{-}\mathrm{Leib}}(\mathfrak{G},\mathfrak{H} \otimes A), 
		\end{equation}
		defined by $\Psi_{\mathfrak{G},A}(\theta) := (\mathrm{id}_\mathfrak{H} \otimes \theta) 
		\circ \eta_\mathfrak{G}$. 
		Consequently, the functor $\mathcal{C}(\mathfrak{G},-): \mathrm{g}\text{-}\mathrm{Leib} \to \mathrm{g}\text{-}\mathrm{ComAlg}$ is left adjoint to the functor $\mathfrak{H} \otimes -:\mathrm{g}\text{-}\mathrm{ComAlg} \to \mathrm{g}\text{-}\mathrm{Leib}$. 
		
		\medskip\noindent
		$(\impliedby)$ The converse part of the proof can be done in a similarly way, by following the arguments provided in Theorem \ref{thm_graded_left_adjoint}.
	\end{proof}

	\begin{definition}
		Let $\mathfrak{H}= \oplus_{p \ge 0} \mathfrak{H}_p$ be a graded Leibniz algebra with each component $\mathfrak{H}_p$ being finite-dimensional. Also, let 
		$\{a_{01},\ldots,a_{0d_0};a_{11},\ldots,a_{1d_1};\ldots;a_{p1},\ldots,a_{pd_p};\ldots\}$ denote a vector space basis of $\mathfrak{H}$, such that the set $\{a_{p1},\ldots,a_{pd_p}\}$ denotes a basis of the $p$-degree component $\mathfrak{H}_p$. The set of scalars, given by   
		\begin{equation}\label{graded_leib_struc_const}
			\bigcup_{p,q \ge 0}\{\alpha_{pi,qj}^{\theta,s}~|~\theta = p+q;~ 1 \le i \le d_p;~ 1 \le j \le d_q;~ 1 \le s \le d_{\theta}\},
		\end{equation}
		denotes the structure constants of $\mathfrak{H}$. 
		Then, the \textbf{universal graded algebra} of $\mathfrak{H}$ is the graded commutative algebra   $\mathcal{C}(\mathfrak{H},\mathfrak{H}):= \K[X^{(\pi)}_{ij} | i,j,\pi \ge 0] \big/ J$. 
		Note that following Remark \ref{remark_5.3} we obtain the generators of the ideal $J$ which are described below. For any $p,q \ge 0$  
		\begin{equation}\label{univ_poly_graded_Leib}
			P_{(r,pi,qj)}^{(\mathfrak{H},\omega)} 
			:= 
			\sum_{u = 1}^{d_{\theta}} \alpha_{pi,qj}^{\theta,u} X_{ru}^{\omega} 
			~- 
			\sum_{\substack{\varepsilon + \tau = \theta-\omega ,\\ 0 \le \varepsilon \le p, \\ 
					0 \le \tau \le q.}}
			\sum_{s=1}^{d_\varepsilon} 
			\sum_{t=1}^{d_\tau} 
			(-1)^{\tau(p-\varepsilon)}\alpha_{\varepsilon s,\tau 
			t}^{\theta-\omega,r} X_{si}^{(p-\varepsilon)} X_{tj}^{(q-\tau)},
		\end{equation}
		where $\theta=p+q,$ $0 \le \omega \le \theta$, $i \in I_p$, $j \in I_q$, and $r=1,\ldots,d_{\theta}$. These polynomials are called the \textbf{universal graded polynomials of $\mathfrak{H}$} of cohomological degree $\omega$. 
	\end{definition}	

\medskip
\subsubsection{\it Graded Lie algebras}~~\\

	\vspace{-0.4cm}
	A graded Lie algebra $\mathfrak{H}=\oplus_{p \ge 0} \mathfrak{H}_p$ is a graded Leibniz algebra, which along with the graded Leibniz identity, satisfying the graded anti-symmetry relation:
	\begin{equation*}\label{graded_Lie_anti_sym}
		[x,y] = - (-1)^{|x||y|} [y,x] \quad \text{for all homogeneous elements } x,y \in \mathfrak{H}.
	\end{equation*} 
	Since, any graded Lie algebra is also a graded Leibniz algebra, Corollary \ref{cor_graded_Leib} also holds for graded Lie algebras and the universal graded polynomials of a graded Lie algebra are the same as the polynomials described in Equation \eqref{univ_poly_graded_Leib}. However, here the universal graded polynomials takes a rather simplified form, 
	because the structure constants as described in \eqref{graded_leib_struc_const} satisfying the following relations:   
	\begin{equation*}
		\alpha_{pi,pi}^{\theta,s} = 0 
		\quad \text{and} \quad 
		\alpha_{pi,qj}^{\theta,s} 
		= 
		-(-1)^{pq} \alpha_{qj,pi}^{\theta,s}. 
	\end{equation*} 
	 
\bigskip
\subsubsection{\it Graded Poisson algebras}~~\\
	
	\vspace{-0.4cm}
	A graded Poisson algebra is a graded vector space $\mathfrak{P}$ equipped with a graded Lie bracket $``\{\cdot,\cdot\}"$ satisfying the following graded Leibniz rule: 
	\begin{equation*}
		\{x,y \cdot z\} = \{x,y\} \cdot z  + (-1)^{|x||y|}y \cdot \{x,z\},  
	\end{equation*}
	for any homogeneous elements $x,y,z \in \mathfrak{P}$, where $|x|,|y|$ denote their degrees. We note that the category of graded Poisson algebra structures is encoded by the operad $\mathcal{P}ois$ (cf. Equation \eqref{def_operad_pois}) by considering $\mathcal{P}ois$ as a graded operad concentrated at degree $0$ in each arity. Now, as a corollary of Theorem \ref{thm_graded_left_adjoint} we obtain the following result. 
	
	\begin{corollary}\label{cor_graded_Pois}
		Let $\mathfrak{P} = \oplus_{p \ge 0} \mathfrak{P}_p$ be a graded Poisson algebra over $\K$. Then the functor $\mathfrak{H} \otimes - : \mathrm{g}\text{-}\mathrm{ComAlg} \to \mathrm{g}\text{-}\mathrm{Pois}$ admits a left adjoint if and only if each component $\mathfrak{P}_p$ of $\mathfrak{P}$ is finite-dimensional. 
	\end{corollary}
	
	\begin{remark}
		We note that the above corollary extends the universal construction of A. L. Agore and G. Militaru \cite[Theorem 2.2]{Ag-Mil2} to the graded context.
	\end{remark}

	Let $\mathfrak{P} = \oplus_{p \ge 0} \mathfrak{P}_p$ be a graded Poisson algebra with each component $\mathfrak{P}_p$ being finite dimensional, where the graded Poisson algebra structure on $\mathfrak{P}$ is given by the graded operad morphism $\gamma: \mathcal{P}ois\rightarrow \mathrm{End}_{\mathfrak{P}}$.
	Let the set  $\big\{a_{01},\ldots,a_{0d_0};a_{11},\ldots,a_{1d_1};\ldots;a_{n1},\ldots,a_{nd_n};\ldots\big\}$ denote a vector space basis of $\mathfrak{P}$ such that $\{a_{p1},\ldots,a_{pd_p}\}$ becomes a vector space basis of the $p$-degree component $\mathfrak{P}_p$. We assume that the set of scalars 
	\begin{equation*}
		\bigcup_{p_1,p_2 \ge 0}\Big\{\alpha^{\theta,s}_{\gamma_2(c),p_1i_1,p_2i_2},~\alpha^{\theta,s}_{\gamma_2(m),p_1i_1,p_2i_2}~|~ \theta = p_1 + p_2;~ i_r = 1,\ldots,d_{p_r} \text{ for any } r = 1,2;~ 1 \le s \le d_{p_1+p_2}\Big\}
	\end{equation*}
	to be the structure constants of $\mathfrak{P}$.	
	From the discussion in Subsection \ref{sub3.2}, we deduce that the universal graded polynomials \eqref{graded universal poly} of cohomological degree $\omega$ for $\mathfrak{P}$ are given as follows.
	\begin{equation}\label{graded_uni1}
		P^{(\mathfrak{P},\omega )}_{(c,r,p_1i_1,p_2i_2)} 
		:=
		\sum_{u=1}^{d_\theta} \alpha^{\theta,u}_{\gamma_2(c),p_1i_1,p_2i_2} X^{(\omega)}_{ru}
		\hspace{0.2cm} -
		\sum_{\substack{\epsilon_1 + \epsilon_2 = \theta-\omega \\ 0 \le \epsilon_k \le p_k}}	
		\Bigg(
		\sum_{s_1=1}^{d_{p_1}}\sum_{s_2=1}^{d_{p_2}} \alpha^{\theta-\omega,r}_{\gamma_2(c),\epsilon_1 s_1,\epsilon_2 s_2}X_{s_1 i_1}^{(p_1-\epsilon_1)} X_{s_2 i_2}^{(p_2-\epsilon_2)}\Bigg),
	\end{equation}
	\begin{equation}\label{graded_uni2}
		P^{(\mathfrak{P},\omega )}_{(m,r,p_1i_1,p_2i_2)} 
		:=
		\sum_{u=1}^{d_\theta} \alpha^{\theta,u}_{\gamma_2(m),p_1i_1,p_2i_2} X^{(\omega)}_{ru}
		\hspace{0.2cm} -
		\sum_{\substack{\epsilon_1 + \epsilon_2 = \theta-\omega \\ 0 \le \epsilon_k \le p_k}}	
		\Bigg(
		\sum_{s_1=1}^{d_{p_1}}\sum_{s_2=1}^{d_{p_2}} \alpha^{\theta-\omega,r}_{\gamma_2(m),\epsilon_1 s_1,\epsilon_2 s_2}X_{s_1 i_1}^{(p_1-\epsilon_1)} X_{s_2 i_2}^{(p_2-\epsilon_2)}\Bigg),
	\end{equation}
	where $\theta = p_1 + p_2$, $0 \le \omega \le \theta$, and $r = 1, \ldots, d_\omega$. Therefore, the universal graded algebra $\mathcal{C}(\mathfrak{P})$ is given by 
	
	\begin{equation*}
		\mathcal{C}(\mathfrak{P}) := \mathbb{K}\big[X^{(\pi)}_{ij}|~i,j,\pi\geq 0 \big] \big/ J,
	\end{equation*}
	where, due to Remark \ref{remark_5.3}, $J$ is the ideal generated by universal graded  polynomials defined in Equation \eqref{graded_uni1} and \eqref{graded_uni2}.

\bigskip
\subsubsection{\it Gerstenhaber algebras}~~\\
	
	\vspace{-0.4cm}
	A Gerstenhaber algebra is an algebra over the graded operad $\mathcal{G}erst$ with the quadratic presentation $(E_{\mathcal{G}erst},R_{\mathcal{G}erst})$. Here, the space of generators is given by 
	\begin{equation*}
		E_{\mathcal{G}erst} :=  \big(0,0,\K m \oplus \K c,0,\ldots\big)
	\end{equation*} 
	which is a direct sum of two one-dimensional trivial representations of $\Sm_2$ concentrated in arity $2$, one in degree $0$ denoted by $m$ and one in degree $1$ denoted by $c$. The space of relations $R_{\mathcal{G}erst}$ is the graded $\K[\Sm_3]$-module generated by the following relations  
	\begin{equation*}
		c \circ_1 c + (c \circ_1 c)^{(123)} + (c \circ_1 c)^{(321)}, \quad c \circ_1 m - m \circ_2 c - (m \circ_1 c)^{(23)}, \quad m \circ_1 m - m \circ_2 m.
	\end{equation*}
	Now, as a corollary of Theorem \ref{thm_graded_left_adjoint} we obtain the following result. 
	\begin{corollary}\label{cor_gerstenhaber}
		Let $\mathfrak{G} = \oplus_{p \ge 0} \mathfrak{G}_p$ be a Gerstenhaber algebra. Then the functor $\mathfrak{G} \otimes - : \mathrm{g}\text{-}\mathrm{ComAlg} \to \mathrm{Gerst}$ admits a left adjoint if and only if each component $\mathfrak{G}_p$ of $\mathfrak{G}$ is finite-dimensional. 
	\end{corollary}
	
	Let $\mathfrak{G} = \oplus_{p \ge 0} \mathfrak{G}_p$ be a Gerstenhaber algebra with each component $\mathfrak{G}_p$ being finite-dimensional. Also, let the set $\big\{g_{01},\ldots,g_{0d_0};g_{11},\ldots,g_{1d_1};\ldots;g_{p1},\ldots,g_{pd_p};\ldots\big\}$ denote a  vector space basis of $\mathfrak{G}$, where $\{g_{p1},\ldots,g_{pd_p}\}$ denotes the basis of $\mathfrak{G}_p$. Then, the set of scalars 
	\begin{equation*}
		\bigcup_{p_1,p_2 \ge 0}\Big\{\alpha^{\theta,s}_{\gamma_2(c),p_1i_1,p_2i_2},~\alpha^{\theta,s}_{\gamma_2(m),p_1i_1,p_2i_2}~|~ \theta = p_1 + p_2;~ i_r = 1,\ldots,d_{p_r} \text{ for any } r = 1,2;~ 1 \le s \le d_{p_1+p_2}\Big\}
	\end{equation*}    
	gives us the structure constants of $\mathfrak{G}$. Hence, the universal graded polynomials of $\mathfrak{G}$ are described as follows: 
	\begin{equation}\label{gerst_graded_uni1}
		P^{(\mathfrak{G},\omega )}_{(m,r,p_1i_1,p_2i_2)} 
		:=
		\sum_{u=1}^{d_\theta} \alpha^{\theta,u}_{\gamma_2(m),p_1i_1,p_2i_2} X^{(\omega)}_{ru}
		\hspace{0.2cm} -
		\sum_{\substack{\epsilon_1 + \epsilon_2 = \theta-\omega \\ 0 \le \epsilon_k \le p_k}}	
		\Bigg(
		\sum_{s_1=1}^{d_{p_1}}\sum_{s_2=1}^{d_{p_2}} \alpha^{\theta-\omega,r}_{\gamma_2(m),\epsilon_1 s_1,\epsilon_2 s_2}X_{s_1 i_1}^{(p_1-\epsilon_1)} X_{s_2 i_2}^{(p_2-\epsilon_2)}\Bigg),
	\end{equation}
	\begin{equation}\label{gerst_graded_uni2}
		P^{(\mathfrak{G},\omega )}_{(c,r,p_1i_1,p_2i_2)} 
		:=
		\sum_{u=1}^{d_\theta} \alpha^{\theta,u}_{\gamma_2(c),p_1i_1,p_2i_2} X^{(\omega)}_{ru}
		\hspace{0.2cm} -
		\sum_{\substack{\epsilon_1 + \epsilon_2 = \theta-\omega-1 \\ 0 \le \epsilon_k \le p_k}}	
		\Bigg(
		\sum_{s_1=1}^{d_{p_1}}\sum_{s_2=1}^{d_{p_2}} \alpha^{\theta-\omega,r}_{\gamma_2(c),\epsilon_1 s_1,\epsilon_2 s_2}X_{s_1 i_1}^{(p_1-\epsilon_1)} X_{s_2 i_2}^{(p_2-\epsilon_2)}\Bigg),
	\end{equation}
	where $\theta = p_1 + p_2$, $0 \le \omega \le \theta$, and $r = 1, \ldots, d_\omega$. Consequently, the universal graded algebra $\mathcal{C}(\mathfrak{G})$ is given by 
	\begin{equation*}
		\mathcal{C}(\mathfrak{G}) := \mathbb{K}\big[X^{(\pi)}_{ij}|~i,j,\pi\geq 0 \big] \big/ J,
	\end{equation*}
	where, due to Remark \ref{remark_5.3}, $J$ is the ideal generated by universal graded  polynomials defined in Equation \eqref{gerst_graded_uni1} and \eqref{gerst_graded_uni2}.
	
\bigskip
\subsubsection{\it Batalin-Vilkovisky algebras}~~\\
		
	\vspace{-0.4cm}
	The notion of Batalin-Vilkovisky (BV)  algebras appear in diverse places of mathematics, such as they appear in differential geometry of polyvector fields, on the Hochschild cohomology of unital cyclic associative algebras, on the homology of free loop spaces, on vertex operator algebras, etc. A BV algebra structure is given by a data $(\mathfrak{G},\cdot,\Delta)$ where $\mathfrak{G}$ is a graded vector space equipped with a graded commutative product $``\cdot"$ of degree $0$, and a unary operation $\Delta$ of degree $1$ satisfying $\Delta \circ \Delta = 0$. Furthermore, these two operation satisfy the following compatibility relation:  	
	\begin{multline*}
		\Delta(x \cdot y \cdot z) 
		- \Delta(x \cdot y) \cdot z 
		- (-1)^{|y|(|x| + |z|)} \Delta(z \cdot x) \cdot y 
		- (-1)^{|x|(|y|+|z|)} \Delta(y \cdot z) \cdot x \\
		+ \Delta(a) \cdot b \cdot c 
		+ (-1)^{|y|(|x| + |z|)} \Delta(z) \cdot x \cdot y 
		+ (-1)^{|x|(|y|+|z|)} \Delta(y) \cdot z \cdot x
		= 0,
 	\end{multline*}
 	for any homogeneous elements $x,y,z \in \mathfrak{G}$. Note that, the data of a BV algebra structure $(\mathfrak{G},\cdot,\Delta)$ is equivalent to that of a Garstenhaber algebra structure $(\mathfrak{G},\cdot,[-,-])$ endowed with a square zero unary operation $\Delta$ of degree $1$ such that the following equality holds:
 	\begin{equation*}
 		[-,-] = \Delta(- \cdot -) - \big(\Delta(-) \cdot -\big) - \big(- \cdot \Delta(-)\big).
 	\end{equation*}
 
 	\medskip
 	The operad corresponding the category of all BV algebras is denoted by $\mathcal{BV} := \T(E_{\mathcal{BV}})/ (R_{\mathcal{BV}})$. The graded $\Sm$-module is given as follows, $E_\mathcal{BV} := (0,\K\Delta,\K m, \ldots)$, where $\K \Delta$ is a one-dimensional graded vector space ($\Sm_1$-module) concentrated in degree $1$ and $\K m$ is a one-dimensional graded vector space ($\Sm_2$-module) concentrated in degree $0$. The space of relations $R_{\mathcal{BV}}$ is the $\K[\Sm_3]$-module generated by 
 	\begin{equation*}
 		\begin{cases}
 			m \circ_1 m - m \circ_2 m, ~ \Delta \circ_1 \Delta, \\
 			\Delta \circ_1 m \circ_1 m 
 			- \big((\Delta \circ_1 m) \circ_1 m\big)^{\mathrm{id} + (123) + (132)} 
 			+ \big(\Delta \circ_1 (m \circ_1 m)\big)^{\mathrm{id} + (123) + (132)}. 
 		\end{cases}
 	\end{equation*}   
 	Finally, as a corollary of Theorem \ref{thm_graded_left_adjoint} we obtain the following result. 
 	\begin{corollary}
 		Let $\mathfrak{G} = \oplus_{p \ge 0} \mathfrak{G}_p$ be a BV algebra over $\K$. Then the functor $\mathfrak{G} \otimes - : \mathrm{g}\text{-}\mathrm{ComAlg} \to \mathcal{BV}$ admits a left adjoint if and only if each component $\mathfrak{G}_p$ of $\mathfrak{G}$ is finite-dimensional.
 	\end{corollary}
 	
 	\begin{remark}
 			The existence and description of universal (co)acting bialgebras and Hopf algebras, within the category of unital associative algebras, have been studied in the setting of $\Omega$-algebras in \cite{Agore-omega}. While $\mathcal{P}$-algebras considered in this work can be seen as a special case of $\Omega$-algebras (by forgetting the symmetries of the generating operations and the S-module structure spanned by the generating operations of a $\mathcal{P}$-algebra), the existence of the universal algebra $\mathcal{C}(\mathfrak{a}, \mathfrak{b})$, associated to two $\mathcal{P}$-algebras $\mathfrak{a}$ and $\mathfrak{b}$ as introduced here, does not follow from the results of \cite{Agore-omega}.
 			This is primarily because the algebra $\mathcal{C}(\mathfrak{a}, \mathfrak{b})$ is required to be universal within the category of commutative algebras. Our focus on this category is motivated by its significance in understanding the automorphism groups of $\mathcal{P}$-algebras and in the classification of abelian group gradings on such algebras, as discussed in Section~\ref{sec-6}.
 			The structured framework of $\mathcal{P}$-algebras is also useful to accommodate a wide range of important examples, including graded Lie and Leibniz algebras, graded Poisson algebras, Gerstenhaber algebras, and Batalin–Vilkovisky (BV) algebras by using graded operads instead of operads.
 	\end{remark}
	

\smallskip
\section{\large Applications} \label{sec-6}
	In this section, we give two applications of the construction of the universal coacting bialgebra of a finite-dimensional $\mathcal{P}$-algebra. First, we describe group of $\mathcal{P}$-algebra automorphisms of $\mathfrak{a}$ in the Theorem \ref{App-bijection}. Second, we characterize the abelian group gradings on a finite-dimensional $\mathcal{P}$-algebra $\mathfrak{a}$ in Theorem \ref{App-gradings}.

\bigskip
\subsection{Automorphism group of a finite-dimensional $\mathcal{P}$-algebra}~~\\
	Let $\fa$ be a finite-dimensional $\p$-algebra and $\mathcal{C}(\fa)$ be its universal coacting bialgebra. We denote the set of automorphisms of the $\p$-algebra $\fa$ by $\mathrm{Aut}_{\p\text{-}\mathrm{Alg}}(\fa)$. For the bialgebra $\mathcal{C}(\fa)$, denote the set of group-like elements 
	$$G(\mathcal{C}(\fa))=\{P\in \mathcal{C}(\fa) ~|~\Delta_\fa(P)=P\otimes P, ~\varepsilon (P)=1\}.$$
	The set $G(\mathcal{C}(\fa))$ is a monoid with respect to the multiplication on $\mathcal{C}(\fa)$. Let $\mathcal{C}(\fa)^\times$ be the finite dual bialgebra of the bialgebra $\mathcal{C}(\fa)$, defined by
	$$\mathcal{C}(\fa)^\times:=\{\chi\in \mathcal{C}(\fa)^*~|~\chi(I)=0, ~\mbox{for some ideal } I \text{ in }\mathcal{C}(\fa) \text{ with } \text{dim}_\mathbb{K}(\mathcal{C}(\fa)/I)<\infty\}.$$
	Then, let us denote the group of  all invertible group-like elements by $U(G(\mathcal{C}(\fa)^\times))$. With these notations, we get the following result generalizing Agore's and Militaru's result \cite[Theorem 3.1]{Ag-Mil} for finite-dimensional $\p$-algebras. 

	\begin{theorem}\label{App-bijection}
		Let $\fa$ be a finite-dimensional $\p$-algebra and $\mathcal{C}(\fa)$ be its universal coacting bialgebra. Then, the group of $\p$-algebra automorphisms of $\mathfrak{a}$ is isomorphic to the group of all invertible group-like elements $U(G(\mathcal{C}(\fa)^\times))$.
	\end{theorem}
 
	\begin{proof}
		It follows that the set $G(\mathcal{C}(\fa)^\times)$ is the same as the set of algebra homomorphisms from $\mathcal{C}(\fa)$ to $\mathbb{K}$. 
 		The bijection given in Equation \eqref{bijection1} implies that we have a bijection 
		$$\Theta:\mathrm{Hom}_{\mathrm{ComAlg}}(\mathcal{C}(\fa),\mathbb{K})\rightarrow \mathrm{Hom}_{\p\text{-Alg}}(\mathfrak{a},\fa), \quad \mbox{defined by } \Theta(\Phi)=(\mathrm{id}\otimes \Phi)\circ \eta_\mathfrak{a}.$$
		The vector space $\mathrm{Hom}_{\mathrm{ComAlg}}(\mathcal{C}(\fa),\mathbb{K})$ is a monoid with respect to the composition of maps. A monoid structure on $\mathrm{Hom}_{\mathrm{ComAlg}}(\mathcal{C}(\fa),\mathbb{K})=G(\mathcal{C}(\fa)^\times)$ is given by the convolution product, i.e., 
		$$\Phi_1\star \Phi_2(x_{ij}):=\sum_{s=1}^n\Phi_1(x_{is})\Phi_2(x_{sj}).$$
		It immediately follows that $\Theta$ is an isomorphism of monoids and thus it induces an isomorphism between the invertible elements in the monoids. If $\{a_1,a_2,\ldots,a_n\}$ is a basis of the vector space $\mathfrak{a}$, then the isomorphism between the invertible elements in the monoids is given by the map $$\zeta:U(G(\mathcal{C}(\fa)^\times))\rightarrow \mathrm{Aut}_{\p\text{-Alg}}(\fa),\quad\zeta(\Phi)(a_i):=\sum_{j=1}^n \Phi(x_{ji})a_j.$$
		Thus, the proof of the theorem is complete.
	\end{proof}

\smallskip
\subsection{Classification of abelian group gradings on  a finite-dimensional $\mathcal{P}$-algebra}

	\begin{definition}
		Let $G$ be an abelian group and $\fa$ be a $\p$-algebra given by the operad map $\gamma:\p\rightarrow \mathrm{End}_\fa$. Then, a $G$-grading on the $\p$-algebra is a collection of vector spaces $\{\fa_\sigma~|~\sigma\in G \}$ such that 
		$$\fa=\oplus_{\sigma\in G} \fa_\sigma,\quad\text{and}\quad \gamma(\mu)(x_{\sigma_1},x_{\sigma_2},\ldots,x_{\sigma_k})\in \fa_{\sigma_1\sigma_2\cdots\sigma_k},$$
		for homogeneous elements $x_{\sigma_i}\in \fa_{\sigma_i}$, $i=1,2,\ldots, k$, and $\mu\in \p(k)$.
	\end{definition}

	Let $\mathbb{K}[G]$ be the group algebra of $G$. Then, by Remark \ref{universal bialgebra const} and Diagram \eqref{universal bialgebra prop} we obtain a bijection
	\begin{equation}\label{App-bijection1}
		\mathrm{Hom}_{\mathrm{ComBiAlg}}(\mathcal{C}(\fa),\mathbb{K}[G])\longleftrightarrow 
		\big\{f\in \mathrm{Hom}_{\p\text{-Alg}}(\fa,\fa\otimes \mathbb{K}[G])~|~\fa \text{ becomes} 
		\text{ a right } \mathbb{K}[G]\text{-comodule}\big\}
	\end{equation}

	We now recall that there is a bijection between between the set of $G$-gradings on a vector space $\fa$ and the set of all right $\mathbb{K}[G]$-module structures $f:\fa\rightarrow \fa\otimes \mathbb{K}[G]$. Let $\fa=\oplus_{\sigma\in G}\fa_\sigma$ be a $G$-grading on the vector space $\fa$. The bijection is given as follows: $x\in \fa_{\sigma}$ then $f(x_\sigma)=x_\sigma\otimes \sigma$. Now, we show that the map $f:\fa\rightarrow \fa\otimes \mathbb{K}[G]$ is a $\p$-algebra morphism if and only if it corresponds to a $G$-grading on the $\p$-algebra. The map $f:\fa\rightarrow \fa\otimes \mathbb{K}[G]$ is a $\p$-algebra morphism if the following diagram is commutative
	\begin{equation}\label{diag-app2}
		\begin{CD}
			\p(k)\otimes (\mathfrak{a})^{\otimes k}@>{\gamma}>> \mathfrak{a}\\
  			@V \mathrm{id}\otimes f^{\otimes k}VV @VV f V\\
			\p(k)\otimes (\mathfrak{a}\otimes \mathbb{K}[G])^{\otimes k}@>\overline{\gamma}>> \fa\otimes \mathbb{K}[G].
		\end{CD} 
	\end{equation}
	Diagram \eqref{diag-app2} commutes if and only if for any $\sigma_1,\sigma_2,\ldots,\sigma_k\in G$ and $x_{\sigma_1},x_{\sigma_2},\ldots,x_{\sigma_k}\in \fa$, we have 
	$$\gamma(\mu)(x_{\sigma_1},x_{\sigma_2},\ldots,x_{\sigma_k})\in \fa_{\sigma_1\sigma_2\cdots\sigma_k}.$$
	Consequently, there is a bijection between the following sets
	\begin{equation}\label{App-bijection2}
		\begin{tikzcd}
			\begin{array}{c} \big\{f \in \mathrm{Hom}_{\p\text{-Alg}}(\fa,\fa\otimes \mathbb{K}[G])~|~\text{the vector} \\ \qquad \text{space } \mathfrak{a} \text{ becomes a right } \K[G]\text{-module}\big\} \end{array} && \begin{array}{c} \text{The set of all } G\text{-gradings on the } \\\p\text{-algebra } \mathfrak{a}.  \end{array}
			\arrow[<->, from=1-1, to=1-3]
		\end{tikzcd}
	\end{equation}	

	\begin{prop}\label{App-bijection3}
		There exists a bijection between the set of all $G$-gradings on the finite-dimensional $\p$-algebra $\mathfrak{a}$ and the set $\mathrm{Hom}_{\mathrm{ComBiAlg}}(\mathcal{C}(\fa),\mathbb{K}[G])$. 
	\end{prop}

	\begin{proof}
		The proof follows from the bijections given by  \eqref{App-bijection1} and \eqref{App-bijection2}.
	\end{proof}

	\begin{definition}\label{equivalence of gradings}
		Let $\fa=\oplus_{\sigma\in G}\fa_\sigma$ and $\fa=\oplus_{\sigma\in G}\fa_\sigma^\prime$ be two $G$-gradings on the finite-dimensional $\p$-algebra $\fa$. Then the two gradings are said to be isomorphic if we have a $\p$-algebra automorphism $\phi: \fa\rightarrow \fa$ such that $\phi(\fa_\sigma)=\fa_\sigma^\prime$.
	\end{definition}

	Next, we define a equivalence relation on the set $\mathrm{Hom}_{\mathrm{ComBiAlg}}(\mathcal{C}(\fa),\mathbb{K}[G])$. 

	\begin{definition}\label{conjugates}
		The morphisms $\Phi_1,\Phi_2\in \mathrm{Hom}_{\mathrm{ComBiAlg}}(\mathcal{C}(\fa),\mathbb{K}[G])$ are called conjugates if there exists an element $g\in U(G(\mathcal{C}(\fa)^\times))$ such that $\Phi_2=g\star\Phi_1\star g^{-1}$. 
		Note that 
		here $``\star"$ is the convolution product in $\mathrm{Hom}_{\mathrm{ComBiAlg}}(\mathcal{C}(\fa),\mathbb{K}[G])$, given by $\Phi_1\star \Phi_2(x)=\sum \Phi_1(x_{(1)})\Phi_2(x_{(2)})$. 
	\end{definition}
	The conjugacy defined in the Definition \ref{conjugates} gives an equivalence relation on the set of bialgebra morphisms from $\mathcal{C}(\fa)$ to $\mathbb{K}[G]$. Let us denote the set of equivalence classes by $\mathrm{Hom}_{\mathrm{ComBiAlg}}(\mathcal{C}(\fa),\mathbb{K}[G])/\sim$. Then, using a similar argument as in the proof of \cite[Theorem 3.5]{Ag-Mil}, we can conclude the following result:

	\begin{theorem}\label{App-gradings}
		Let $G$ be an abelian group and $\fa$ be a finite-dimensional $\p$-algebra. Then, the isomorphism classes of $G$-gradings on $\fa$ bijectively corresponds to the equivalence classes in the set $\mathrm{Hom}_{\mathrm{ComBiAlg}}(\mathcal{C}(\fa),\mathbb{K}[G])/\sim$.
	\end{theorem}


\bigskip
\section*{\large Conclusion}

	In this paper, we gave an explicit construction of a universal coacting algebra for any finite-dimensional $\mathcal{P}$-algebra. Using the operadic approach, we also extended the constructions to the graded context. Our approach is also prudent to address the following problems.

	\begin{itemize}

		\item[I.] \textbf{A universal module construction and functors between module categories.}\\
		In \cite{Rep-Th1}, A. L. Agore defined functors between the category of Lie algebra modules and the category of associated universal algebra modules. The results in \cite{Rep-Th1} leads to a representation-theoretic counterpart of Manin-Tambara’s universal coacting objects \cite{Manin, Tambara}. Later on, for other quadratic algebras such as Poisson algebras and Lie-Yamaguti algebras, the representation theoretic universal constructions have been addressed in \cite{Ag-Mil2, Lie-Y}. We already discussed here how the operadic approach unifies the universal constructions in \cite{Ag-Mil, Ag-Mil2, Lie-Y, Mil}. In a separate work, by considering modules over $\p$-algebras for a symmetric (graded) operad $\p$, we obtain a universal module construction to define functors between module categories which not only unifies the results in \cite{Rep-Th1, Ag-Mil2, Lie-Y} but also yield the explicit constructions and related results for different (graded) quadratic algebras.

		\item[II.] \textbf{Construction of a universal algebra for finite-dimensional coalgebras over (graded) Operads}\\
		For any vector space $\fa$, the coendomorphism operad $\mathrm{coEnd}_\fa$ is given by
		\begin{equation*}
			\mathrm{coEnd}_\fa(n) := \mathrm{Hom}(\fa,\fa^{\otimes n}) \quad \text{for all } n \ge 0,
		\end{equation*}
 		with the partial composition maps dual to the maps in $\mathrm{End}_{\fa}$ (Example \ref{ExEnd}). A $\p$-coalgebra structure on a vector space $V$ is an operad morphism from $\p$ to $\mathrm{coEnd}(V)$. The construction in Section \ref{sec-3} can be dualised to obtain a universal coacting algebra for finite-dimensional coalgebras over (graded) Operads.
	\end{itemize}

\bigskip

\

\end{document}